\documentclass[a4paper,11pt]{article}
\usepackage{amssymb,amsmath,amsthm,amsfonts}
\usepackage{bookmark}
\usepackage{mathrsfs}
\usepackage{multirow}
\usepackage{graphicx}
\usepackage{authblk}
\usepackage{indentfirst}
\usepackage{multicol}
\usepackage{tabu}
\usepackage{url}
\usepackage{fancyhdr}
\usepackage[numbers]{natbib}
\usepackage[all]{xy}
\usepackage{mathtools}
\usepackage[english]{babel}
\usepackage{colortbl}
\usepackage{caption}
\usepackage{hyperref}\usepackage{subcaption}
\usepackage{tikz}\usepackage{float}\usepackage{enumitem}\usepackage{xcolor}\usepackage{tikz-cd}\usepackage{bm}\usepackage{booktabs}
\usepackage{longtable}

\usepackage{mathtools}

\newtheorem{theorem}{Theorem}[section]
\newtheorem{corollary}[theorem]{Corollary}

\newtheorem{lemma}[theorem]{Lemma}
\newtheorem{proposition}[theorem]{Proposition}

\theoremstyle{definition}
\newtheorem{definition}[theorem]{Definition}
\newtheorem{example}[theorem]{Example}
\newtheorem{remark}[theorem]{Remark}

\DeclareMathOperator{\im}{im}

\DeclareMathOperator{\rank}{rank}
\DeclareMathOperator{\sgn}{sgn}
\DeclareMathOperator{\spann}{span_{\mathbb{F}}}

\DeclareMathOperator{\mat}{\mathfrak{mat}}
\DeclareMathOperator{\kh}{\textit{Kh}}

\DeclareMathOperator{\cube}{\mathbb{B}}
\DeclareMathOperator{\cobo}{\mathcal{C}\mathit{ob}}
\DeclareMathOperator{\cob}{\mathcal{C}\mathit{ob}^{3}_{/\sim}}
\DeclareMathOperator{\dob}{\overrightarrow{\mathcal{C}\mathit{ob}}^{3}_{/\sim}}
\DeclareMathOperator{\fun}{\mathbf{Fun}}
\DeclareMathOperator{\hocolim}{\mathrm{hocolim}}
\DeclareMathOperator{\hocof}{\mathrm{hocofib}}

\title{Khovanov homology: pro-tangles, derived colimits and spectral sequences}

\author[1]{Jian Liu}%
\author[1]{Kefeng Liu\thanks{Corresponding author: kefengliu@cqut.edu.cn}}
\author[2]{Li Shen}

\affil[1]{Mathematical Science Research Center, Chongqing University of Technology, Chongqing 400054, China}
\affil[2]{NSF-Simons National Institute for Theory and Mathematics in Biology, Chicago, IL 60637, USA.}

\makeatletter
    \renewcommand*{\@fnsymbol}[1]{\ensuremath{\ifcase#1\or \dagger\or *\or *\or
   \mathsection\or \else\@ctrerr\fi}}
\makeatother
\date{}

\begin{document}
    \maketitle

    \paragraph{Abstract}
    
    This paper introduces \emph{pro-tangles}, a natural generalization of classical tangles, which are functors from the Boolean cube to Bar-Natan's cobordism category. By employing the simplicial Yoneda embedding, we construct the Khovanov simplicial presheaf of a pro-tangle as the total homotopy cofiber over the punctured Boolean cube, and prove that this simplicial presheaf is representable, with representing object the classical Khovanov simplicial object. We establish a fully faithful embedding showing that the weak equivalence class of this simplicial presheaf is determined by the chain homotopy type of the Khovanov complex. Furthermore, we utilize Boolean cube decompositions to construct an algebraic spectral sequence for pro-tangles. This spectral sequence converges to the total Khovanov homology, and its $E_1$ page is explicitly expressed in terms of the Khovanov homology of reduced tangles. This categorical setup yields a functorial interpretation of Reidemeister invariance in terms of morphisms of spectral sequences. By applying the tangle TQFT construction, we study this spectral sequence for Hopf clasps, the fundamental structural building blocks in tangle and link theory. We show that the spectral sequence collapses at the $E_3$ page, which further specializes to an $E_2$-collapse under the restriction to Hopf sums. Finally, we investigate connected sums of pro-tangles and pro-links. To address the module-action dependencies arising from tensor products in multi-connected sums, we introduce a state-dependent modified tensor operator and prove a structural decomposition theorem that generalizes the classical result at the chain complex level.



\tableofcontents 

\section{Introduction}

Knot and link invariants are central objects in low-dimensional topology, offering systematic tools for distinguishing topological types and connecting geometric intuition with algebraic structures \cite{adams2004knot,kauffman2016introduction,rolfsen2003knots}. The development of polynomial invariants began with the Alexander polynomial \cite{alexander1928topological}, followed by the Jones polynomial in the 1980s, which detects knot chirality and initiated the field of quantum topology \cite{jones1985polynomial,jones1987hecke,witten1989quantum}. The HOMFLY-PT polynomial subsequently unified and generalized these invariants \cite{freyd1985new}.

A major breakthrough came with Khovanov's categorification of the Jones polynomial, which lifts a numerical invariant to a homological theory encoding richer topological information \cite{khovanov2000categorification}. Around the same time, Ozsv\'ath--Szab\'o and Rasmussen independently introduced knot Floer homology, a symplectic-geometric invariant that categorifies the Alexander polynomial \cite{ozsvath2004holomorphic,rasmussen2003floer}. These developments are linked by spectral sequences relating Khovanov and knot Floer homologies \cite{dowlin2024spectral,ozsvath2005heegaard}. Further refinements, such as Lee's deformation and Rasmussen's $s$-invariant, have led to powerful applications including bounds on the slice genus and a combinatorial proof of the Milnor conjecture for torus knots \cite{lee2005endomorphism,rasmussen2010khovanov}.

Attention has gradually shifted from links to tangles, which provide a more flexible and compositional framework. Bar-Natan extended Khovanov homology to tangles by constructing a functor from a cobordism category to chain complexes, thereby clarifying its functorial and TQFT-like structure \cite{bar2005khovanov,khovanov2002functor}. Subsequent work resolved foundational issues such as sign ambiguities in functoriality \cite{clark2009fixing}.

From a homotopy-theoretic perspective, Khovanov homology has been lifted to stable homotopy types, yielding stronger invariants. Constructions using the Burnside category and equivariant methods have further expanded its scope and applications \cite{borodzik2021khovanov,lawson2020khovanov,lipshitz2014khovanov}. On the computational side, efficient algorithms have been developed for links \cite{bar2007fast,schmidhuber2025quantum}, while recent work emphasizes tangle-based decompositions and TQFT structures for practical computation. Besides, the simplicial homotopy type of Khovanov homology for closed links was explored via categorical lifts in \cite{everitt2014homotopy} and via state-sum combinatorics in \cite{kauffman2018simplicial}.

Despite these advances, a systematic treatment that simultaneously captures the algebraic, functorial, and homotopy-theoretic aspects of tangles has not been fully established. In this paper, we introduce \emph{pro-tangles} as a natural extension of classical tangles designed to address this gap. Let $\cobo^3(B)$ be the category whose objects are 1-manifolds with boundary $B$, and whose morphisms are compact cornered surface cobordisms relative to $B$. Bar-Natan's cobordism category $\cob(B)$ is the quotient of $\cobo^3(B)$ by the local relations $S$, $T$ and $4Tu$. A pro-tangle is formally defined as a functor $F:  \cube^n \to \cob(B)$ from the Boolean cube to $\cob(B)$. Endowed with the Alexandroff topology on $\cube^n$, pro-tangles acquire a canonical cosheaf structure. The category of pro-tangles, $\fun_B$, is defined as a fibered total category over the class of finite Boolean posets, where objects are diagrammatic functors $F: \cube^n \to \cob(B)$ of variable ranks, and morphisms are structured pairs consisting of graphical functors and natural cobordism transformations.

Fix a ground ring $\mathbf{k}$, and let $\mathcal{A}_B = \mat(\cob(B))$ be the additive closure of the cobordism category $\cob(B)$. Composing the constant simplicial embedding with the simplicial Yoneda embedding yields the functorial assignment
\[
\mathcal{Y}_{\delta} :  \mathcal{A}_B \hookrightarrow s\mathcal{A}_B \xrightarrow{\;\mathcal{Y}\;} s\mathbf{Mod}_{\mathbf{k}}^{\mathcal{A}_B^\mathrm{op}} = \operatorname{Fun}_{\mathbf{k}}(\mathcal{A}_B^\mathrm{op},s\mathbf{Mod}_{\mathbf{k}}),
\]
where the target category $s\mathbf{Mod}_{\mathbf{k}}^{\mathcal{A}_B^\mathrm{op}}$ is equipped with the standard projective model structure.

For a pro-tangle $F: \cube^n \to \cob(B)$, the \textit{Khovanov simplicial presheaf} of $F$ is defined as the total homotopy cofiber of the diagram over the punctured Boolean cube:
\[
\mathbf{Kh}(F) = \operatorname{tcof}_{\cube^n}(\mathcal{Y}_{\delta}\circ\mat\circ F) \;\in\; s\mathbf{Mod}_{\mathbf{k}}^{\mathcal{A}_B^{\mathrm{op}}}.
\]

Let $\mathcal{S}(F)\in s\mathcal{A}_B$ be the Khovanov simplicial object, whose normalized Moore complex recovers the classical Khovanov chain complex. We prove the representability of this presheaf construction (see Theorem~\ref{theorem:representation}):
\begin{theorem}
The Khovanov simplicial presheaf $\mathbf{Kh}(F)$ is represented by the simplicial object $\mathcal{S}(F)$, i.e., for any object $M \in \mathcal{A}_B$, there is a natural weak equivalence of simplicial $\mathbf{k}$-modules
\begin{equation*}
    \mathbf{Kh}(F)(M) \simeq \operatorname{Hom}_{\mathcal{A}_B}\left(M, \mathcal{S}(F)\right).
\end{equation*}
\end{theorem}

Moreover, we establish functoriality of this construction on natural subcategories of pro-tangles (see Theorems~\ref{theorem:homotopy_functoriality} and \ref{theorem:functoriality}):
\begin{theorem}
The Khovanov simplicial presheaf construction descends to a functor between homotopy categories on the subcategory of puncture-preserving pro-tangle morphisms, and restricts to a strict functor on the subcategory of admissible (coweight-preserving) morphisms:
\[
\mathbf{Kh}: \fun_B^{\mathrm{ad}} \longrightarrow  s\mathbf{Mod}_{\mathbf{k}}^{\mathcal{A}_B^{\mathrm{op}}}.
\]
\end{theorem}

We emphasize an important distinction between our construction and the Khovanov homotopy types introduced by Lipshitz--Sarkar~\cite{lipshitz2014khovanov} and Lawson--Lipshitz--Sarkar~\cite{lawson2020khovanov}. 
The Lipshitz--Sarkar invariant constructs a strict stable homotopy type for closed links whose topological invariants refine the classical Khovanov homology by capturing higher-order Steenrod operations. In contrast, our Khovanov simplicial presheaf and its homotopy type are defined for pro-tangles, taking values in the usual unstable homotopy category of additive simplicial presheaves.
This construction supplies a homotopical foundation and topological justification for Bar-Natan's chain homotopies of tangles, lifting purely algebraic cobordism relations to consistent and well-behaved homotopical gluing operations.

Because $\mathcal{A}_B$ lacks arbitrary kernels and cokernels, the category of chain complexes $\mathbf{Ch}_{\ge 0}(\mathcal{A}_B)$ does not admit a standard projective model structure based on quasi-isomorphisms. Instead, it is equipped with the Hurewicz model structure, whose weak equivalences are strict chain homotopy equivalences, and whose localized homotopy category is denoted by $\mathbf{K}_{\ge 0}(\mathcal{A}_B)$. Let $\operatorname{Ho}(s\mathbf{Mod}_{\mathbf{k}}^{\mathcal{A}_B^\mathrm{op}})$ denote the homotopy category of $s\mathbf{Mod}_{\mathbf{k}}^{\mathcal{A}_B^\mathrm{op}}$.

The classical Khovanov complex $\mathcal{C}(F)$ is the normalized Moore complex of $\mathcal{S}(F)$.
The following result guarantees that the weak equivalence class of the Khovanov simplicial presheaf $\mathbf{Kh}(F)$ is completely determined and characterized by the classical chain homotopy type of the Khovanov complex $\mathcal{C}(F)$ (see Theorem \ref{theorem:derived_enriched_embedding} and Corollary \ref{corollary:algebraic_representation}).
\begin{theorem}
The post-composition of the Yoneda embedding with the inverse Dold--Kan correspondence $\mathcal{N}^{-1}$ induces a fully faithful functor of localized categories
\[
\mathcal{J}:  \mathbf{K}_{\ge 0}(\mathcal{A}_B) \;\hookrightarrow\; \operatorname{Ho}\left( s\mathbf{Mod}_{\mathbf{k}}^{\mathcal{A}_B^\mathrm{op}} \right).
\]
Furthermore, there is a natural isomorphism of presheaves of complexes
\[
    \mathcal{N}\left(\mathbf{Kh}(F)\right) \;\cong\; \mathbf{h}_{\mathcal{C}(F)},
\]
where $\mathcal{N}: s\mathbf{Mod}_{\mathbf{k}}^{\mathcal{A}_B^{\mathrm{op}}} \to \mathbf{Ch}_{\ge 0}(\mathbf{k})^{\mathcal{A}_B^{\mathrm{op}}}$ is the objectwise normalized Moore complex functor.
\end{theorem}

Let $\mathcal{A}$ be a $\mathbf{k}$-linear abelian category, and let $\Theta:  \mathcal{A}_B \to \mathcal{A}$ be a $\mathbf{k}$-linear additive functor. The representability established above provides the homotopical justification for a canonical construction of Khovanov homology directly via the linearized diagram $F^\Theta = \Theta \circ \mat\circ F:  \cube^n \to \mathcal{A}$.
We show that this Khovanov homology admits three equivalent characterizations. At the chain complex level, it is given by the cellular cube homology of the diagram $F^\Theta$, a signed totalization constructed via iterated mapping cones over the Boolean cube. Equivalently, it coincides with the homology of the normalized Moore complex of the Khovanov simplicial object with coefficients in $\Theta$. At the simplicial level, it is isomorphic to the simplicial homotopy groups of the total homotopy cofiber of the Eilenberg–MacLane resolution diagram obtained via iterated simplicial homotopy cofibers.

Using Boolean cube decompositions, we construct an algebraic spectral sequence associated with a pro-tangle (see Theorem \ref{theorem:spectral_sequence}). 
Let $F$ be a pro-tangle, and let $\cube^k$ denote the Boolean cube associated with a fixed choice of $k$ crossings in $F$. For each vertex $s \in \cube^k$, let $F_s$ be the corresponding induced pro-tangle smoothing. Then there is a first-quadrant spectral sequence with the $E_1$ page given by
\[
E_1^{p,q} \;\cong\; \bigoplus_{|s|=p} \kh^q(F_s).
\]
This spectral sequence collapses at the $E_{k+1}$ page, i.e., $E_{k+1} \cong E_{\infty}$, and converges to the total Khovanov homology
\[
E_1^{p,q} \;\Rightarrow\; \kh^{p+q}(F).
\]
This construction also provides a functorial interpretation of Reidemeister invariance, formulated in terms of morphisms of spectral sequences.

It is worth noting that we present an algebraic TQFT construction for tangles: we assign the circle to the Frobenius algebra $V=\mathbb{F}[x]/(x^2)$, each open arc to the ideal $W=xV$ of $V$, and further characterize the associated saddle morphisms. This construction enables the explicit computation of pro-tangle Khovanov homology. The Hopf clasp, illustrated in Figure \ref{figure:hopf_clasp}, is a fundamental building block in tangle and link theory.
\begin{figure}[H]
  \centering
  \includegraphics[width=0.5\textwidth]{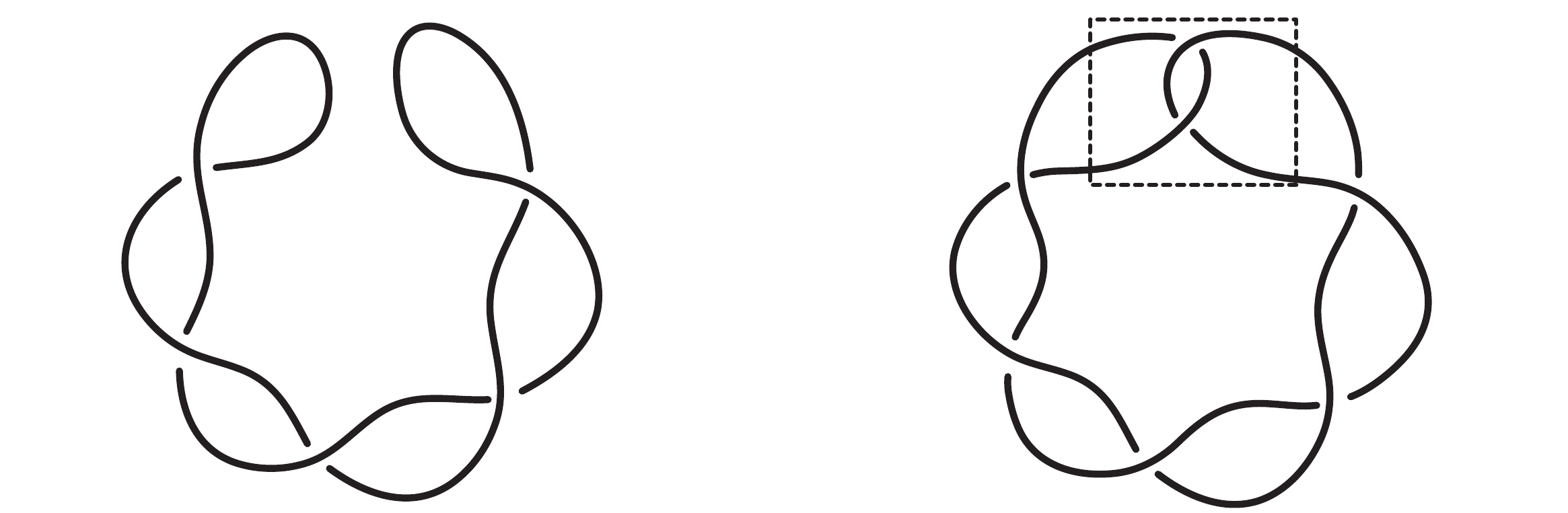}
  \caption{Illustration of the Hopf clasp for the unknot yielding a $7_2$ knot. The dashed enclosure highlights the Hopf clasp region.}\label{figure:hopf_clasp}
\end{figure}
By combining algebraic spectral sequences with the pro-tangle TQFT construction, we determine the Khovanov homology for the typical configuration of the Hopf clasp (see Theorem \ref{theorem:hopf_clasp}).
 
\begin{theorem}
Let $F\colon \cube^n \to \cob(B)$ be a pro-tangle admitting a Hopf clasp at designated crossings $i$ and $j$. There is a first-quadrant spectral sequence that collapses at the $E_3$ page ($E_3 \cong E_{\infty}$) and converges to the total Khovanov homology:
\[
E_1^{p,q} \;\Rightarrow\; \kh^{p+q}(F).
\]
Furthermore, the $E_1$ page is concentrated in columns $p \in \{0, 1, 2\}$, and there exists a pro-tangle $G\colon \cube^{n-2} \to \cob(B)$ such that the $E_1^{p,\bullet}$ columns are given by
\begin{align*}
    E_1^{0, \bullet} &\cong \kh(G) \otimes V, \\
    E_1^{1, \bullet} &\cong \kh(G) \oplus \kh(G), \\
    E_1^{2, \bullet} &\cong \kh(F_{1,1})
\end{align*}
for a Hopf clasp of type \textup{(I)}, and 
\begin{align*}
    E_1^{0, \bullet} &\cong \kh(F_{0,0}), \\
    E_1^{1, \bullet} &\cong \kh(G) \oplus \kh(G), \\
    E_1^{2, \bullet} &\cong \kh(G) \otimes V
\end{align*}
for a Hopf clasp of type \textup{(II)}. Here, $F_{0,0}$ and $F_{1,1}$ denote the $(0,0)$-smoothing and $(1,1)$-smoothing of $F$ at crossings $i$ and $j$, respectively.
\end{theorem}

In particular, the spectral sequences for two specific configurations of the Hopf clasp, namely the Hopf sum and the Hopf twist illustrated in Figure \ref{figure:introduction_hopf}, are explicitly characterized to facilitate the computation of the Khovanov homology of tangles (see Theorems \ref{theorem:hopf_sum} and \ref{theorem:hopf_twist}). It is worth noting that the spectral sequence for Hopf sum collapses at the $E_2$ page.
\begin{figure}[H]
  \centering
  \includegraphics[width=0.8\textwidth]{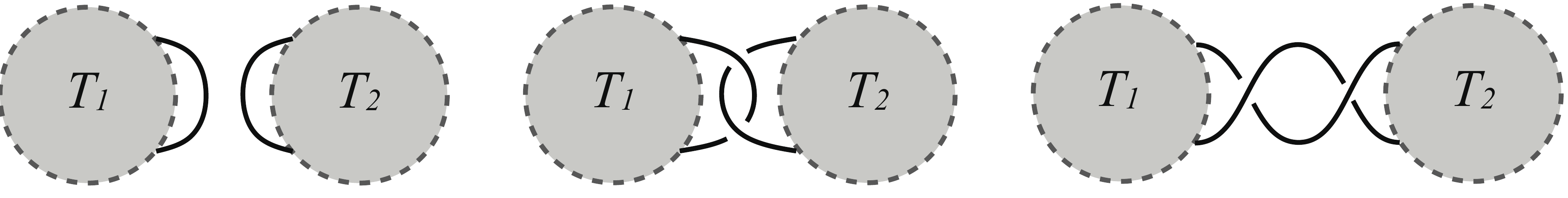}
  \caption{Illustration of Hopf sum and Hopf twist: a pair of disjoint tangles (left), their type \textup{(I)} Hopf sum (center), and their type \textup{(I)} Hopf twist (right).}\label{figure:introduction_hopf}
\end{figure}

Finally, we investigate planar connected sums of pro-tangles and pro-links. We show that the Khovanov complex of a connected sum decomposes as the tensor product of their respective Khovanov complexes over $V$, leading to a Künneth spectral sequence that computes the Khovanov homology of the resulting connected sum.

It is worth noting that a pro-tangle structured upon a collection of $k$ designated base strands inherits a $V^{\otimes k}$-module structure. Consequently, the multi-connected sums of pro-tangles and pro-links can become highly intricate, as tensor products over $V^{\otimes k}$ inevitably introduce complex topological interference. 

To formalize this for a double connected sum, let $F$ be a pro-tangle and let $L$ be a pro-link. Suppose $\mathbf{R}$ denotes a set of two distinct pairs of junction strands between $F$ and $L$ that govern the connected sum. We introduce a modified correction factor $\otimes^{\mathbf{R}} V$, where the specialized tensor operator $\otimes^{\mathbf{R}}$ acts on each state $s$ as $\otimes  V^{\otimes I_{\mathbf{R}}(s)}$. Here, the indicator $I_{\mathbf{R}}(s) \in \{0, 1\}$ is determined entirely by the local conformation of the connected sum at the given state. This allows us to establish the following modified decomposition for the Khovanov chain complex of the connected sum (see Theorem \ref{thm:dual_junction_isomorphism}).

\begin{theorem}
There exists an isomorphism of chain complexes
\begin{equation*}
    \mathcal{C}(F \#_{\mathbf{R}} L) \;\cong\; \left(\mathcal{C}(F) \otimes_{V \otimes V} \mathcal{C}(L)\right) \otimes^{\mathbf{R}} V.
\end{equation*}
\end{theorem}

The paper is organized as follows. Section \ref{section:protangles} introduces pro-tangles. Section \ref{section:protangles_invariants} develops their homotopy and homology theory. Section \ref{section:spectral_sequence} constructs the associated spectral sequence. Section \ref{section:pro-tangle_tqft} presents explicit computations. The final section studies planar connected sums of pro-tangles and pro-links.

\section{Pro-tangles}\label{section:protangles}

In this section, we introduce the concept of pro-tangles as a generalization of tangles, which possess functorial properties, and present their related notions. Additionally, we re-interpret the geometry of pro-tangles from the view of cosheaves.

\subsection{From tangles to pro-tangles}

\subsubsection{Tangles as functors}

In knot theory, a \emph{tangle} is a proper embedding of arcs and circles in a $3$-ball (or equivalently, in the cylinder $D^2 \times I$) such that all endpoints of the arcs lie on the boundary disk $D^2 \times \{0\} \cup D^2 \times \{1\}$. More precisely, an $(n,m)$-\emph{tangle} consists of $n$ arcs connecting the top boundary to the bottom boundary and $m$ closed circles in the interior.

A \emph{tangle diagram} $T$ is a generic projection of a tangle onto a disk (or $\mathbb{R}\times I$), equipped with over/under information at each double point, resulting in a finite number of crossings. 

From the categorical perspective, tangles can be interpreted using the category of $3$-dimensional cobordisms. Let $\cobo^3(B)$ denote the category whose objects are $1$-manifolds with boundary $B$ (typically disjoint unions of intervals and circles corresponding to smoothings of the diagram), and whose morphisms are compact surfaces with corners arising from the boundary points in $B$. 

The category $\cob(B)$ is the quotient of the category $\cobo^3(B)$ by the local relations $S$, $T$, and $4Tu$ introduced in \cite{bar2005khovanov}. These relations are as follows: $S$ states that closed spheres evaluate to zero; $T$ states that a disjoint tube multiplies the cobordism by two; and $4Tu$ is a specific relation on a sphere with four tubes. The resulting quotient category assigns to each tangle a chain complex whose homotopy class is a tangle invariant and categorifies the Jones polynomial via Khovanov homology. 

Bar-Natan's perspective is to view a tangle as a category. In this work, we attempt to approach tangles from a different angle by considering them as functors. 

Let $\cube = \{0 \to 1\}$ denote the interval category, and let $\cube^n$ be the $n$-fold Cartesian product category, representing the $n$-dimensional cubical lattice of resolutions. For a tangle diagram $T$ with $n$ ordered crossings, the associated tangle functor of $T$ is a functor $\mathcal{F}_T: \cube^n \to \cob(B)$, where $B$ denotes the set of boundary points. For each object $v \in \cube$, the assignment $\mathcal{F}_T(v)$ yields a 1-manifold. For each morphism (edge) $e: v \to v'$ where $|v'| = |v| + 1$, $\mathcal{F}_T(e)$ is a generating cobordism representing a single saddle transformation. In this setting, a tangle can be regarded as a functor.

\subsubsection{The category of pro-tangles}

Motivated by the functorial perspective of tangles, we introduce the notion of pro-tangles, where ordinary tangles can be regarded as special pro-tangles.

\begin{definition}[Pro-tangle]
A \textit{pro-tangle} of rank $n$ is a functor 
\[ F: \cube^{n} \to \cob(B) \]
which assigns to each resolution vertex $v \in \{0,1\}^n$ an object in $\cob(B)$ and to each edge a cobordism (morphism).
\end{definition}

A pro-tangle $L:\cube^{n} \to \cob(\emptyset)$ is referred to as a \textit{pro-link}. When $n=0$, the pro-link $L:\cube^{0} \to \cob(\emptyset)$ is \textit{unknotted}.

\begin{remark}
We use the term pro-tangle to indicate a pro-type combinatorial construction. A pro-tangle of rank $n$ is a functor defined on the $n$-fold Boolean cube $\cube^n$, and as $n$ varies, such objects form a projective system via the face inclusions. This is analogous to, but distinct from, the categorical notion of a pro-object, which is a formal cofiltered limit of objects in a category. Pro-tangles as defined here are \emph{not} pro-objects in the strict categorical sense; they are ordinary functors on a finite poset.
\end{remark}

A functor $p: \cube^{n} \to \cube^{m}$ is a \textit{graphical functor} if it is induced by a morphism of the underlying cube graphs. Specifically, for any elementary edge $e$ in $\cube^n$, its image $p(e)$ must be either an identity morphism or a single elementary edge in $\cube^m$, thereby prohibiting $p$ from mapping an edge to a composite of multiple edges (a diagonal).

\begin{definition}
The category of pro-tangles $\fun_B$ consists of the following data:
\begin{itemize}
    \item \textit{Objects:} functors $F: \cube^{n} \to \cob(B)$ for any $n \in \mathbb{N}_{\ge 0}$.
    \item \textit{Morphisms:} a morphism $\Phi: F \to G$ between $F: \cube^{n} \to \cob(B)$ and $G: \cube^{m} \to \cob(B)$ is a pair $(p, \eta)$, where:
    \begin{enumerate}[label=$(\roman*)$]
        \item $p: \cube^{n} \to \cube^{m}$ is a graphical functor between the resolution spaces.
        \item $\eta: F \Rightarrow G \circ p$ is a natural transformation.
    \end{enumerate}
\end{itemize}
\end{definition}

The data of a morphism $\Phi = (p, \eta)$ is characterized by the following diagrammatic filler:
\begin{equation*}
\vcenter{
\xymatrix@R=1pc@C=1pc{
\cube^{n} \ar[rrdd]^-{F} \ar[dd]_-{p} &&  \\
&&\\
\cube^{m} \ar[rr]_-{G}\ar@{<=}[ru]^-{\eta} && \cob(B).
}}
\end{equation*}

Given morphisms $\Phi = (p, \eta): F \to G$ and $\Psi = (q, \zeta): G \to H$, their composition $\Psi \circ \Phi: F \to H$ is defined by the pair
\begin{equation*}
  (q \circ p, (\zeta p) \circ \eta),
\end{equation*}
where $\zeta p: G \circ p \Rightarrow H \circ q \circ p$ is the whisker composition of the natural transformation $\zeta$ with the functor $p$.

\begin{remark}
In this setting, the functor $p$ represents a combinatorial re-indexing of crossings or smoothing choices, while $\eta$ provides the geometric cobordisms between the resulting 1-manifolds. This allows for a treatment of Reidemeister moves as natural transformations within $\fun_B$.
\end{remark}

\begin{lemma}
Let $\Phi = (p, \eta): F \to G$ and $\Psi = (q, \zeta): G \to H$ be morphisms in $\fun_B$. Their composition $\Psi \circ \Phi$ is a well-defined morphism in $\fun_B$.
\end{lemma}

\begin{proof}
The composition $q \circ p: \cube^n \to \cube^k$ is a functor by the standard composition in $\mathbf{Cat}$. For any object $v \in \cube^n$, the map
\begin{equation*}
    F(v) \xrightarrow{\eta_v} G(p(v)) \xrightarrow{\zeta_{p(v)}} H(q(p(v)))
\end{equation*}
gives a valid composition of cobordisms in $\cob(B)$. 

To verify naturality, let $f: v \to w$ be a morphism in $\cube^n$. The naturality of $\eta$ and $\zeta$ yields the following commuting diagram
\begin{equation*}
\xymatrix@C=4pc{
F(v) \ar[r]^{\eta_v} \ar[d]_{F(f)} & G(p(v)) \ar[r]^{\zeta_{p(v)}} \ar[d]_{G(p(f))} & H(q(p(v))) \ar[d]^{H(q(p(f)))} \\
F(w) \ar[r]_{\eta_w} & G(p(w)) \ar[r]_{\zeta_{p(w)}} & H(q(p(w)))}
\end{equation*}
The horizontal concatenation $(\zeta p)_w \circ \eta_w \circ F(f) = H(q(p(f))) \circ (\zeta p)_v \circ \eta_v$ follows immediately from the commutativity of the left and right squares. Thus, $(\zeta p) \circ \eta: F \to H \circ (q \circ p)$ is a natural transformation.
\end{proof}

The structure $(\fun_B, \circ, \mathrm{id})$ satisfies the categorical axioms:
\begin{enumerate}[label=$(\roman*)$]
    \item \textit{Associativity:} for any three morphisms $\Phi, \Psi, \Xi$, we have $\Xi \circ (\Psi \circ \Phi) = (\Xi \circ \Psi) \circ \Phi$. This follows from the associativity of functor composition in $\mathbf{Cat}$ and the associativity of vertical and whisker compositions of natural transformations.
    \item \textit{Identity:} for each object $F: \cube^n \to \cob(B)$, the identity morphism is defined as $\mathbf{1}_F = (\mathrm{id}_{\cube^n}, \mathrm{id}_F)$, where $\mathrm{id}_{\cube^n}$ is the identity functor on $\cube^n$ and $\mathrm{id}_F: F \Rightarrow F$ is the identity natural transformation. It is trivial to verify that $\mathbf{1}_G \circ \Phi = \Phi = \Phi \circ \mathbf{1}_F$ for any $\Phi: F \to G$.
\end{enumerate}

We denote the full subcategory of $\fun_B$ consisting of pro-tangles of rank $n$ by $\fun_B^{(n)}$.

\begin{definition}
Let $F \in \fun^{(n)}_B$ and $G \in \fun^{(m)}_{B'}$ be pro-tangles with disjoint boundaries $B$ and $B'$, respectively. The \textit{disjoint union} of $F$ and $G$ is defined as the pro-tangle functor
\begin{equation*}
    F \sqcup G: \cube^{n+m} \longrightarrow \cob(B \sqcup B')
\end{equation*}
such that for any vertex $(v, w) \in \cube^{n} \times \cube^{m}$, the resolution is given by the spatial disjoint union $(F \sqcup G)(v, w) = F(v) \sqcup G(w)$.
\end{definition}

\subsubsection{Combinatorial and geometric perspectives on tangles}

In the combinatorial formulation, a tangle $T$ is encoded as a functor $\mathcal{F}_T: \cube^n \to \cob(B)$ from the Boolean cube to the Bar-Natan cobordism category. Comparatively, the geometric category $\cobo^4(B)$ (following Bar-Natan \cite{bar2005khovanov}) treats tangles as boundaries of $4$-dimensional manifolds, where morphisms are generic link cobordisms---$2$-dimensional surfaces embedded in $D \times I$ up to boundary-relative isotopy. While $\cobo^4(B)$ captures continuous topological deformations and global link isotopies, the functorial approach discretizes these transitions across a combinatorial lattice.

\begin{remark}
In general, there is no natural functor $\mathcal{F}: \cobo^4(B) \to \fun_B$. The objects in $\fun_B$ are rigid, highly structured functors that do not necessarily arise as the resolutions of a realizable geometric tangle diagram.
\end{remark}

The relationship between a natural transformation $\eta: F \Rightarrow G \circ p$ in $\fun_B$ and a geometric cobordism $C \in \text{Hom}_{\cobo^4}(T, T')$ is characterized by a functorial localization. Specifically, $\eta$ is uniquely determined by a family of $2$-dimensional cobordisms $\{ \eta_v \}_{v \in \text{Ob}(\cube^n)}$ in $\cob(B)$. For each morphism $e: v \to v'$ in $\cube^n$, the compatibility of these local data with the smoothing functors is expressed by the commutativity of the following diagram:
\begin{equation*}
\begin{tikzcd}[column sep=large, row sep=large]
F(v) \arrow[r, "F(e)"] \arrow[d, "\eta_v"'] & F(v') \arrow[d, "\eta_{v'}"] \\
G(p(v)) \arrow[r, "G(p(e))"'] & G(p(v')).
\end{tikzcd}
\end{equation*}
Geometrically, $\eta$ constitutes a stratified $2$-dimensional cobordism. In contrast to a global morphism in $\cobo^4(B)$, a morphism in $\fun_B$ decomposes the topological transition into cellular components indexed by the face posets of $\cube^n$. Consequently, $\eta$ provides a strict combinatorial lift of the continuous link cobordism, preserving the grading configurations required for the construction of the Khovanov chain maps.

\subsection{Smoothings and crossing map}

The smoothing of crossings in a tangle diagram is a fundamental operation in the construction of the Jones polynomial and Khovanov homology. In this section, we formalize these local topological transitions within the language of pro-tangles, characterizing the smoothing operations as categorical functors and natural transformations.

\subsubsection{Smoothings of pro-tangles}

\begin{definition}
Let $F: \cube^{n} \to \cob(B)$ be a pro-tangle of rank $n$. For any index $1 \le i \le n$ and resolution $\epsilon \in \{0, 1\}$, the \textit{$i$-th $\epsilon$-smoothing} $F^{(i)}_\epsilon$ is a pro-tangle of rank $n-1$ defined by the composition
\begin{equation*}
    F^{(i)}_\epsilon = F \circ \delta_{i,\epsilon},
\end{equation*}
where $\delta_{i,\epsilon}: \cube^{n-1} \to \cube^n$ is the $i$-th face inclusion functor
\begin{equation*}
    \delta_{i,\epsilon}(x_1, \dots, x_{n-1}) = (x_1, \dots, x_{i-1}, \epsilon, x_i, \dots, x_{n-1}).
\end{equation*}
\end{definition}

\begin{remark}
Although smoothing reduces the rank of a pro-tangle, the resulting object $F^{(i)}_\epsilon$ is formally related to its parent $F$ via a morphism in $\fun_B$. Specifically, the \textit{smoothing inclusion} $\mathcal{I}_{i,\epsilon}: F^{(i)}_\epsilon \to F$ is defined by the pair
\begin{equation*}
    \mathcal{I}_{i,\epsilon} = (\delta_{i,\epsilon}, \mathrm{id}_{F^{(i)}_\epsilon}),
\end{equation*}
where $\mathrm{id}_{F^{(i)}_\epsilon}: F^{(i)}_\epsilon \Rightarrow F \circ \delta_{i,\epsilon}$ is the identity natural transformation.
\end{remark}

\begin{definition}
For a pro-tangle $F$ of rank $n$, the categorical transition between its $0$-resolution and $1$-resolution at the $i$-th crossing is defined as the \textit{crossing map} $\Phi_i: F^{(i)}_0 \to F^{(i)}_1$. It is the morphism in $\fun_B$ represented by the pair $(\mathrm{id}_{\cube^{n-1}}, \eta^{(i)})$, where the natural transformation $\eta^{(i)}: F^{(i)}_0 \Rightarrow F^{(i)}_1$ is given by
\begin{equation*}
    \eta^{(i)}_v = F(e_{v,i}), \quad \forall v \in \cube^{n-1}.
\end{equation*}
Here, $e_{v,i}$ is the unique edge in $\cube^n$ from $\delta_{i,0}(v)$ to $\delta_{i,1}(v)$.
\end{definition}

\subsubsection{Compatibility of smoothings}

The consistency of smoothing operations under the morphisms of $\fun_{B}$ ensures that local topological resolutions are compatible with the global categorical structure.

\begin{theorem}
Let $F: \cube^n \to \cob(B)$ and $G: \cube^m \to \cob(B)$ be pro-tangles. Let $\Phi = (p, \eta): F \to G$ be a morphism in $\fun_{B}$. Suppose the functor $p: \cube^n \to \cube^m$ restricts to the faces of the resolution cubes, i.e., there exists a functor $p': \cube^{n-1} \to \cube^{m-1}$ such that $p \circ \delta_{i,\epsilon} = \delta_{j,\epsilon} \circ p'$ for fixed $i, j$ and $\epsilon \in \{0, 1\}$. Then there exists an induced morphism $\Phi^{(i,j)}_\epsilon: F^{(i)}_\epsilon \to G^{(j)}_\epsilon$ in $\fun_{B}$ defined by the pair $(p', \eta \cdot \delta_{i,\epsilon})$, and the following diagram commutes:
\begin{equation*}
\begin{tikzcd}
F^{(i)}_\epsilon \arrow[r, "\Phi^{(i,j)}_\epsilon"] \arrow[d, "\mathcal{I}^F_{i,\epsilon}"'] & G^{(j)}_\epsilon \arrow[d, "\mathcal{I}^G_{j,\epsilon}"] \\
F \arrow[r, "\Phi"'] & G.
\end{tikzcd}
\end{equation*}
\end{theorem}

\begin{proof}
The induced morphism $\Phi^{(i,j)}_\epsilon = (p', \eta \cdot \delta_{i,\epsilon})$ is well-defined if the component $(\eta \cdot \delta_{i,\epsilon})_u$ maps the resolution $(F^{(i)}_\epsilon)(u)$ to the target $(G^{(j)}_\epsilon \circ p')(u)$ for each $u \in \cube^{n-1}$. By the definition of face inclusions and the naturality of $\eta$, we have
\begin{equation*}
    (\eta \cdot \delta_{i,\epsilon})_u = \eta_{\delta_{i,\epsilon}(u)} :  F(\delta_{i,\epsilon}(u)) \longrightarrow G(p(\delta_{i,\epsilon}(u))).
\end{equation*}
Applying the commutativity of the resolution spaces $p \circ \delta_{i,\epsilon} = \delta_{j,\epsilon} \circ p'$, the target side transforms as
\begin{equation*}
    G(p(\delta_{i,\epsilon}(u))) = G(\delta_{j,\epsilon}(p'(u))) = (G^{(j)}_\epsilon \circ p')(u).
\end{equation*}
This confirms that $\eta \cdot \delta_{i,\epsilon}$ is a valid natural transformation between the restricted functors.

To verify the commutativity of the diagram, we compare the compositions $\Phi \circ \mathcal{I}^F_{i,\epsilon}$ and $\mathcal{I}^G_{j,\epsilon} \circ \Phi^{(i,j)}_\epsilon$ in $\fun_B$. Recalling the composition rule $(p_2, \eta_2) \circ (p_1, \eta_1) = (p_2 \circ p_1, (\eta_2 p_1) \cdot \eta_1)$, the left-hand path yields
\begin{equation*}
    (p, \eta) \circ (\delta_{i,\epsilon}, \text{id}) = (p \circ \delta_{i,\epsilon}, \eta \cdot \delta_{i,\epsilon}).
\end{equation*}
Correspondingly, the right-hand path yields
\begin{equation*}
    (\delta_{j,\epsilon}, \text{id}) \circ (p', \eta \cdot \delta_{i,\epsilon}) = (\delta_{j,\epsilon} \circ p', \text{id} \cdot (\eta \cdot \delta_{i,\epsilon})) = (\delta_{j,\epsilon} \circ p', \eta \cdot \delta_{i,\epsilon}).
\end{equation*}
The assumption $p \circ \delta_{i,\epsilon} = \delta_{j,\epsilon} \circ p'$ ensures that both functors and natural transformations coincide, thereby completing the proof.
\end{proof}

\subsection{The geometry of pro-tangles}

\subsubsection{Alexandroff topology on Boolean cubes}

A geometric interpretation of the Boolean cube is facilitated by the \textit{simplicial nerve} $N(\cube)$. Within this view, the Boolean $n$-cube $\cube = \{0, 1\}^n$ is realized as a simplicial complex, where each $k$-simplex $\sigma = (v_0 < v_1 < \dots < v_k)$ corresponds to an ordered chain under the natural partial ordering of the Boolean cube. 

The resulting \textit{geometric realization} $|N(\cube)|$ provides a canonical topological space that mirrors the underlying categorical structure. By virtue of the nerve functor, the discrete combinatorial architecture of $\cube$ is faithfully embedded into a continuous topological space.

On the other hand, the Boolean $n$-cube $\cube = \{0, 1\}^n$ is initially a discrete set of $2^n$ smoothing states. We endow it with the downward \textit{Alexandroff Topology} induced by the partial order $v \le v'$. 
In this topology, a set $U \subseteq \cube$ is open if it is a lower set:
\begin{equation*}
  v \in U \text{ and } v' \le v \implies v' \in U.
\end{equation*}
The minimal open neighborhoods $U_v = \{ v' \in \cube \mid v' \le v \}$ form a basis for this topology.

An important consequence of McCord's theorem is that for a finite Alexandroff space $X$, the geometric realization of its nerve (equivalently, its order complex) $|N(X)|$ is weakly homotopy equivalent to $X$ \cite{mccord1966singular}. In particular, if $\cube$ denotes the Boolean cube viewed as a finite Alexandroff space, then $\cube$ is weakly homotopy equivalent to the geometric realization of its simplicial nerve $N(\cube)$.

\subsubsection{Regarding a pro-tangle as a cosheaf}

The embedding of the Boolean cube into its dual Alexandroff space induces a natural cosheaf structure associated to the pro-tangle $F: \cube^n \to \cob(B)$.

\begin{definition}
For each vertex $v \in \cube^n$, let $U_v = \{ w \in \cube^n \mid w \le v \}$ denote the minimal open neighborhood of $v$ in the downward Alexandroff topology. The costalk of the cosheaf $\mathcal{S}_F$ at $v$ is defined by $\Gamma(U_v, \mathcal{S}_F) = F(v)$. For a nested pair of open sets $U_v \subseteq U_{v'}$ (which holds if and only if $v \le v'$), the corestriction morphism 
\begin{equation*}
    \iota_{U_v, U_{v'}}: \Gamma(U_v, \mathcal{S}_F) \longrightarrow \Gamma(U_{v'}, \mathcal{S}_F)
\end{equation*}
is defined as the cobordism $F(v \to v') \in \text{Hom}_{\cob(B)}(F(v), F(v'))$.
\end{definition}

\begin{proposition}
The assignment $\mathcal{S}_F$ constitutes a cosheaf on the downward Alexandroff space $\cube^n$.
\end{proposition}

\begin{proof}
The topology of the finite poset $\cube^n$ is generated by the basis of principal down-sets $U_v$. The functoriality of $F$ ensures that the assignments $\Gamma(U_v, \mathcal{S}_F)$ and $\iota_{U_v, U_{v'}}$ satisfy the identity and transitivity axioms of a pre-cosheaf. For an arbitrary open set $U \subseteq \cube^n$, the cosheaf property requires the canonical morphism $\varinjlim_{U_v \subseteq U} \Gamma(U_v, \mathcal{S}_F) \to \Gamma(U, \mathcal{S}_F)$ to be an isomorphism. This condition is satisfied identically via the left Kan extension of the functor $F$ from the basis of principal open sets to the lattice of open sets \cite[Chapter~X, Section~3]{mac1971categories}.
\end{proof}

The functorial assignment of $\mathcal{S}_F$ allows for a localized decomposition of the pro-tangle complex over the simplicial nerve $N(\cube^n)$. Each corestriction $\iota_{U_v, U_{v'}}$ corresponds to a canonical $2$-dimensional morphism in $\cob(B)$ acting on the local $1$-manifold components.

The cosheaf $\mathcal{S}_F$ is constructible with respect to the cubical stratification, meaning the costalk functors are locally constant on the open cells. For any $2$-simplex $\sigma = (v \to v' \to w)$ in $N(\cube^n)$, the functorial composition law enforces the commuting triangle:
\begin{equation*}
    \iota_{U_{v'}, U_w} \circ \iota_{U_v, U_{v'}} = \iota_{U_v, U_w}.
\end{equation*}
This commutativity guarantees that the local grading configurations are invariant under the permutation of smoothing choices. Furthermore, the cellular boundaries of these $2$-simplices encode the signs of the combinatorial differential; by equipping the corestrictions with the standard cubical sign assignment $(-1)^{\sum_{i < j} v_i}$, the cosheaf transitions recover the nilpotency property $\partial^2 = 0$ as a consequence of the local commutativity of the face maps.

Consequently, the global topological character of the pro-tangle is fully determined by the combinatorial consistency of $\mathcal{S}_F$ on $\cube^n$. The construction of the associated cosheaf homology is achieved by composing $\mathcal{S}_F$ with a functor from $\cob(B)$ into a target abelian category.

\subsection{Orientation of pro-tangles}

The oriented cobordism category $\dob(B)$ is the category whose objects are oriented $1$-manifolds $M$ with fixed boundary $\partial M = B$, and whose morphisms are oriented cobordisms $S: M_0 \to M_1$ relative to the cylinder $B \times [0,1]$, subject to the $(S, T, 4Tu)$ relations of the Bar-Natan category.

\begin{definition}
An \textit{oriented pro-tangle} of rank $n$ is a functor
\begin{equation*}
    F: \cube^n \longrightarrow \dob(B),
\end{equation*}
which assigns each vertex to an oriented resolution and each edge to a canonical oriented cobordism.
\end{definition}

\begin{definition}
For an oriented pro-tangle $F: \cube^n \longrightarrow \dob(B)$, an \textit{orientation} is a choice of a vertex $v \in \cube^n$.
\end{definition}

The chosen orientation vertex $v$ is called the \textit{Seifert vertex}. The corresponding object $F(v)$ is referred to as the \textit{Seifert resolution}, which is the unique state such that all local smoothings are mutually compatible with respect to the specified orientation flow of the pro-tangle.

\begin{definition}\label{definition:orientation}
Let $F: \cube^n \longrightarrow \dob(B)$ be an oriented pro-tangle with an orientation $v=(v_1,\dots,v_n) \in \cube^n$. For each $1\leq i\leq n$, the $i$-th crossing of the pro-tangle is called \textit{right-handed} if $v_i = 0$, and \textit{left-handed} otherwise.
\end{definition}

For an oriented pro-tangle $F$ with a fixed orientation, the number of right-handed crossings is denoted by $n_{+}(F)$, or simply $n_{+}$, while the number of left-handed crossings is denoted by $n_{-}(F)$, or $n_{-}$ for brevity. In the context of the classical writhe convention, right-handed crossings correspond to positive crossings with local writhe $+1$, whereas left-handed crossings correspond to negative crossings with local writhe $-1$.

\begin{remark}
We clarify the relationship between the Seifert vertex $v \in \mathbb{B}^n$ and the classical Seifert circle construction. Given an oriented tangle diagram $T$ with $n$ crossings, the \emph{Seifert resolution} of each crossing is determined by the orientation. In Bar-Natan's convention \cite{bar2005khovanov}, the Seifert resolution is the $0$-smoothing at a positive crossing and the $1$-smoothing at a negative crossing.
The Seifert vertex $v_{\mathrm{Seifert}}$ is therefore the vertex with $v_i=0$ if crossing $i$ is right-handed and $v_i=1$ if crossing $i$ is left-handed. This is consistent with Definition~\ref{definition:orientation}, where the $i$-th crossing is right-handed if and only if $v_i=0$.
\end{remark}

\section{Homotopy invariants for pro-tangles}\label{section:protangles_invariants}

This section lays the geometric and homological foundations of pro-tangle invariants. Given a pro-tangle functor, we construct its associated Khovanov simplicial presheaf and prove its representability, identifying the classical Khovanov simplicial object as the underlying representing object. Via the Dold–Kan correspondence, normalization of this simplicial presheaf yields a presheaf of chain complexes, represented by the Khovanov complex—the Moore complex of the Khovanov simplicial object. The homotopy category of Khovanov complexes then provides the algebraic setup for classifying weak equivalence classes of Khovanov simplicial presheaves. To extract the corresponding homological invariants, we pass to an abelian target category to define generalized Khovanov homology groups. This homological construction admits three equivalent characterizations, establishing generalized Khovanov homology as the homological realization of the Khovanov simplicial presheaf.

\subsection{The homotopy type of tangles}\label{subsection:Khovanov_simplicial_presheaf}

\subsubsection{Khovanov simplicial presheaf}

Let $\cob(B)$ be the pre-additive category of $1$-manifolds and cobordisms with fixed boundary $B$. Suppose the ground ring $\mathbf{k}$ is a commutative ring.
The \textit{additive closure} $\mathcal{A}_B = \mat(\cob(B))$ is defined as the formal matrix closure of this category. The objects of $\mathcal{A}_B$ are finite formal direct sums $\bigoplus_{i} M_i$ of objects in $\cob(B)$. For any pair of objects $M, M' \in \mathcal{A}_B$, the morphism spaces $\operatorname{Hom}_{\mathcal{A}_B}(M, M')$ inherit a canonical $\mathbf{k}$-module structure induced by matrix-valued $\mathbf{k}$-linear combinations of the cobordism classes.

Let $s\mathcal{A}_B = \operatorname{Fun}(\Delta^\mathrm{op}, \mathcal{A}_B)$ denote the category of simplicial objects over $\mathcal{A}_B$. The discrete embedding is given by the constant simplicial functor $\delta: \mathcal{A}_B \hookrightarrow s\mathcal{A}_B$. Let $\mathbf{Mod}_{\mathbf{k}}$ denote the category of modules over the ground ring $\mathbf{k}$, and let $s\mathbf{Mod}_{\mathbf{k}} = \operatorname{Fun}(\Delta^\mathrm{op}, \mathbf{Mod}_{\mathbf{k}})$ be the category of simplicial $\mathbf{k}$-modules. We define the simplicial Yoneda embedding into the category of simplicial $\mathbf{k}$-modules by
\begin{equation*}
    \mathcal{Y}: s\mathcal{A}_B \longrightarrow \operatorname{Fun}_{\mathbf{k}}(\mathcal{A}_B^{\mathrm{op}}, s\mathbf{Mod}_{\mathbf{k}}), \quad X_\bullet \longmapsto \operatorname{Hom}_{\mathcal{A}_B}(-, X_\bullet).
\end{equation*}
Let $s\mathbf{Mod}_{\mathbf{k}}^{\mathcal{A}_B^{\mathrm{op}}} = \operatorname{Fun}_{\mathbf{k}}(\mathcal{A}_B^{\mathrm{op}}, s\mathbf{Mod}_{\mathbf{k}})$ denote the category of linear simplicial presheaves over the additive closure.

\begin{proposition}[\cite{goerss2009simplicial,jardine2015local}]\label{proposition:global_model}
Since $s\mathbf{Mod}_{\mathbf{k}}$ is closed under all limits and colimits and possesses sufficiently many projectives, $s\mathbf{Mod}_{\mathbf{k}}^{\mathcal{A}_B^{\mathrm{op}}}$ inherits a canonical cofibrantly generated projective model structure. A morphism $f: \mathcal{X}_\bullet \to \mathcal{Y}_\bullet$ is a weak equivalence (resp. fibration) if and only if for each object $M \in \mathcal{A}_B$, the map $f(M): \mathcal{X}_\bullet(M) \to \mathcal{Y}_\bullet(M)$ is a weak equivalence (resp. Kan fibration) in $s\mathbf{Mod}_{\mathbf{k}}$.
\end{proposition}

The composition of the discrete embedding with the simplicial Yoneda embedding induces a functorial assignment $\mathcal{Y} \circ \delta: \mathcal{A}_B \hookrightarrow s\mathbf{Mod}_{\mathbf{k}}^{\mathcal{A}_B^{\mathrm{op}}}$, under which constant simplicial objects are mapped to representable linear simplicial presheaves.

We denote by $\mathbf{Ch}_{\ge 0}(\mathbf{Mod}_{\mathbf{k}})$ the category of non-negatively graded chain complexes over $\mathbf{Mod}_{\mathbf{k}}$. The classical Dold--Kan correspondence asserts the existence of an equivalence of categories:
\begin{equation*}
    \mathbf{N}: s\mathbf{Mod}_{\mathbf{k}} \xrightarrow{ \simeq } \mathbf{Ch}_{\ge 0}(\mathbf{Mod}_{\mathbf{k}}),
\end{equation*}
where $\mathbf{N}$ is the \textit{normalized chain complex functor}. For any simplicial module $A_\bullet \in s\mathbf{Mod}_{\mathbf{k}}$, its image $\mathbf{N}(A_\bullet)$ is the chain complex whose degree $k$ term is defined by the intersection of the kernels of the face maps:
\begin{equation*}
    \mathbf{N}(A_\bullet)_k = \bigcap_{i=1}^k \ker(d_i: A_k \to A_{k-1}),
\end{equation*}
with the differential $\partial_k: \mathbf{N}(A_\bullet)_k \to \mathbf{N}(A_\bullet)_{k-1}$ given by the restriction of the remaining face map $d_0$. Under this equivalence, the simplicial homotopy groups $\pi_k(A_\bullet)$ are identified with the homology groups of the normalized complex.

\begin{lemma}\label{lemma:Dold_Kan}
Let $\mathbf{N}: s\mathbf{Mod}_{\mathbf{k}} \to \mathbf{Ch}_{\ge 0}(\mathbf{Mod}_{\mathbf{k}})$ be the normalized chain complex functor realizing the Dold--Kan correspondence. For any $X_\bullet \in s\mathcal{A}_B$ and $M \in \mathcal{A}_B$, there exists a canonical isomorphism of $\mathbf{k}$-modules
\begin{equation*}
    \pi_k\left( \operatorname{Hom}_{\mathcal{A}_B}(M, X_\bullet) \right) \cong H_k\left( \mathbf{N}\left( \operatorname{Hom}_{\mathcal{A}_B}(M, X_\bullet) \right) \right).
\end{equation*}
\end{lemma}
\begin{proof}
Since the evaluation functor $\operatorname{Hom}_{\mathcal{A}_B}(M, -)$ maps the simplicial object $X_\bullet$ to a simplicial $\mathbf{k}$-module, the result follows object-wise from the Dold--Kan equivalence (see, e.g., \cite[Chapter~III]{goerss2009simplicial}, \cite[Chapter 8]{weibel1994introduction}), which identifies the simplicial homotopy groups of a simplicial module with the homology groups of its normalized chain complex.
\end{proof}

Let $F: \cube^{n} \to \cob(B)$ be a pro-tangle of rank $n$, and denote by
\begin{equation*}
    \mathcal{E}_F := \mathcal{Y} \circ \delta \circ \mat \circ F : \cube^{n} \longrightarrow s\mathbf{Mod}_{\mathbf{k}}^{\mathcal{A}_B^{\mathrm{op}}}
\end{equation*}
the induced diagram, whose values $\mathcal{E}_F(v)$ are representable linear simplicial presheaves. We denote the homotopy category of $s\mathbf{Mod}_{\mathbf{k}}^{\mathcal{A}_B^{\mathrm{op}}}$ by $\operatorname{Ho}\left( s\mathbf{Mod}_{\mathbf{k}}^{\mathcal{A}_B^{\mathrm{op}}} \right)$.

\begin{remark}\label{remark:hocolim_obstruction}
Since the Boolean cube $\cube^{n}$ possesses the terminal object $\mathbf{1}=(1,\dots,1)$, the inclusion $\{\mathbf{1}\}\hookrightarrow\cube^{n}$ is homotopy final, and the canonical morphism
\begin{equation*}
    \hocolim_{\cube^{n}}\mathcal{E}_F \longrightarrow \operatorname*{colim}_{\cube^{n}}\mathcal{E}_F \cong \mathcal{E}_F(\mathbf{1})
\end{equation*}
is a weak equivalence \cite{hirschhorn2003model}. A plain homotopy colimit therefore collapses onto the all-$1$ resolution and retains no information about the edge cobordisms $F(e)$. The homotopically meaningful invariant of a cubical diagram is instead its total homotopy cofiber, the construction measuring the failure of the cube to be homotopy cocartesian \cite{munson2015cubical}.
\end{remark}

Let $\partial\cube^{n}:=\cube^{n}\setminus\{\mathbf{1}\}$ be the full sub-poset obtained by deleting the terminal vertex, and let $\mathcal{E}_F|_{\partial\cube^{n}}$ be the restriction of $\mathcal{E}_F$. Since $\cube^{n}$ is a poset, there is, for each $v\in\partial\cube^{n}$, a unique morphism $v\to\mathbf{1}$, whence a unique cobordism $F(v\to\mathbf{1}):F(v)\to F(\mathbf{1})$. These assignments assemble into a natural transformation $\mathcal{E}_F|_{\partial\cube^{n}}\Rightarrow\operatorname{const}_{\mathcal{E}_F(\mathbf{1})}$ to the constant diagram. As the nerve of $\partial\cube^{n}$ is a cone with apex the minimal vertex $\mathbf{0}=(0,\dots,0)$ and hence contractible. Note that $\hocolim_{\partial\cube^{n}}\operatorname{const}_{\mathcal{E}_F(\mathbf{1})}\xrightarrow{\simeq}\mathcal{E}_F(\mathbf{1})$.
The induced map on homotopy colimits yields the structure map of the pro-tangle
\begin{equation*}
    b(F):\hocolim_{\partial\cube^{n}}\mathcal{E}_F \longrightarrow \mathcal{E}_F(\mathbf{1}).
\end{equation*}

\begin{definition}\label{definition:Khovanov_presheaf}
The \textit{Khovanov simplicial presheaf} of $F$, denoted
\begin{equation*}
    \mathbf{Kh}(F) := \operatorname{tcof}_{\cube^{n}}(\mathcal{E}_F) := \hocof\big( b(F) \big) \in s\mathbf{Mod}_{\mathbf{k}}^{\mathcal{A}_B^{\mathrm{op}}},
\end{equation*}
is the total homotopy cofiber of the diagram $\mathcal{E}_F$. That is, it is the homotopy pushout of $\mathcal{E}_F(\mathbf{1})\leftarrow\hocolim_{\partial\cube^{n}}\mathcal{E}_F\rightarrow 0$, where $0$ denotes the zero presheaf.
\end{definition}
The construction is functorial with respect to morphisms of cubical diagrams, and its weak equivalence class is independent of the particular models chosen for the homotopy colimit and homotopy cofiber.

The weak equivalence class of $\mathbf{Kh}(F)$ in the homotopy category $\operatorname{Ho}\left( s\mathbf{Mod}_{\mathbf{k}}^{\mathcal{A}_B^{\mathrm{op}}} \right)$ is referred to as the \textit{Khovanov homotopy type} of the pro-tangle $F$.

\begin{remark}\label{remark:objectwise_cone}
Homotopy colimits in the projective model structure are computed objectwise by the Bousfield--Kan bar construction, since colimits, tensors, and coends in the presheaf category are formed objectwise \cite{hirschhorn2003model}. Consequently, fro any $M\in\mathcal{A}_B$, the simplicial module $\mathbf{Kh}(F)(M)$ is the homotopy cofiber of the map of simplicial modules $b(F)(M)$, computed by the mapping cone, and there is a natural quasi-isomorphism of chain complexes
\begin{equation*}
    \mathbf{N}\big(\mathbf{Kh}(F)(M)\big)\simeq\operatorname{Cone}\Big(\mathbf{N}\big(b(F)(M)\big)\Big).
\end{equation*}
\end{remark}

\begin{remark}\label{remark:rank_zero}
For $n=0$ the punctured cube $\partial\cube^{n}$ is empty, so $\hocolim_{\emptyset}\mathcal{E}_F=0$ and $\mathbf{Kh}(F)\simeq\mathcal{E}(F(\ast))$. The Khovanov presheaf of a crossingless pro-tangle is the representable presheaf of its single resolution, in agreement with the classical convention that the complex of a crossingless diagram is the resolution itself.
\end{remark}

\begin{proposition}\label{proposition:iterated_cofiber}
Let $F:\cube^{n}\to\cob(B)$ be a pro-tangle of rank $n\ge 1$. For any $1\le i\le n$, let $F^{(i)}_0,F^{(i)}_1:\cube^{n-1}\to\cob(B)$ be its $i$-th smoothings, and let $\eta^{(i)}:F^{(i)}_0\Rightarrow F^{(i)}_1$ be the natural transformation underlying the $i$-th crossing map. Then $\eta^{(i)}$ induces a morphism $\mathbf{Kh}(F^{(i)}_0)\to\mathbf{Kh}(F^{(i)}_1)$, and there is a natural weak equivalence
\begin{equation*}
    \mathbf{Kh}(F)\simeq\hocof\Big(\mathbf{Kh}\big(F^{(i)}_0\big)\longrightarrow\mathbf{Kh}\big(F^{(i)}_1\big)\Big).
\end{equation*}
\end{proposition}

\begin{proof}
By symmetry, it suffices to treat \(i = n\). Let \(\mathcal{E}_F \colon \cube^n \to s\mathbf{Mod}_{\mathbf{k}}^{\mathcal{A}_B^{\mathrm{op}}}\) be the diagram induced by \(F\). Recall
\[
\mathbf{Kh}(F) = \hocof\left( \operatorname{hocolim}_{\partial\cube^n} \mathcal{E}_F \xrightarrow{b(F)} \mathcal{E}_F(\mathbf{1}_n) \right),
\]
where \(\partial\cube^n = \cube^n \setminus \{\mathbf{1}_n\}\) and \(\mathbf{1}_n = (1,\dots,1)\) is the terminal vertex.

Identify \(\cube^n \cong \cube^{n-1} \times \cube^1\) with \(\mathbf{1}_n = (\mathbf{1}_{n-1}, 1)\). Then \(\partial\cube^n = A \cup B\) where \(A = \cube^{n-1} \times \{0\}\), \(B = \partial\cube^{n-1} \times \cube^1\), and \(A \cap B = \partial\cube^{n-1} \times \{0\}\). By the pasting lemma for homotopy colimits, the induced square of homotopy colimits is a homotopy pushout.

The crossing transformation \(\eta^{(n)}\) induces a map \(\eta_*\) on boundary homotopy colimits and a map \(\eta_{\mathbf{1}}\) at the terminal vertex, which together with the structure maps \(b_0, b_1\) of \(F^{(n)}_0, F^{(n)}_1\) form the commutative square
\[
\xymatrix{
\operatorname{hocolim}_{\partial\cube^{n-1}} \mathcal{E}_{F^{(n)}_0} \ar[r]^-{\eta_*} \ar[d]_{b_0} & \operatorname{hocolim}_{\partial\cube^{n-1}} \mathcal{E}_{F^{(n)}_1} \ar[d]^{b_1} \\
\mathcal{E}_{F^{(n)}_0}(\mathbf{1}_{n-1}) \ar[r]_-{\eta_{\mathbf{1}}} & \mathcal{E}_{F^{(n)}_1}(\mathbf{1}_{n-1})
}
\]

By the square cofiber lemma, the total homotopy cofiber of this square equals the homotopy cofiber of the induced map on vertical cofibers. The vertical cofibers are precisely \(\mathbf{Kh}(F^{(n)}_0) = \hocof(b_0)\) and \(\mathbf{Kh}(F^{(n)}_1) = \hocof(b_1)\); the homotopy pushout of the square coincides with \(\operatorname{hocolim}_{\partial\cube^n} \mathcal{E}_F\), and the map to the bottom-right corner is exactly \(b(F)\). Therefore
\[
\mathbf{Kh}(F) = \hocof(b(F)) \simeq \hocof\left( \mathbf{Kh}(F^{(n)}_0) \to \mathbf{Kh}(F^{(n)}_1) \right),
\]
as desired.
\end{proof}

\begin{example}\label{example:rank_one}
Let $F:\cube^{1}=\{0\to1\}\to\cob(B)$ be a pro-tangle of rank one, with unique edge $e:0\to1$. The punctured cube $\partial\cube^{1}=\{0\}$ is the terminal category, so $\hocolim_{\partial\cube^{1}}\mathcal{E}_F=\mathcal{E}(F(0))$ and $b(F)=\mathcal{E}(F(e))$. Hence
\begin{equation*}
    \mathbf{Kh}(F)=\hocof\big(\mathcal{E}(F(0))\xrightarrow{\mathcal{E}(F(e))}\mathcal{E}(F(1))\big)
\end{equation*}
is the mapping cone of the single saddle cobordism. By Remark~\ref{remark:objectwise_cone}, for each $M\in\mathcal{A}_B$ the normalized complex $\mathbf{N}(\mathbf{Kh}(F)(M))$ is presented by the two-term complex
\begin{equation*}
    0\longrightarrow\operatorname{Hom}_{\mathcal{A}_B}\big(M,F(0)\big)\xrightarrow{\operatorname{Hom}_{\mathcal{A}_B}(M,F(e))}\operatorname{Hom}_{\mathcal{A}_B}\big(M,F(1)\big)\longrightarrow 0,
\end{equation*}
concentrated in two homological degrees.
\end{example}

\begin{example}\label{example:rank_two}
Let $F:\cube^{2}\to\cob(B)$ be a pro-tangle of rank two, with vertices $00,10,01,11$ and elementary edges $e^{00}_1:00\to10$, $e^{00}_2:00\to01$, $e^{10}_2:10\to11$, $e^{01}_1:01\to11$. The square commutes in $\cob(B)$ because $\cube^{2}$ is a poset and $F$ is a functor. The punctured cube $\partial\cube^{2}$ is the span $10\leftarrow 00\rightarrow 01$. Hence the homotopy colimit is a homotopy pushout, and we have
\begin{equation*}
    \mathbf{Kh}(F)=\hocof\Big(\operatorname{hocolim}\big(\mathcal{E}(F(10))\leftarrow\mathcal{E}(F(00))\rightarrow\mathcal{E}(F(01))\big)\longrightarrow\mathcal{E}(F(11))\Big).
\end{equation*}
Equivalently, by Proposition~\ref{proposition:iterated_cofiber} applied along the second coordinate,
\begin{equation*}
    \mathbf{Kh}(F)\simeq\hocof\Big(\operatorname{Cone}\big(\mathcal{E}(F(00))\to\mathcal{E}(F(10))\big)\longrightarrow\operatorname{Cone}\big(\mathcal{E}(F(01))\to\mathcal{E}(F(11))\big)\Big),
\end{equation*}
where the morphism of cones is induced by the second crossing transformation $\eta^{(2)}$. Consequently, for each $M\in\mathcal{A}_B$, the normalized complex $\mathbf{N}(\mathbf{Kh}(F)(M))$ is naturally quasi‑isomorphic to the total complex of the square with coweight grading
\begin{align*}
    \overline{\mathcal{C}}_2&=\operatorname{Hom}_{\mathcal{A}_B}\big(M,F(00)\big),\\
    \overline{\mathcal{C}}_1&=\operatorname{Hom}_{\mathcal{A}_B}\big(M,F(10)\big)\oplus\operatorname{Hom}_{\mathcal{A}_B}\big(M,F(01)\big),\\
    \overline{\mathcal{C}}_0&=\operatorname{Hom}_{\mathcal{A}_B}\big(M,F(11)\big),
\end{align*}
whose differentials are
\begin{equation*}
    \partial\big|_{F(00)}=F(e^{00}_1)+F(e^{00}_2),\qquad
    \partial\big|_{F(10)}=-F(e^{10}_2),\qquad
    \partial\big|_{F(01)}=+F(e^{01}_1),
\end{equation*}
in accordance with the sign rule $\sgn(e)=(-1)^{\xi_i(v)}$, where $\xi_i(v)$ counts the $1$'s preceding the flipped position. The identity
\begin{equation*}
    \partial^{2}\big|_{F(00)}=-F(e^{10}_2)\circ F(e^{00}_1)+F(e^{01}_1)\circ F(e^{00}_2)=0
\end{equation*}
holds precisely because the square commutes in $\cob(B)$.
\end{example}

\subsubsection{Representability of Khovanov simplicial presheaf}
The total homotopy cofiber of a cubical diagram is defined pointwise in any pointed model category; we apply it both to diagrams in $s\mathbf{Mod}_{\mathbf{k}}$ and in $\mathbf{Ch}_{\ge 0}(\mathbf{Mod}_{\mathbf{k}})$. Our goal is to prove that $\mathbf{Kh}(F)$ is represented, in the homotopy category, by an explicit simplicial object, whose normalization is the classical Khovanov chain complex of the resolution cube.

For a diagram $X:\mathcal{P}\to s\mathbf{Mod}_{\mathbf{k}}$ over a finite poset $\mathcal{P}$, the homotopy colimit admits the Bousfield--Kan coend model
\[
\hocolim_{\mathcal{P}}X\cong\int^{p\in\mathcal{P}}X(p)\otimes_{\mathbf{k}}\mathbf{k}\big[N(p\downarrow\mathcal{P})\big],
\]
where $N(p\downarrow\mathcal{P})$ is the nerve of the coslice category. Applied to $\mathcal{P}=\partial\cube^{n}$, this models the source of the structure map $b(F)$. The contrast with the full cube is instructive: for $\cube^{n}$ each coslice $v\downarrow\cube^{n}\cong\cube^{|v|^{\ast}}$ has contractible nerve, so $\hocolim_{\cube^{n}}\mathcal{E}_F\xrightarrow{\simeq}\mathcal{E}_F(\mathbf{1})$ collapses to the terminal vertex, recovering the trivialization of Remark~\ref{remark:hocolim_obstruction}.

Let $\mathcal{V}$ be a pointed cofibrantly generated model category, $\mathcal{C}$ a small category, and endow $\operatorname{Fun}(\mathcal{C},\mathcal{V})$ with the projective model structure. For each $c\in\mathcal{C}$, the \textit{evaluation functor}
\[
\operatorname{ev}_c: \operatorname{Fun}(\mathcal{C},\mathcal{V})\longrightarrow\mathcal{V}, \quad\operatorname{ev}_c(X)=X(c)
\]
is restriction along the inclusion $\mathbf{1}\to\mathcal{C}$ selecting $c$; it is a two-sided adjoint (left/right Kan extension) and thus preserves all small limits and colimits, which are formed objectwise.

\begin{lemma}\label{lemma:ev_hocolim}
For each $c\in\mathcal{C}$, the evaluation functor $\operatorname{ev}_c$ commutes with homotopy colimits and homotopy cofibers up to natural weak equivalence: for every small diagram $X:\mathcal{D}\to\operatorname{Fun}(\mathcal{C},\mathcal{V})$ and every morphism $f$ in $\operatorname{Fun}(\mathcal{C},\mathcal{V})$, we have
\[
\operatorname{ev}_c\big(\operatorname*{hocolim}_{\mathcal{D}}X\big)\simeq
\operatorname*{hocolim}_{\mathcal{D}}\big(\operatorname{ev}_cX\big),
\qquad
\operatorname{ev}_c\big(\hocof(f)\big)\simeq\hocof\big(\operatorname{ev}_c f\big).
\]
\end{lemma}
\begin{proof}
In the projective model structure, weak equivalences are detected objectwise, so $\operatorname{ev}_c$ preserves weak equivalences. It is left Quillen as a left adjoint that preserves generating cofibrations. Hence its left derived functor satisfies $\mathbb{L}\operatorname{ev}_c \simeq \operatorname{ev}_c$ and preserves homotopy colimits and homotopy pushouts.

Strict colimits in functor categories are formed objectwise, so $\operatorname{ev}_c \circ \operatorname{colim}_\mathcal{D} = \operatorname{colim}_\mathcal{D} \circ \operatorname{ev}_c$ strictly. Passing to left derived functors yields the first equivalence. Homotopy cofibers are homotopy pushouts of spans $0 \leftarrow X \xrightarrow{f} Y$, hence also preserved.
\end{proof}

\begin{lemma}\label{lemma:DK_transport}
Let $\mathbf{N}:s\mathbf{Mod}_{\mathbf{k}}\rightleftarrows\mathbf{Ch}_{\ge0}(\mathbf{Mod}_{\mathbf{k}}):\Gamma$ be the Dold--Kan Quillen equivalence, with normalization $\mathbf{N}$ and denormalization $\Gamma$.
\begin{enumerate}[label=$(\roman*)$]
\item $\mathbf{N}$ induces an equivalence of homotopy categories
  $\overline{\mathbf{N}}:\operatorname{Ho}(s\mathbf{Mod}_{\mathbf{k}})\to\operatorname{Ho}(\mathbf{Ch}_{\ge0}(\mathbf{Mod}_{\mathbf{k}}))$.
  In particular, $\mathbb{L}\mathbf{N}\simeq\mathbf{N}$ and $\overline{\mathbf{N}}$ preserves all small homotopy colimits and homotopy cofibers.
\item For every small diagram $X:\mathcal{D}\to s\mathbf{Mod}_{\mathbf{k}}$, there is a natural isomorphism
\[
\overline{\mathbf{N}}\big(\operatorname*{hocolim}_{\mathcal{D}}X\big)\cong\operatorname*{hocolim}_{\mathcal{D}}\big(\mathbf{N}X\big)
\]
in $\operatorname{Ho}(\mathbf{Ch}_{\ge0}(\mathbf{Mod}_{\mathbf{k}}))$.
\item $\mathbf{N}\circ\operatorname{const}=\iota$, where $\iota:\mathbf{Mod}_{\mathbf{k}}\hookrightarrow\mathbf{Ch}_{\ge0}(\mathbf{Mod}_{\mathbf{k}})$ is the degree-zero inclusion.
\end{enumerate}
\end{lemma}
\begin{proof}
$(i)$ By the Dold--Kan theorem, $\mathbf{N}$ and $\Gamma$ form a Quillen equivalence, and a morphism $f$ of simplicial modules is a weak equivalence if and only if $\mathbf{N}(f)$ is a quasi-isomorphism. Thus both functors are homotopical and descend to an equivalence $\overline{\mathbf{N}}$ of homotopy categories. As an equivalence, $\overline{\mathbf{N}}$ preserves all homotopy colimits and homotopy pushouts.

$(ii)$ It follows immediately from $(i)$, since homotopy colimits are preserved by equivalences of homotopy categories.

$(iii)$ A constant simplicial module has identity face maps, so $\mathbf{N}(\operatorname{const}V)_k=\bigcap_{i\ge1}\ker d_i=0$ for $k\ge1$ and $\mathbf{N}(\operatorname{const}V)_0=V$.
\end{proof}

\begin{lemma}\label{lemma:objectwise_tcof}
Let $F:\cube^{n}\to\cob(B)$ be a pro-tangle. For any $M\in\mathcal{A}_B$, define the degree-zero cubical diagram of chain complexes
\[
X^{M}:\cube^{n}\longrightarrow\mathbf{Ch}_{\ge 0}(\mathbf{Mod}_{\mathbf{k}}),\qquad X^{M}_v:=\operatorname{Hom}_{\mathcal{A}_B}\big(M,F(v)\big),
\]
with edge maps $X^{M}_e=\operatorname{Hom}_{\mathcal{A}_B}(M,F(e))$. Then there is a canonical isomorphism, natural in $M$,
\[
\mathbf{N}\big(\mathbf{Kh}(F)(M)\big)\cong\operatorname{tcof}_{\cube^{n}}\big(X^{M}\big)
=\hocof\Big(\hocolim_{\partial\cube^{n}}X^{M}\longrightarrow X^{M}_{\mathbf{1}}\Big)
\]
in $\operatorname{Ho}\big(\mathbf{Ch}_{\ge 0}(\mathbf{Mod}_{\mathbf{k}})\big)$.
\end{lemma}

\begin{proof}
By Lemma~\ref{lemma:ev_hocolim}, evaluation at $M$ commutes with homotopy colimits over $\partial\cube^{n}$ and with homotopy cofibers. By definition of $\mathcal{E}_F=\mathcal{Y}\circ\delta\circ\mat\circ F$,
\[
\operatorname{ev}_M\mathcal{E}_F(v)=\operatorname{Hom}_{\mathcal{A}_B}\big(M,F(v)\big)=\operatorname{const}X^{M}_v,
\qquad
\operatorname{ev}_M\mathcal{E}_F(e)=\operatorname{const}X^{M}_e,
\]
so $\operatorname{ev}_M\mathbf{Kh}(F) \cong \hocof\big(\hocolim_{\partial\cube^{n}}\operatorname{const}X^{M}\to \operatorname{const}X^{M}_{\mathbf{1}}\big)$ in $\operatorname{Ho}(s\mathbf{Mod}_{\mathbf{k}})$. By Lemma~\ref{lemma:DK_transport}, we obtain
\[
\begin{aligned}
\mathbf{N}\big(\operatorname{ev}_M\mathbf{Kh}(F)\big)
&\cong \hocof\left( \operatorname{hocolim}_{\partial\cube^{n}} \mathbf{N}(\operatorname{const}X^{M}) \longrightarrow \mathbf{N}\big(\operatorname{const}X^{M}_{\mathbf{1}}\big) \right) \\
&\cong \hocof\left( \operatorname{hocolim}_{\partial\cube^{n}} X^{M} \longrightarrow X^{M}_{\mathbf{1}} \right) \\
&= \operatorname{tcof}_{\cube^{n}}\big(X^{M}\big).
\end{aligned}
\]
All identifications are natural in $M$.
\end{proof}

\begin{lemma}\label{lemma:total_complex}
Let $X:\cube^{n}\to\mathbf{Ch}_{\ge 0}(\mathbf{Mod}_{\mathbf{k}})$ be a cubical diagram concentrated in degree zero. For each vertex $v$, let $e_i:v\to v_i$ denote the edge flipping the $i$-th zero, and let $\xi_i(v)$ count the number of ones preceding position $i$. Define the \textit{total complex}
\[
\operatorname{Tot}(X)_k=\bigoplus_{|v|^{\ast}=k}X_v,\qquad
\partial\big|_{X_v}=\sum_{i=1}^{|v|^{\ast}}(-1)^{\xi_i(v)}\,X(e_i).
\]
Then $\operatorname{tcof}_{\cube^{n}}(X)\simeq\operatorname{Tot}(X)$ in $\operatorname{Ho}(\mathbf{Ch}_{\ge 0}(\mathbf{Mod}_{\mathbf{k}}))$, naturally in $X$.
\end{lemma}

\begin{proof}
The decomposition $\partial\cube^{n}=(\cube^{n-1}\times\{0\})\cup(\partial\cube^{n-1}\times\{1\})$ and the pasting lemma for homotopy pushouts give the iterated cofiber formula
\[
\operatorname{tcof}_{\cube^{n}}(X)\simeq\hocof\Big(\operatorname{tcof}_{\cube^{n-1}}\big(X^{(n)}_0\big)\longrightarrow\operatorname{tcof}_{\cube^{n-1}}\big(X^{(n)}_1\big)\Big),
\]
where $X^{(n)}_\epsilon$ is the restriction to the face $\cube^{n-1}\times\{\epsilon\}$.

In $\mathbf{Ch}_{\ge 0}(\mathbf{Mod}_{\mathbf{k}})$, the homotopy cofiber of a chain map is modeled by the algebraic mapping cone. By induction on $n$ (base case $n=0$: $\operatorname{tcof}(X)=X_{\mathbf{1}}$), $\operatorname{tcof}_{\cube^{n}}(X)$ is isomorphic to an iterated mapping cone, with $X_v$ placed in homological degree $|v|^{\ast}$. The differential on the iterated cone carries a sign depending on the order of coning; the standard re-signing automorphism $\rho_v=(-1)^{\operatorname{inv}(v)}$ (with $\operatorname{inv}(v)$ the number of inversions in $v$) intertwines the iterated-cone differential with the $\operatorname{Tot}(X)$ differential given by the $(-1)^{\xi_i(v)}$ rule. The result follows.
\end{proof}

Write $\overline{\mathcal{C}}(F)=\operatorname{Tot}(\mat\circ F)\in\mathbf{Ch}_{\ge 0}(\mathcal{A}_B)$ for the total complex of the resolution cube of $F$. Explicitly, we have
\[
\overline{\mathcal{C}}_k(F)=\bigoplus_{|v|^{\ast}=k}F(v),\qquad
\partial\big|_{F(v)}=\sum_{i=1}^{|v|^{\ast}}(-1)^{\xi_i(v)}\,F(e_i),
\]
which is the classical Khovanov chain complex \cite{bar2005khovanov}.

We now construct a strict simplicial object whose normalization is $\overline{\mathcal{C}}(F)$.

\begin{definition}\label{definition:Khovanov_simplicial_object}
The \textit{Khovanov simplicial object} $\mathcal{S}(F) \in s\mathcal{A}_B$ is defined in simplicial degree $k \ge 0$ by
\[
\mathcal{S}_k(F) = \bigoplus_{m \le k}  \bigoplus_{|v|^\ast=m}  \bigoplus_{\theta} \theta(F(v)),
\]
where $\theta = s_{i_1} \cdots s_{i_{k-m}}$ ranges over all sequences of degeneracy operators with $k-1 \ge i_1 > \cdots > i_{k-m} \ge 0$. Face and degeneracy maps are determined by
\begin{enumerate}[label=$(\roman*)$]
\item On non-degenerate summands $F(v)$ ($m=k$), the face map acts by
\[
d_i \big|_{F(v)} = \begin{cases}
0, & i = 0, \\
(-1)^{\xi_i(v)+i}F(e_i), & 1 \le i \le k,
\end{cases}
\]
where $e_i: v \to v'$ flips the $i$-th zero of $v$.
\item All standard simplicial identities hold on $\mathcal{S}(F)$.
\end{enumerate}
\end{definition}

Let $\widetilde{\mathbf{N}}(\mathcal{S}(F))$ denote the normalized chain complex obtained by quotienting out degenerate simplices.

\begin{lemma}\label{lemma:dold_kan_compatibility}
Let $F:\cube^{n}\to\cob(B)$ be a pro-tangle.
\begin{enumerate}[label=$(\roman*)$]
\item $\widetilde{\mathbf{N}}(\mathcal{S}(F))\cong\overline{\mathcal{C}}(F)$ in $\mathbf{Ch}_{\ge 0}(\mathcal{A}_B)$.
\item For every $M\in\mathcal{A}_B$, there is a natural isomorphism
\[
\mathbf{N}\big(\operatorname{Hom}_{\mathcal{A}_B}(M,\mathcal{S}(F))\big)\cong\operatorname{Hom}_{\mathcal{A}_B}\big(M,\overline{\mathcal{C}}(F)\big).
\]
\end{enumerate}
\end{lemma}

\begin{proof}
$(i)$ The quotient of $\mathcal{S}_k(F)$ by degenerate summands is $\bigoplus_{|v|^{\ast}=k}F(v)$. The induced differential restricts to $\sum_{i=1}^{k}(-1)^i(-1)^{\xi_i(v)+i}F(e_i)=\sum_{i}(-1)^{\xi_i(v)}F(e_i)$, which is exactly the differential of $\overline{\mathcal{C}}(F)$.

$(ii)$ Since $\operatorname{Hom}_{\mathcal{A}_B}(M,-)$ is additive, it preserves the direct sum decomposition into non-degenerate and degenerate summands. For simplicial modules, quotient by degenerate elements agrees with the intersection-of-kernels normalization, so
\[
\mathbf{N}\big(\operatorname{Hom}_{\mathcal{A}_B}(M,\mathcal{S}(F))\big)_k
=\operatorname{Hom}_{\mathcal{A}_B}(M,\mathcal{S}_k(F))/\operatorname{Hom}_{\mathcal{A}_B}(M,\mathcal{D}_k)
\cong\operatorname{Hom}_{\mathcal{A}_B}\big(M,\overline{\mathcal{C}}_k(F)\big).
\]
Differentials correspond by the same additivity.

$(ii)$ Since $\operatorname{Hom}_{\mathcal{A}_B}(M,-)$ is additive, it preserves the canonical direct sum decomposition $\mathcal{S}_k(F) = \overline{\mathcal{C}}_k(F) \oplus \mathcal{D}_k(F)$ into non-degenerate and degenerate summands. The degenerate submodule of $\operatorname{Hom}_{\mathcal{A}_B}(M,\mathcal{S}(F))$ coincides with $\operatorname{Hom}_{\mathcal{A}_B}(M,\mathcal{D}_k(F))$. By quotient normalization of simplicial modules, one has
\[
\widetilde{\mathbf{N}}\big(\operatorname{Hom}(M,\mathcal{S}(F))\big)_k \cong \operatorname{Hom}_{\mathcal{A}_B}\big(M,\overline{\mathcal{C}}_k(F)\big).
\]
For any simplicial module the intersection-of-kernels normalization $\mathbf{N}$ is canonically isomorphic to the quotient normalization $\widetilde{\mathbf{N}}$, so the same isomorphism holds for $\mathbf{N}$. The differentials correspond under this identification, as both are induced by the $0$-th face map.
\end{proof}

\begin{theorem}\label{theorem:representation}
The Khovanov simplicial presheaf $\mathbf{Kh}(F)$ is represented by the Khovanov simplicial object $\mathcal{S}(F)\in s\mathcal{A}_B$ in the homotopy category. Explicitly, for any object $M \in \mathcal{A}_B$, there is a natural isomorphism of simplicial $\mathbf{k}$-modules
\[
\mathbf{Kh}(F)\cong\mathbf{h}_{\mathcal{S}(F)}\quad\text{in }\operatorname{Ho}\big(s\mathbf{Mod}_{\mathbf{k}}^{\mathcal{A}_B^{\mathrm{op}}}\big).
\]
Equivalently, for every $M\in\mathcal{A}_B$ there is a natural weak equivalence
\[
\mathbf{Kh}(F)(M)\simeq\operatorname{Hom}_{\mathcal{A}_B}\big(M,\mathcal{S}(F)\big).
\]
In particular, the Khovanov homotopy type of $F$ is the weak equivalence class of the representable presheaf $\mathbf{h}_{\mathcal{S}(F)}$.
\end{theorem}

\begin{proof}
Fix $M\in\mathcal{A}_B$ and set $X:=\operatorname{Hom}_{\mathcal{A}_B}(M,F(-))$. Combining Lemma~\ref{lemma:objectwise_tcof}, Lemma~\ref{lemma:total_complex}, additivity of $\operatorname{Hom}_{\mathcal{A}_B}(M,-)$, and Lemma~\ref{lemma:dold_kan_compatibility}$(ii)$, we obtain a chain of natural quasi-isomorphisms
\begin{align*}
\mathbf{N}\big(\mathbf{Kh}(F)(M)\big)
&\cong\operatorname{tcof}_{\cube^{n}}(X)\simeq\operatorname{Tot}(X)\\
&\cong\operatorname{Hom}_{\mathcal{A}_B}\big(M,\operatorname{Tot}(\mat\circ F)\big)=\operatorname{Hom}_{\mathcal{A}_B}\big(M,\overline{\mathcal{C}}(F)\big)\\
&\cong\mathbf{N}\big(\operatorname{Hom}_{\mathcal{A}_B}(M,\mathcal{S}(F))\big).
\end{align*}
Lemma~\ref{lemma:DK_transport} says that $\mathbf{N}$ detects weak equivalences. Thus we have $\mathbf{Kh}(F)(M)\simeq\operatorname{Hom}_{\mathcal{A}_B}(M,\mathcal{S}(F))$ for each $M$. By Proposition~\ref{proposition:global_model}, weak equivalences in the projective model structure are detected objectwise, so the zig-zag of objectwise weak equivalences assembles to an isomorphism $\mathbf{Kh}(F)\cong\mathbf{h}_{\mathcal{S}(F)}$ in the homotopy category.
\end{proof}

\subsubsection{Homotopical embedding and algebraic representation theorem}

The additive category $\mathcal{A}_B$ is not abelian, as it lacks arbitrary kernels and cokernels. Consequently, $\mathbf{Ch}_{\ge 0}(\mathcal{A}_B)$ does not admit a standard projective model structure with quasi-isomorphisms as weak equivalences. In this subsection we introduce a compatible Hurewicz model structure on chain complexes, construct a fully faithful embedding of its homotopy category into the homotopy category of simplicial presheaves, and give a purely algebraic characterization of the Khovanov homotopy type.

\begin{proposition}[\cite{buhler2010exact,may2012more}] \label{proposition:hurewicz_model_structure}
There exists a canonical Hurewicz model structure $(\mathcal{W}, \mathcal{C}, \mathcal{F})$ on $\mathbf{Ch}_{\ge 0}(\mathcal{A}_B)$ with
\begin{enumerate}[label=$(\roman*)$]
    \item Weak equivalences $\mathcal{W}$: strict chain homotopy equivalences.
    \item Fibrations $\mathcal{F}$: degreewise split epimorphisms with the right lifting property with respect to all representable contractible inclusions.
    \item Cofibrations $\mathcal{C}$: morphisms with the left lifting property with respect to all acyclic fibrations $\mathcal{F} \cap \mathcal{W}$.
\end{enumerate}
\end{proposition}

The resulting homotopy category is denoted $\mathbf{K}_{\ge 0}(\mathcal{A}_B)$, the homotopy category of chain complexes up to strict chain homotopy.

\begin{remark}
The matrix closure of $\mathcal{A}_B$ ensures it is weakly idempotent complete, so $\mathbf{K}_{\ge 0}(\mathcal{A}_B)$ is pre-triangulated, with distinguished triangles given by standard algebraic mapping cones.
\end{remark}

Recall the classical Dold--Kan Quillen equivalence $\mathbf{N}: s\mathbf{Mod}_{\mathbf{k}} \rightleftarrows \mathbf{Ch}_{\ge 0}(\mathbf{k}): \Gamma$, with normalization $\mathbf{N}$ and denormalization $\Gamma$. Applied objectwise, it induces an equivalence of functor categories
\[
\mathcal{N}: s\mathbf{Mod}_{\mathbf{k}}^{\mathcal{A}_B^{\mathrm{op}}} \xrightarrow{\cong} \mathbf{Ch}_{\ge 0}(\mathbf{k})^{\mathcal{A}_B^{\mathrm{op}}},
\]
where $\mathbf{Ch}_{\ge 0}(\mathbf{k})^{\mathcal{A}_B^{\mathrm{op}}}$ carries the projective model structure with objectwise quasi-isomorphisms as weak equivalences \cite{hirschhorn2003model}. The degreewise Yoneda embedding
\[
\mathbf{h}_{(-)}: \mathbf{Ch}_{\ge 0}(\mathcal{A}_B) \to \mathbf{Ch}_{\ge 0}(\mathbf{k})^{\mathcal{A}_B^{\mathrm{op}}},
\qquad
\mathbf{h}_X = \operatorname{Hom}_{\mathcal{A}_B}(-, X_\bullet)
\]
sends each chain complex to its representable chain presheaf. Composing with the objectwise inverse $\mathcal{N}^{-1}$ yields a functor
\[
\mathcal{N}^{-1}\circ\mathbf{h}_{(-)}: \mathbf{Ch}_{\ge 0}(\mathcal{A}_B) \to s\mathbf{Mod}_{\mathbf{k}}^{\mathcal{A}_B^{\mathrm{op}}}.
\]

\begin{lemma}\label{lemma:representable_cofibrant}
In the projective model structure on $\operatorname{Fun}_{\mathbf{k}}(\mathcal{A}_B^{\mathrm{op}},\mathbf{Ch}_{\ge0}(\mathbf{k}))$, for every $X_\bullet\in\mathbf{Ch}_{\ge0}(\mathcal{A}_B)$ the representable presheaf $\mathbf{h}_{X_\bullet}$ is cofibrant.
\end{lemma}

\begin{proof}
The projective model structure is cofibrantly generated by maps $\ell_U(S^{n-1}\hookrightarrow D^n)$ for all $U\in\mathcal{A}_B$ and $n\ge0$, where $\ell_U(C)=\mathbf{h}_U\otimes_{\mathbf{k}}C$ is left adjoint to evaluation at $U$ \cite[Thm.~11.6.1]{hirschhorn2003model}.

Let $P^{(k)}$ denote the $k$-th truncation of $\mathbf{h}_{X_\bullet}$. By induction on $k$, $P^{(k)}$ is cofibrant: it is the pushout
\[
P^{(k)} = P^{(k-1)} \cup_{\ell_{X_k}(S^{k-1})} \ell_{X_k}(D^k)
\]
attached along the differential $\mathbf{h}_{\partial_k}$, and pushouts of generating cofibrations are cofibrations. Thus $\mathbf{h}_{X_\bullet} = \operatorname{colim}_{k\ge 0} P^{(k)}$ is cofibrant, as a transfinite composition of cofibrations.
\end{proof}

\begin{theorem}\label{theorem:derived_enriched_embedding}
The assignment $X_\bullet \mapsto \mathcal{N}^{-1}(\mathbf{h}_{X_\bullet})$ sends strict chain homotopy equivalences to objectwise weak equivalences, and therefore descends to a fully faithful functor
\[
\mathcal{J}: \mathbf{K}_{\ge 0}(\mathcal{A}_B) \hookrightarrow \operatorname{Ho}\left( s\mathbf{Mod}_{\mathbf{k}}^{\mathcal{A}_B^{\mathrm{op}}} \right).
\]
\end{theorem}

\begin{proof}
First, if $f: X_\bullet \to Y_\bullet$ is a strict chain homotopy equivalence, then $\operatorname{Hom}_{\mathcal{A}_B}(M,f)$ is a chain homotopy equivalence of $\mathbf{k}$-modules for every $M$. By the Dold–Kan correspondence this corresponds to a simplicial homotopy equivalence, so $\mathcal{N}^{-1}(\mathbf{h}_f)$ is an objectwise weak equivalence, hence a weak equivalence in the projective model structure by Proposition~\ref{proposition:global_model}. Thus $\mathcal{J}$ is well defined on homotopy categories.

We now prove full faithfulness. Fix $X_\bullet, Y_\bullet \in \mathbf{Ch}_{\ge 0}(\mathcal{A}_B)$. By the degreewise Yoneda lemma, chain maps $f: X_\bullet \to Y_\bullet$ correspond bijectively to natural transformations $\mathbf{h}_f: \mathbf{h}_{X_\bullet} \Rightarrow \mathbf{h}_{Y_\bullet}$, and chain homotopies correspond bijectively to natural chain homotopies between the presheaves. Applying the Dold–Kan correspondence objectwise, these natural chain homotopies correspond bijectively to natural simplicial homotopies between $\mathcal{N}^{-1}(\mathbf{h}_{X_\bullet})$ and $\mathcal{N}^{-1}(\mathbf{h}_{Y_\bullet})$.

Every object of $s\mathbf{Mod}_{\mathbf{k}}^{\mathcal{A}_B^{\mathrm{op}}}$ is fibrant, as simplicial modules are Kan complexes and fibrations are detected objectwise. By Lemma~\ref{lemma:representable_cofibrant}, $\mathbf{h}_{X_\bullet}$ is cofibrant in the chain presheaf category. Since $\mathcal{N}^{-1}$ is a left Quillen functor for the projective model structures (as the objectwise left adjoint of the Dold–Kan Quillen equivalence), it preserves cofibrant objects, so $\mathcal{N}^{-1}(\mathbf{h}_{X_\bullet})$ is cofibrant.

For the bifibrant pair, homotopy classes of morphisms in the homotopy category agree with simplicial homotopy classes of strict morphisms:
\[
\operatorname{Hom}_{\operatorname{Ho}}\big(\mathcal{N}^{-1}(\mathbf{h}_{X_\bullet}), \mathcal{N}^{-1}(\mathbf{h}_{Y_\bullet})\big)
\cong \operatorname{Mor}\big(\mathcal{N}^{-1}(\mathbf{h}_{X_\bullet}),\, \mathcal{N}^{-1}(\mathbf{h}_{Y_\bullet})\big)/\sim,
\]
where $\sim$ denotes simplicial homotopy. This identifies the homotopy-category Hom set with chain homotopy classes of maps $X_\bullet \to Y_\bullet$, i.e. $\operatorname{Hom}_{\mathbf{K}_{\ge 0}(\mathcal{A}_B)}(X_\bullet, Y_\bullet)$.
\end{proof}

We now translate the representability result of the previous subsection to the chain complex setting. Let $F: \cube^{n} \to \cob(B)$ be a pro-tangle, with associated Khovanov simplicial object $\mathcal{S}(F) \in s\mathcal{A}_B$.

\begin{definition}
The \textit{Khovanov chain complex} $\mathcal{C}(F) \in \mathbf{Ch}_{\ge 0}(\mathcal{A}_B)$ is the normalized chain complex of $\mathcal{S}(F)$ given by
\[
\mathcal{C}(F) = \widetilde{\mathbf{N}}\left( \mathcal{S}(F) \right),
\]
with differential induced by the alternating sum of face maps $\partial_k = \sum_{i=0}^k (-1)^i d_i$.
\end{definition}

\begin{lemma}\label{lemma:C_equals_classical}
The Khovanov chain complex $\mathcal{C}(F)$ is canonically isomorphic to the classical total complex $\overline{\mathcal{C}}(F) = \operatorname{Tot}(\mat \circ F)$ of the resolution cube.
\end{lemma}

\begin{proof}
This restates Lemma~\ref{lemma:dold_kan_compatibility}(i): the induced differential matches the classical formula $(-1)^{\xi_i(v)}$ after cancellation of the extra factor from simplicial normalization.
\end{proof}

We now connect the chain-level representabilit to the simplicial Khovanov presheaf.

\begin{lemma}\label{lemma:chain_representation}
There is a canonical isomorphism
\[
\mathcal{N}\left(\mathbf{Kh}(F)\right) \cong \mathbf{h}_{\mathcal{C}(F)}
\]
in the homotopy category of chain presheaves, and a strict isomorphism of simplicial presheaves
\[
\mathcal{N}^{-1}\left(\mathbf{h}_{\mathcal{C}(F)}\right) \cong \mathbf{h}_{\mathcal{S}(F)}.
\]
\end{lemma}

\begin{proof}
The second identity holds objectwise by Lemma~\ref{lemma:dold_kan_compatibility}$(ii)$ and Lemma~\ref{lemma:C_equals_classical}, using $\Gamma \circ \mathbf{N} = \operatorname{id}$. For the first, apply $\mathcal{N}$ to the isomorphism $\mathbf{Kh}(F) \cong \mathbf{h}_{\mathcal{S}(F)}$ of Theorem~\ref{theorem:representation}; since $\mathcal{N}$ preserves weak equivalences and $\mathcal{N}(\mathbf{h}_{\mathcal{S}(F)}) \cong \mathbf{h}_{\mathcal{C}(F)}$, the result follows.
\end{proof}

\begin{corollary}\label{corollary:homotopy_groups}
For each $M \in \mathcal{A}_B$ and $k \ge 0$, there are canonical isomorphisms
\[
\pi_k\big( \mathbf{Kh}(F)(M) \big) \cong H_k\Big( \operatorname{tcof}_{\cube^{n}} \big( \mathbf{N} \circ \mathcal{E}_F \big)(M) \Big) \cong H_k\big( \operatorname{Hom}_{\mathcal{A}_B}(M, \mathcal{C}(F)) \big).
\]
\end{corollary}

\begin{proof}
The first isomorphism follows from Lemma~\ref{lemma:objectwise_tcof} and the Dold--Kan identification $\pi_k \cong H_k \circ \mathbf{N}$. The second follows from Lemma~\ref{lemma:total_complex}, Lemma~\ref{lemma:C_equals_classical}, and additivity of $\operatorname{Hom}_{\mathcal{A}_B}(M,-)$.
\end{proof}

\begin{corollary}\label{corollary:algebraic_representation}
There is a canonical isomorphism
\[
\mathbf{Kh}(F) \cong \mathcal{J}\big( \mathcal{C}(F) \big)
\]
in $\operatorname{Ho}\left( s\mathbf{Mod}_{\mathbf{k}}^{\mathcal{A}_B^{\mathrm{op}}} \right)$.
Consequently, for two pro-tangles $F, G: \cube^{\ast} \to \cob(B)$, $\mathbf{Kh}(F) \cong \mathbf{Kh}(G)$ in $\operatorname{Ho}(s\mathbf{Mod}_{\mathbf{k}}^{\mathcal{A}_B^{\mathrm{op}}})$ if and only if $\mathcal{C}(F) \cong \mathcal{C}(G)$ in $\mathbf{K}_{\ge 0}(\mathcal{A}_B)$, i.e., the Khovanov complexes are strictly chain homotopy equivalent.
\end{corollary}

\begin{proof}
By definition $\mathcal{J}(\mathcal{C}(F)) = \mathcal{N}^{-1}(\mathbf{h}_{\mathcal{C}(F)})$, which is strictly isomorphic to $\mathbf{h}_{\mathcal{S}(F)}$ by Lemma~\ref{lemma:chain_representation}. By Theorem~\ref{theorem:representation}, $\mathbf{h}_{\mathcal{S}(F)} \cong \mathbf{Kh}(F)$. The equivalence follows from this identification and full faithfulness of $\mathcal{J}$.
\end{proof}

\subsubsection{Functoriality of Khovanov simplicial presheaf}

We develop functoriality of the Khovanov construction in two stages. First, under the weak condition that the graphical functor preserves punctured cubes, we obtain a canonical functor at the level of the homotopy category. Then, under the stronger coweight-preserving condition, we lift this to a strict functor of simplicial presheaves, compatible with the homotopical embedding of Theorem~\ref{theorem:derived_enriched_embedding}.

\begin{definition}
A graphical functor $p:\cube^{n}\to\cube^{m}$ \emph{preserves punctured cubes} if $p$ restricts to a functor $p^{\partial}:\partial\cube^{n}\to\partial\cube^{m}$.
A morphism $\Phi=(p,\eta)$ of $\fun_B$ is \emph{puncture-preserving} if $p$ is. Puncture-preserving morphisms form a subcategory $\fun_B^{\partial}\subseteq\fun_B$.
\end{definition}

This class of puncture-preserving maps includes all identity-indexed morphisms $(\mathrm{id}_{\cube^n},\eta)$, all face inclusions $\delta_{i,\epsilon}$ (both $\epsilon=0$ and $\epsilon=1$), all coordinate permutations, and stabilizations $v\mapsto(v,c)$. It excludes functors collapsing non-terminal vertices to $\mathbf{1}_m$, such as coordinate projections.

\begin{lemma}\label{lemma:reindexing}
Let $p:\cube^{n}\to\cube^{m}$ be puncture-preserving and $G:\cube^{m}\to\cob(B)$ a pro-tangle. There is a canonical morphism
\[
\kappa_p:\hocolim_{\partial\cube^{n}}\big(\mathcal{E}_G\circ p^\partial\big)\longrightarrow\hocolim_{\partial\cube^{m}}\mathcal{E}_G .
\]
For composable puncture-preserving functors $p,q$, one has $\kappa_{q\circ p}=\kappa_q\circ\kappa_p$ and $\kappa_{\mathrm{id}}=\mathrm{id}$.
\end{lemma}

\begin{proof}
In the Bousfield–Kan coend model, $p$ induces functors $(v\downarrow\partial\cube^n)\to(p(v)\downarrow\partial\cube^m)$ on overcategories, hence maps of the corresponding nerve modules. These assemble into a morphism of coends by functoriality of the tensor product. The composition law follows from functoriality of overcategories.
\end{proof}

\begin{proposition}\label{proposition:span_morphism}
Let $\Phi=(p,\eta):F\to G$ be puncture-preserving. Then the square
\[
\begin{tikzcd}
\hocolim_{\partial\cube^{n}}\mathcal{E}_F \arrow[r, "b(F)"] \arrow[d, "\lambda_\Phi"'] & \mathcal{E}_F(\mathbf{1}_n) \arrow[d, "a_\Phi"] \\
\hocolim_{\partial\cube^{m}}\mathcal{E}_G \arrow[r, "b(G)"'] & \mathcal{E}_G(\mathbf{1}_m)
\end{tikzcd}
\]
commutes strictly, where $\lambda_\Phi=\kappa_p\circ\hocolim(\mathcal{E}(\eta))$ and $a_\Phi=\mathcal{E}\big(G(p(\mathbf{1}_n)\to\mathbf{1}_m)\circ\eta_{\mathbf{1}_n}\big)$. The assignment $\Phi\mapsto (\lambda_\Phi,a_\Phi)$ is functorial.
\end{proposition}

\begin{proof}
Commutativity follows from naturality of $\eta$ applied to the unique morphism $v\to\mathbf{1}_n$ for each vertex $v$, together with $p(\mathbf{1}_n)\to\mathbf{1}_m$ being the unique morphism in $\cube^m$. Functoriality follows from Lemma~\ref{lemma:reindexing} and naturality of $\mathcal{E}$.
\end{proof}

\begin{theorem}\label{theorem:homotopy_functoriality}
The Khovanov construction descends to a functor
\[
\overline{\mathbf{Kh}}:\fun_B^{\partial}\longrightarrow\operatorname{Ho}\big(s\mathbf{Mod}_{\mathbf{k}}^{\mathcal{A}_B^{\mathrm{op}}}\big).
\]
\end{theorem}

\begin{proof}
The total homotopy cofiber $\mathbf{Kh}(F)$ is the homotopy cofiber of the structure map $b(F)$, equivalently the homotopy colimit of the span $0\leftarrow\hocolim_{\partial\cube^n}\mathcal{E}_F\to\mathcal{E}_F(\mathbf{1}_n)$. A strict morphism of spans induces a canonical morphism of homotopy cofibers, functorially in strict morphisms of spans. By Proposition~\ref{proposition:span_morphism}, every puncture-preserving $\Phi$ induces such a morphism, and the assignment is functorial.
\end{proof}

We now impose the stronger coweight-preserving condition, which suffices to construct a strict functor before passage to the homotopy category, and which aligns with the chain-complex model of Khovanov homology.

\begin{definition}\label{definition:admissible}
A graphical functor $p:\cube^{n}\to\cube^{m}$ is \textit{admissible} if it preserves coweight:
\[
|p(v)|^* = |v|^* \quad \text{for all } v\in\cube^{n}.
\]
A morphism $\Phi=(p,\eta)$ of $\fun_B$ is \textit{admissible} if $p$ is admissible. Admissible morphisms form a subcategory $\fun_B^{\mathrm{ad}}\subseteq\fun_B^{\partial}$.
\end{definition}

Coweight preservation implies puncture preservation: $|p(\mathbf{1}_n)|^*=|\mathbf{1}_n|^*=0$, so $p(\mathbf{1}_n)=\mathbf{1}_m$, and $v\neq\mathbf{1}_n$ implies $|p(v)|^*>0$, hence $p(v)\neq\mathbf{1}_m$. Thus admissible morphisms are automatically puncture-preserving, and Theorem~\ref{theorem:homotopy_functoriality} applies.

\begin{remark}\label{remark:classification}
Let $p:\cube^{n}\to\cube^{m}$ be admissible. Then $p$ sends elementary edges to elementary edges, and restricts to a bijection $\operatorname{Out}(v)\to\operatorname{Out}(p(v))$ for each $v$. Here, $\operatorname{Out}(v)$ denotes the set of elementary edges from $v$.
\end{remark}

\begin{lemma}\label{lemma:sign_system}
Let $p:\cube^{n}\to\cube^{m}$ be admissible. For a vertex $v\in\cube^{n}$ with zero positions $r_1<\dots<r_Z$ ($Z=|v|^*$), let $j_i=j_i(v)$ be the index such that $p(e_i)$ flips the $j_i$-th zero of $p(v)$, and define $\operatorname{inv}(j_\bullet):=\#\{\,i<i'\mid j_i>j_{i'}\,\}$, the inversion number of the zero-index sequence. Define
\[
\lambda^{p}_v := (-1)^{\operatorname{inv}(p(v))-\operatorname{inv}(v)+\operatorname{inv}(j_\bullet)}.
\]
Then:
\begin{enumerate}[label=$(\roman*)$]
    \item For an elementary edge $e_i:v\to v_i$ flipping the $i$-th zero of $v$,
    \[
    \lambda^{p}_{v_i}/\lambda^{p}_{v} = (-1)^{\xi_{j_i}(p(v))-\xi_{i}(v)}.
    \]
    \item $\lambda^{\mathrm{id}}\equiv 1$, and $\lambda^{q\circ p}_v=\lambda^{q}_{p(v)}\cdot\lambda^{p}_v$ for composable admissible functors.
    \item If $p$ appends terminal $1$s, then $\lambda^{p}\equiv 1$.
\end{enumerate}
\end{lemma}

\begin{proof}
By injectivity of $p$ and coweight preservation, the coordinates flipped by $p(e_i)$ are pairwise distinct and exhaust the zeros of $p(v)$; thus $(j_1,\dots,j_Z)$ is a permutation of $\{1,\dots,Z\}$. Also $p(\mathbf{1}_n)=\mathbf{1}_m$, and $\lambda^{p}_{\mathbf{1}_n}=1$ since there are no zeros and no inversions.

$(i)$ Let $e_i:v\to v_i$ flip the $i$-th zero of $v$ at position $r_i$, and let $c_i$ be the coordinate flipped by $p(e_i)$ (the $j_i$-th zero of $p(v)$). Flipping the zero at rank $t$ in a string with $Z$ zeros changes the inversion count by $(Z-t)-\xi$, where $\xi$ counts ones before the flipped position. Using the identities $\xi_{j_i}(p(v))=c_i-j_i$ and $\xi_i(v)=r_i-i$, we have
\[
\operatorname{inv}(p(v_i))-\operatorname{inv}(p(v)) = Z - c_i,
\qquad
\operatorname{inv}(v_i)-\operatorname{inv}(v) = Z - r_i.
\]
We now compute the change in $\operatorname{inv}(j_\bullet)$. Passing from $v$ to $v_i$ removes the entry $j_i$ at position $i$ and decrements all entries larger than $j_i$. Decrementing preserves inversions; deleting the entry at position $i$ removes exactly the inversions involving it, so
\[
\operatorname{inv}(j_\bullet(v_i))-\operatorname{inv}(j_\bullet(v))
= -\big(\#\{k<i\mid j_k>j_i\}+\#\{k>i\mid j_k<j_i\}\big)
\equiv j_i - i \pmod{2}.
\]
Combining the three contributions, the exponent difference is
\begin{align*}
&\big[\operatorname{inv}(p(v_i))-\operatorname{inv}(v_i)+\operatorname{inv}(j_\bullet(v_i))\big]
- \big[\operatorname{inv}(p(v))-\operatorname{inv}(v)+\operatorname{inv}(j_\bullet(v))\big] \\
&\quad= \big[(Z-c_i)-(Z-r_i)\big] + (j_i - i) \\
&\quad= (r_i - c_i) + (j_i - i) \\
&\quad\equiv (c_i - j_i) - (r_i - i)
= \xi_{j_i}(p(v)) - \xi_i(v) \pmod{2}.
\end{align*}

$(ii)$ Define $\mu(v):=\lambda^{q\circ p}_v\cdot\big(\lambda^{q}_{p(v)}\lambda^{p}_v\big)^{-1}$. By $(i)$ applied to $p$, to $q$ at edge $p(e_i)$, and to $q\circ p$, the exponents telescope and $\mu$ has trivial increments along every elementary edge. Since $\cube^n$ is connected and $\mu(\mathbf{1}_n)=1$ (all three sign terms vanish at the terminal vertex), we have $\mu\equiv1$. For $p=\mathrm{id}$, $c_i=r_i$ and $j_i=i$, so all terms vanish and $\lambda^{\mathrm{id}}\equiv1$.

$(iii)$ If $p$ appends terminal $1$s, then $c_i=r_i$ and $j_i=i$ for all $i$, and trailing ones create no inversions; hence all three terms vanish and $\lambda^p\equiv1$.
\end{proof}

\begin{theorem}\label{theorem:functoriality}
The Khovanov construction extends to a strict functor
\[
\mathbf{Kh}:\fun_B^{\mathrm{ad}}\longrightarrow s\mathbf{Mod}_{\mathbf{k}}^{\mathcal{A}_B^{\mathrm{op}}}.
\]
For an admissible morphism $\Phi=(p,\eta):F\to G$, the induced chain map $\mathcal{C}(\Phi):\mathcal{C}(F)\to\mathcal{C}(G)$ is given on summands by
\[
\mathcal{C}(\Phi)\big|_{F(v)} = \lambda^{p}_v\cdot\eta_v : F(v)\longrightarrow G(p(v)),
\]
and $\mathbf{Kh}(\Phi) := \mathcal{N}^{-1}\big(\mathbf{h}_{\mathcal{C}(\Phi)}\big)$.
\end{theorem}

\begin{proof}
By coweight preservation, $\mathcal{C}(\Phi)$ is a degree-zero map of graded objects. Fix $v$ with $|v|^*=k$ and outgoing edges $e_i:v\to v_i$. By injectivity of $p$, $p$ restricts to a bijection from $\operatorname{Out}(v)$ to $\operatorname{Out}(p(v))$. Using naturality of $\eta$ and Lemma~\ref{lemma:sign_system}$(i)$, we have
\begin{align*}
\partial\circ\mathcal{C}(\Phi)\big|_{F(v)}
&= \lambda^{p}_v \sum_{i=1}^k (-1)^{\xi_{j_i}(p(v))} G(p(e_i))\circ\eta_v \\
&= \lambda^{p}_v \sum_{i=1}^k (-1)^{\xi_{j_i}(p(v))} \eta_{v_i}\circ F(e_i) \\
&= \sum_{i=1}^k (-1)^{\xi_i(v)} \lambda^{p}_{v_i} \eta_{v_i}\circ F(e_i)
= \mathcal{C}(\Phi)\circ\partial\big|_{F(v)}.
\end{align*}
Hence $\mathcal{C}(\Phi)$ is a chain map.

For $\Psi=(q,\zeta):G\to H$, by Lemma~\ref{lemma:sign_system}$(ii)$, one has
\[
\mathcal{C}(\Psi)\circ\mathcal{C}(\Phi)\big|_{F(v)}
= \lambda^{q}_{p(v)}\lambda^{p}_v \zeta_{p(v)}\circ\eta_v
= \lambda^{q\circ p}_v (\zeta p \circ \eta)_v
= \mathcal{C}(\Psi\circ\Phi)\big|_{F(v)}.
\]
Since $\lambda^{\mathrm{id}}\equiv1$, $\mathcal{C}(\mathbf{1}_F)=\mathrm{id}$. Thus $\mathcal{C}$ is a strict functor $\fun_B^{\mathrm{ad}}\to\mathbf{Ch}_{\ge 0}(\mathcal{A}_B)$, and $\mathbf{Kh}=\mathcal{N}^{-1}\circ\mathbf{h}_{(-)}\circ\mathcal{C}$ is also a strict functor.
\end{proof}

\begin{proposition}\label{proposition:geometric_compatibility}
Let $\Phi=(p,\eta)$ be admissible.
\begin{enumerate}[label=$(\roman*)$]
    \item The morphism $\operatorname{tcof}(\Phi):\operatorname{tcof}_{\cube^n}(\mathcal{E}_F)\to\operatorname{tcof}_{\cube^m}(\mathcal{E}_G)$ induced by the span morphism of Proposition~\ref{proposition:span_morphism} is well-defined.
    \item If $p=\mathrm{id}_{\cube^{n}}$, then $\operatorname{tcof}(\Phi)=\mathbf{Kh}(\Phi)$ in $\operatorname{Ho}\big(s\mathbf{Mod}_{\mathbf{k}}^{\mathcal{A}_B^{\mathrm{op}}}\big)$.
\end{enumerate}
\end{proposition}

\begin{proof}
$(i)$ For a pro-tangle $F$ of rank $n$, the span
\[
S_F = \Big( 0 \longleftarrow \hocolim_{\partial\cube^n} \mathcal{E}_F \xrightarrow{b(F)} \mathcal{E}_F(\mathbf{1}_n) \Big),
\]
is a diagram of shape $\bullet\leftarrow\bullet\rightarrow\bullet$. Since admissible morphisms are puncture-preserving, Proposition~\ref{proposition:span_morphism} yields a strict span morphism $S_\Phi : S_F \to S_G$ with components
\[
\lambda_\Phi = \kappa_p \circ \hocolim\big(\mathcal{E}(\eta)\big) : \hocolim_{\partial\cube^n} \mathcal{E}_F \to \hocolim_{\partial\cube^m} \mathcal{E}_G
\]
on the left, and
\[
a_\Phi = \mathcal{E}\Big( G\big(p(\mathbf{1}_n)\to\mathbf{1}_m\big) \circ \eta_{\mathbf{1}_n} \Big) : \mathcal{E}_F(\mathbf{1}_n) \to \mathcal{E}_G(\mathbf{1}_m)
\]
on the right, satisfying the commutation relation $b(G) \circ \lambda_\Phi = a_\Phi \circ b(F)$.

The total homotopy cofiber $\operatorname{tcof}(F) = \hocof\big(b(F)\big)$ is the homotopy pushout of the span $S_F$. A strict morphism of spans canonically induces a morphism of homotopy pushouts, so $\operatorname{tcof}(\Phi)$ is well-defined.

$(ii)$ For $p=\mathrm{id}$, $\lambda^p\equiv1$, so $\mathcal{C}(\Phi)$ has components $\eta_v$. The identification $\mathbf{N}(\operatorname{tcof}(\mathcal{E}_F)(M))\cong\operatorname{Hom}(M,\overline{\mathcal{C}}(F))$ is natural with respect to fixed-shape natural transformations, as all intervening constructions (Bousfield–Kan models, mapping cones, re-signing) are natural. Under this identification, $\operatorname{tcof}(\Phi)$ corresponds to $\mathcal{C}(\Phi)$, so it agrees with $\mathbf{Kh}(\Phi)$ in the homotopy category by full faithfulness of $\mathcal{J}$.
\end{proof}

\begin{remark}
Crossing maps $\Phi_i=(\mathrm{id},\eta^{(i)})$ and upper smoothing inclusions $\mathcal{I}_{i,1}=(\delta_{i,1},\mathrm{id})$ are admissible. Under the cone decomposition $\mathcal{C}(F)\cong\operatorname{Cone}\big(\mathcal{C}(F^{(i)}_0)\to\mathcal{C}(F^{(i)}_1)\big)$, $\mathcal{C}(\Phi_i)$ is the cone arrow, while $\mathcal{C}(\mathcal{I}_{i,1})$ is the canonical inclusion of the target into the mapping cone, up to the sign $\lambda^{\delta_{i,1}}$. The smoothing compatibility theorem restricts to admissible morphisms for $\epsilon=1$, and $\mathbf{Kh}$ sends it to a commutative square. Lower face inclusions are puncture-preserving but not admissible; their role is encoded in the cofiber sequences of Proposition~\ref{proposition:iterated_cofiber}.
\end{remark}

\begin{remark}
The functors $\overline{\mathbf{Kh}}$ and $\mathbf{Kh}$ are not faithful. For $n\ge 2$, let $\tau$ be a coordinate transposition and $\eta=0$ the zero natural transformation. Then $(\mathrm{id},0)$ and $(\tau,0)$ are distinct admissible morphisms, but both induce the zero map on Khovanov objects.
\end{remark}

\begin{remark}
Since $\fun_B$ combines cubes of all ranks, it is neither complete nor cocomplete, which obstructs a direct model category structure on $\fun_B$ itself. The functor $\mathbf{Kh}$ realizes Khovanov homotopy types as simplicial presheaves, but an intrinsic homotopical structure on $\fun_B$ (e.g. a weak model structure) remains an open question, as does extending functoriality beyond the admissible subcategory.
\end{remark}

\subsection{The homologies of pro-tangles}

The category $\mathcal{A}_B$ is pre-additive but not abelian, so direct homological invariants are not available internally. We extract homological invariants by post-composing with a $\mathbf{k}$-linear additive functor into an abelian category, and we show that the resulting homology admits three equivalent characterizations: as the total complex of the resolution cube, as the homology of the Khovanov simplicial object, and as the homotopy groups of an Eilenberg--MacLane total homotopy cofiber; for representable coefficients, these are further identified with the homotopy groups of the global Khovanov simplicial presheaf.

\subsubsection{Khovanov homology with coefficients}

Let $\mathcal{A}$ be a $\mathbf{k}$-linear abelian category and $\Theta:\mathcal{A}_B\to\mathcal{A}$ a $\mathbf{k}$-linear additive functor. By objectwise post-composition, $\Theta$ induces a functor between categories of simplicial objects:
\[
s\Theta: s\mathcal{A}_B \longrightarrow s\mathcal{A}, \qquad X_\bullet \longmapsto \Theta\circ X_\bullet.
\]
Explicitly, $(s\Theta(X_\bullet))_k = \Theta(X_k)$ in each simplicial degree, with face and degeneracy maps induced by functoriality of $\Theta$. Since $\Theta$ is $\mathbf{k}$-linear and additive, it preserves the finite direct sum decompositions of simplicial objects and commutes with the simplicial structure maps.

Let $\mathbf{N}: s\mathcal{A} \to \mathbf{Ch}_{\ge 0}(\mathcal{A})$ denote the normalized Moore complex functor for simplicial objects over the abelian category $\mathcal{A}$; this is the right adjoint of the Dold--Kan Quillen equivalence, which identifies simplicial homotopy groups with homology groups of the normalized complex.

\begin{definition}
The \textit{Khovanov chain complex with coefficients in $\Theta$} is
\[
\mathcal{C}(F;\Theta) := \mathbf{N}\big(s\Theta(\mathcal{S}(F))\big) \in \mathbf{Ch}_{\ge 0}(\mathcal{A}),
\]
where $\mathcal{S}(F)\in s\mathcal{A}_B$ is the Khovanov simplicial object of Definition~\ref{definition:Khovanov_simplicial_object}. The \textit{Khovanov homology} of $F$ with coefficients in $\Theta$ is the homology of this chain complex
\[
\kh_k(F;\Theta) := H_k\big(\mathcal{C}(F;\Theta)\big) \in \mathcal{A}.
\]
\end{definition}

By additivity of $\Theta$, the canonical decomposition of each $\mathcal{S}_k(F)$ into non-degenerate and degenerate summands is preserved under $s\Theta$; hence normalization commutes with $\Theta$, and applying $\mathbf{N}\circ s\Theta$ to $\mathcal{S}(F)$ recovers the total complex of the resolution cube
\[
\mathcal{C}(F;\Theta) \cong \operatorname{Tot}(\Theta\circ\mat\circ F),
\]
where $\operatorname{Tot}$ denotes the total complex of a cubical diagram with coweight grading $|v|^*$ and sign rule $(-1)^{\xi_i(v)}$, exactly as in Lemma~\ref{lemma:total_complex}. Indeed, since $\Theta$ is additive, it preserves the splitting of each $\mathcal{S}_k(F)$ into non-degenerate and degenerate summands, and since $\Theta$ is $\mathbf{k}$-linear, it commutes with the signs $(-1)^{\xi_i(v)}$ occurring in the differential; the identifications of Lemma~\ref{lemma:dold_kan_compatibility} therefore descend to $\Theta$-coefficients.

In particular, for the representable functor $\Theta_M = \operatorname{Hom}_{\mathcal{A}_B}(M,-)$ associated with a testing object $M\in\mathcal{A}_B$, we recover the chain complex
\[
\mathcal{C}(F;\Theta_M) \cong \operatorname{Hom}_{\mathcal{A}_B}\big(M,\mathcal{C}(F)\big).
\]
Combined with Theorem~\ref{theorem:representation} and the Dold--Kan identification $\pi_k \cong H_k\circ\mathbf{N}$, this yields:
\begin{corollary}
For every $M\in\mathcal{A}_B$ and every $k\ge 0$, there is a canonical isomorphism
\[
\pi_k\big(\mathbf{Kh}(F)(M)\big) \cong \kh_k(F;\Theta_M).
\]
\end{corollary}

Passing to an abelian target category $\mathcal{A}$ gives access to standard homological tools: long exact sequences, spectral sequences, and (if $\mathcal{A}$ has enough projectives) derived functor formalisms.

\subsubsection{Cellular cube homology}

Khovanov homology admits a direct combinatorial description as cellular homology of the Boolean cube lattice, independent of simplicial machinery.

Let $P$ be a finite poset, graded by a function $\operatorname{rk}:P\to\mathbb{Z}$ that decreases along the arrows of $P$, and let $\mathcal{A}$ be a $\mathbf{k}$-linear abelian category. For any functor $X:P\to\mathcal{A}$, the \textit{cellular chain complex} of $X$ is defined by
\[
C_k(P;X) = \bigoplus_{\operatorname{rk}(v)=k} X(v),
\]
with boundary differential $\partial_k: C_k(P;X)\to C_{k-1}(P;X)$ given by
\[
\partial_k\big|_{X(v)} = \sum_{v\to u} (-1)^{\operatorname{sgn}(v,u)} X(v\to u),
\]
where the sum runs over all covering arrows $v\to u$ out of $v$ (so that $\operatorname{rk}(u)=\operatorname{rk}(v)-1$), and the orientation signs are chosen so that all commuting squares of covering arrows anti-commute; this guarantees $\partial^2=0$. The homology of this complex is the \textit{cellular poset homology} of $X$, denoted $H_k^\mathrm{cell}(P;X)$.

We now specialize to the Boolean cube $\cube^n$, equipped with the coweight grading $|v|^* = n-|v|$ as rank function, which decreases along the edges $v\to v_i$. For an elementary edge $e_i:v\to v_i$ flipping the $i$-th zero of $v$, the standard orientation sign is $(-1)^{\xi_i(v)}$, where $\xi_i(v)$ counts the number of $1$s preceding the flipped position; the anti-commutation of squares holds by the classical Khovanov sign computation of Lemma~\ref{lemma:total_complex}.

Given a pro-tangle $F:\cube^n\to\cob(B)$, define the coefficient diagram
\[
F^\Theta = \Theta\circ\mat\circ F : \cube^n \longrightarrow \mathcal{A}.
\]

\begin{definition}
The \textit{cellular cube complex} of $F^\Theta$ is the cellular chain complex $C_\bullet(\cube^n;F^\Theta)$ of the poset $\cube^n$ graded by coweight. The \textit{cellular cube homology} is
\[
H_k^\mathrm{cell}(\cube^n;F^\Theta) = H_k\big(C_\bullet(\cube^n;F^\Theta), \partial\big).
\]
\end{definition}

\begin{remark}
The construction of the cellular cube complex requires only the additive structure of $\mathcal{A}$, that is, finite direct sums. Taking homology, however, uses kernels and cokernels and thus the abelian structure. The assumption that $\mathcal{A}$ has enough projectives is not needed for the construction, only for derived functor interpretations.
\end{remark}

\begin{proposition}\label{proposition:cellular_equals_khovanov}
There is a canonical isomorphism of chain complexes
\[
C_\bullet(\cube^n;F^\Theta) \cong \mathcal{C}(F;\Theta),
\]
inducing canonical isomorphisms on homology
\[
H_k^\mathrm{cell}(\cube^n;F^\Theta) \cong \kh_k(F;\Theta),\qquad k\ge 0.
\]
\end{proposition}

\begin{proof}
The coweight grading identifies the summands degreewise: $C_k(\cube^n;F^\Theta) = \bigoplus_{|v|^*=k} F^\Theta(v) = \mathcal{C}_k(F;\Theta)$. The boundary differentials are both given by the alternating sum of edge morphisms with the same sign rule $(-1)^{\xi_i(v)}$. The identity on summands therefore extends to an isomorphism of chain complexes, which descends to an isomorphism on homology.
\end{proof}

\begin{example}
For a rank-$1$ pro-tangle $F:\cube^1\to\cob(B)$ with unique edge $e:0\to1$, the cellular complex reduces to the two-term complex
\[
F^\Theta(0) \xrightarrow{\Theta(F(e))} F^\Theta(1),
\]
concentrated in degrees $1$ and $0$. Its homology is $H_1^\mathrm{cell}=\ker\Theta(F(e))$ and $H_0^\mathrm{cell}=\operatorname{coker}\Theta(F(e))$, matching the standard Khovanov homology of a single crossing. This coincides with the fact that the $1$-cube cellular complex is exactly the mapping cone of the edge map.
\end{example}

\subsubsection{Eilenberg--MacLane simplicial realization}

We finally realize Khovanov homology as the simplicial homotopy groups of a cubical diagram of Eilenberg--MacLane objects, via the total homotopy cofiber construction.

For any object $X\in\mathcal{A}$ and integer $n\ge 0$, let $K_\mathcal{A}(X,n)\in s\mathcal{A}$ denote the $n$-th Eilenberg--MacLane simplicial object, characterized by
\[
\pi_k K_\mathcal{A}(X,n) \cong
\begin{cases}
X, & k=n, \\
0, & k\neq n.
\end{cases}
\]
Under the Dold--Kan correspondence, the normalized chain complex of $K_\mathcal{A}(X,n)$ is the shifted chain complex
\[
\mathbf{N}\, K_\mathcal{A}(X,n) \cong X[n],
\]
equivalently $K_\mathcal{A}(X,n)=\Gamma(X[n])$ for the denormalization functor $\Gamma$.

For a pro-tangle $F:\cube^n\to\cob(B)$, let $F^\Theta: \cube^n \longrightarrow \mathcal{A}$ be a coefficient diagram. We define the \textit{Eilenberg--MacLane resolution diagram}
\[
K_\mathcal{A}(F^\Theta,0):\cube^n\to s\mathcal{A},
\qquad
v\longmapsto K_\mathcal{A}\big(F^\Theta(v),0\big)=\Gamma\big(F^\Theta(v)\big),
\]
with edge maps $\Gamma\big(F^\Theta(e)\big)$ for each elementary edge $e$ of $\cube^n$. Since $\mathbf{N}\circ\Gamma \cong \mathrm{id}$, this is a lifting of the degree-zero diagram $F^\Theta$ along $\mathbf{N}$.

Let $\operatorname{tcof}_{\cube^n} : \operatorname{Fun}(\cube^n, \mathbf{Ch}_{\ge 0}(\mathcal{A})) \to \mathbf{Ch}_{\ge 0}(\mathcal{A})$ denote the total cofiber functor on cubical diagrams, given by iterated mapping cones along each coordinate direction. We define its simplicial counterpart
\[
\operatorname{tcof}_{\cube^n} : \operatorname{Fun}(\cube^n, s\mathcal{A}) \to s\mathcal{A}
\]
as the lift along the Dold–Kan adjunction:
\[
\operatorname{tcof}_{\cube^n}(Y) := \Gamma \circ \operatorname{tcof}_{\cube^n} \circ \mathbf{N}\, (Y),
\]
where $\Gamma: \mathbf{Ch}_{\ge 0}(\mathcal{A}) \to s\mathcal{A}$ is the denormalization functor, left adjoint to the normalization functor $\mathbf{N}$.

When $\mathcal{A} = \mathbf{Mod}_{\mathbf{k}}$, the simplicial object $s\mathcal{A}$ admits the standard projective model structure and $\Gamma$ is left Quillen. In this case $\operatorname{tcof}_{\cube^n}(Y)$ is canonically weakly equivalent to the homotopy cofiber of the punctured cube structure map described in Definition~\ref{definition:Khovanov_presheaf}.

Set
\[
\mathbb{X}_{\kh} = \operatorname{tcof}_{\cube^n} K_\mathcal{A}(F^\Theta,0) \in s\mathcal{A}.
\]

\begin{proposition}\label{proposition:khovanov_bk_realization}
There is a canonical natural isomorphism
\[
\pi_k\big(\mathbb{X}_{\kh}\big) \cong \kh_k(F;\Theta),\qquad k\ge 0.
\]
\end{proposition}

\begin{proof}
Since $\mathbf{N}\circ\Gamma \cong \mathrm{id}$, we have $\mathbf{N}\circ K_\mathcal{A}(F^\Theta,0)=F^\Theta$ as degree-zero cubical diagrams, whence
\[
\mathbf{N}\big(\mathbb{X}_{\kh}\big)
= \mathbf{N}\Big(\operatorname{tcof}_{\cube^n} K_\mathcal{A}(F^\Theta,0)\Big)
\cong \operatorname{tcof}_{\cube^n}\big(F^\Theta\big)
\cong \operatorname{Tot}(F^\Theta)
\cong \mathcal{C}(F;\Theta),
\]
where the first isomorphism holds by the definition of $\operatorname{tcof}$, the middle one is Lemma~\ref{lemma:total_complex}, and the last one is the identification established above. Taking homology and applying the Dold--Kan isomorphism $\pi_k \cong H_k\circ\mathbf{N}$ yields
\[
\pi_k\big(\mathbb{X}_{\kh}\big) \cong H_k\big(\mathbf{N}(\mathbb{X}_{\kh})\big) \cong H_k\big(\mathcal{C}(F;\Theta)\big) = \kh_k(F;\Theta),
\]
as required.
\end{proof}

\begin{proposition}\label{proposition:presheaf_EM_comparison}
Let $F$ be a pro-tangle.
\begin{enumerate}[label=$(\roman*)$]
\item If $\Theta=\Theta_M=\operatorname{Hom}_{\mathcal{A}_B}(M,-)$ for some $M\in\mathcal{A}_B$, then
$\mathbb{X}_{\kh} \simeq \mathbf{Kh}(F)(M)$.
\item For general additive $\Theta:\mathcal{A}_B\to\mathcal{A}$, $\mathbb{X}_{\kh}$ is the image of the Khovanov simplicial object $\mathcal{S}(F)$ under $\Gamma\circ s\Theta$. The weak equivalence class of $\mathbf{Kh}(F)$ determines all $\mathbb{X}_{\kh}[\Theta]$, but not vice versa.
\end{enumerate}
\end{proposition}

\begin{proof}
$(i)$ By Theorem~\ref{theorem:representation} and Lemma~\ref{lemma:dold_kan_compatibility}$(ii)$, one has
\[
\mathbf{N}\big(\mathbf{Kh}(F)(M)\big) \simeq \operatorname{Hom}_{\mathcal{A}_B}(M,\mathcal{C}(F)) = \mathcal{C}(F;\Theta_M).
\]
By the proof of Proposition~\ref{proposition:khovanov_bk_realization}, we obtain $\mathbf{N}(\mathbb{X}_{\kh}) \cong \mathcal{C}(F;\Theta_M)$. Since $\mathbf{N}$ reflects weak equivalences, the quasi-isomorphisms lift to $\mathbb{X}_{\kh} \simeq \mathbf{Kh}(F)(M)$.

$(ii)$ By construction, we have
\[
\mathbb{X}_{\kh} = \operatorname{tcof}_{\cube^n} K_\mathcal{A}(F^\Theta,0) = \Gamma\big(\operatorname{tcof}_{\cube^n}(F^\Theta)\big) \cong \Gamma(\mathcal{C}(F;\Theta)),
\]
and $\mathcal{C}(F;\Theta) = \mathbf{N}(s\Theta(\mathcal{S}(F)))$. Hence $\mathbb{X}_{\kh}$ is the image of the Khovanov simplicial object $\mathcal{S}(F)$ under the composite functor $\Gamma\circ s\Theta: s\mathcal{A}_B \to s\mathcal{A}$. By Corollary~\ref{corollary:algebraic_representation}, the weak equivalence class of $\mathbf{Kh}(F)$ equals the strict chain homotopy class of $\mathcal{C}(F)$. Additive functors send strict chain homotopy equivalences to quasi-isomorphisms, so $\mathbf{Kh}(F)$ determines each $\mathbb{X}_{\kh}[\Theta]$.
The converse fails over general ground rings, as quasi-isomorphisms need not admit chain homotopy inverses.
\end{proof}

In summary, for any $\mathbf{k}$-linear additive functor $\Theta\colon \mathcal{A}_B \to \mathcal{A}$ taking values in an abelian category, the Khovanov homology $\kh_k(F;\Theta)$ obtained from the Khovanov simplicial object by Moore normalization, the cellular cube homology $H_k^\mathrm{cell}(\cube^n;F^\Theta)$ defined combinatorially on the resolution cube, and the simplicial homotopy groups $\pi_k(\mathbb{X}_{\kh})$ of the total homotopy cofiber of the Eilenberg--MacLane resolution diagram are all canonically isomorphic.

\section{Algebraic spectral sequence of pro-tangles}\label{section:spectral_sequence}

In this section, we present the decomposition of pro-tangle functors into local resolutions and constructs the associated algebraic spectral sequences. We determine the precise relation between the local Khovanov homology of these fiber diagrams and the global homological invariant. As a practical application, this decomposition provides the categorical basis to verify invariance under Reidemeister equivalence using pure diagrammatic coherence.

\subsection{Local homotopy reduction}

Consider a functor $F: \cube^{n} \to \mathcal{A}_B$ and the standard decomposition of the resolution cube into a base and a fiber,
\[
\cube^n \cong \cube^{k}\times\cube^{n-k},
\]
where a vertex $v=(s,t)\in\cube^n$ is written as $v=(s,t)$ with $s\in\cube^{k}$ the base part and $t\in\cube^{n-k}$ the fiber part. For each base vertex $s\in\cube^{k}$, the local fiber diagram is the restriction functor
\[
F_s = F\circ j_s = F(s,\cdot):\cube^{n-k}\longrightarrow\mathcal{A}_B,
\]
where $j_s:\cube^{n-k}\hookrightarrow\cube^{n}$, $t\mapsto(s,t)$, is the canonical inclusion.

The local Khovanov complex of the fiber is $\mathcal{C}(F_s)\in\mathbf{Ch}_{\ge0}(\mathcal{A}_B)$, the total complex of the $(n-k)$-cube $\mat\circ F_s$. The base structure is encoded in the diagram
\[
Y_F:\cube^{k}\longrightarrow\mathbf{Ch}_{\ge0}(\mathcal{A}_B),
\qquad s\longmapsto\mathcal{C}(F_s),
\]
whose structure map along a base edge $e:s\to s'$ is the chain map $Y_F(e):\mathcal{C}(F_s)\to\mathcal{C}(F_{s'})$ induced componentwise by $F(e\times\mathrm{id})$. This is the $p=\mathrm{id}$ instance of Theorem~\ref{theorem:functoriality}, which also guarantees that $Y_F$ is a functor.

\begin{theorem}
\label{theorem:iterated_decomposition}
There is a natural isomorphism of chain complexes
\[
\mathcal{C}(F) \cong \operatorname{tcof}_{\cube^{k}}\big(Y_F\big)
 = \operatorname{tcof}_{\cube^{k}}\big(s\mapsto\mathcal{C}(F_s)\big).
\]
\end{theorem}

\begin{proof}
By Lemma~\ref{lemma:total_complex}, $\mathcal{C}(F)$ is the total complex of the $n$-dimensional cubical diagram $\mat\circ F$, with edge signs given by counting ones preceding each flipped zero coordinate. We establish the result via a chain of canonical isomorphisms.

Set $l = n-k$. Along the decomposition $\cube^n = \cube^k \times \cube^l$, define the bigraded module $K^{\bullet,\bullet}$ with components
\[
K^{p,q} = \bigoplus_{\substack{s\in\cube^k,\, t\in\cube^l \\ |s|=p,\, |t|=q}} F(s,t).
\]
Endow $K^{\bullet,\bullet}$ with differentials: the vertical differential $\delta^0: K^{p,q} \to K^{p,q+1}$, the internal cubical differential of the fiber direction, with signs counted locally within each fiber vertex $t$; the horizontal differential $\delta^1: K^{p,q} \to K^{p+1,q}$, the base cubical differential with signs counted locally within each base vertex $s$, twisted by the Koszul sign $(-1)^q$ in fiber degree $q$.

By the Koszul sign convention, $\delta^0\delta^1 + \delta^1\delta^0 = 0$, so $(K^{\bullet,\bullet}, \delta^0, \delta^1)$ is a bicomplex. Write $\operatorname{Tot}(K)$ for its total complex with total differential $D = \delta^0 + \delta^1$.

The underlying graded modules of $\mathcal{C}(F)$ and $\operatorname{Tot}(K)$ are identical, placing $F(s,t)$ in total degree $p+q$. The two differentials differ by the sign discrepancy between global and local zero-indexing, which is absorbed by the graded automorphism
\[
\Phi\big|_{F(s,t)} = (-1)^{|s|(l - |t|)}.
\]
A direct computation shows $\Phi \circ d_{\mathcal{C}(F)} = D \circ \Phi$, hence $\mathcal{C}(F) \cong \operatorname{Tot}(K)$.

Now let $Y_F: \cube^k \to \mathbf{Ch}_{\ge0}(\mathcal{A}_B)$ be the base cubical diagram sending each vertex $s$ to the fiber Khovanov complex $\mathcal{C}(F_s)$, with structure maps induced by $F(e\times\mathrm{id})$ along base edges $e$. The degree-wise sign automorphisms $(-1)^{q|s|}$ assemble to an isomorphism of cubical diagrams between $Y_F$ and the twisted base diagram whose totalization recovers $\operatorname{Tot}(K)$. Applying the totalization functor yields $\operatorname{Tot}(K) \cong \operatorname{Tot}(Y_F)$.

By Lemma~\ref{lemma:total_complex}, $\operatorname{Tot}(Y_F) \cong \operatorname{tcof}_{\cube^k}(Y_F)$. Composing these canonical isomorphisms gives the result. Naturality follows from functoriality of each construction.
\end{proof}

\subsection{Spectral sequence for pro-tangle decompositions}
Recall that for a pro-tangle functor $F: \cube^n \to \cob(B)$, composing with the TQFT-type construction $\overline{\Theta} = \Theta \circ \mat: \cob(B) \to \mathcal{A}$ yields the diagram $F^\Theta = \overline{\Theta} \circ F: \cube^n \to \mathcal{A}$ within the abelian category $\mathcal{A}$.

For $k$ selected crossings of $F$, we consider a decomposition of the resolution cube $\cube^n \cong \cube^k \times \cube^{n-k}$, where $\cube^k$ acts as the base and $\cube^{n-k}$ as the fiber. Define the bigraded module $K^{\bullet,\bullet}$ by
\[
    K^{p,q} = \bigoplus_{\substack{s\in\cube^k,\, t\in\cube^{n-k} \\ |s|=p,\, |t|=q}} F^{\Theta}(s, t),
\]
where $p$ is the base Hamming weight and $q$ is the fiber Hamming weight.

\begin{definition}
The \textit{total complex} $(\operatorname{Tot}(K), D)$ is the totalization of the bicomplex $(K^{\bullet,\bullet}, \delta^0, \delta^1)$, with total differential $D = \delta^0 + \delta^1$. Here:
\begin{itemize}
    \item the vertical (fiber) differential $\delta^0: K^{p,q} \to K^{p,q+1}$ is the standard cubical differential along the fiber direction, with edge signs given by counting ones preceding the flipped zero coordinate within each fiber vertex $t$;
    \item the horizontal (base) differential $\delta^1: K^{p,q} \to K^{p+1,q}$ is given by
      \[
      \delta^1\big|_{F^\Theta(s,t)} = \sum_{e: s\to s'} (-1)^{\xi_e(s) + q} F^\Theta(e\times\operatorname{id}),
      \]
      where the sum runs over elementary edges $e$ of the base cube $\cube^k$, $\xi_e(s)$ counts ones preceding the zero flipped by $e$ within $s$, and $(-1)^q$ is the Koszul degree twist.
\end{itemize}
\end{definition}

By construction, $(\delta^0)^2 = 0$ and $(\delta^1)^2 = 0$ as each is a cubical differential. The Koszul sign ensures the anti-commutation relation
\[
\delta^0\delta^1 + \delta^1\delta^0 = 0,
\]
so that $D^2 = 0$ and $\operatorname{Tot}(K)$ is a well-defined chain complex.

This complex admits a bounded descending filtration $I^{\bullet}$ defined by the base degree:
\[
I^p = \bigoplus_{i \ge p} \bigoplus_{q} K^{i,q}.
\]
It follows that $D(I^p) \subseteq I^p$, rendering each $I^p$ a subcomplex of $\operatorname{Tot}(K)$.

\begin{theorem}\label{theorem:spectral_sequence}
The filtration $I^{\bullet}$ induces a first-quadrant spectral sequence $\{E_r, d_r\}$ with
\[
E_1^{p,q} \Rightarrow \operatorname{Kh}^{p+q}(F),
\]
which collapses at the $(k+1)$-th page, i.e., $E_{k+1} \cong E_{\infty}$. The initial pages are characterized as follows:
\begin{itemize}
    \item $E_0^{p,q} = I^p / I^{p+1} \cong K^{p,q}$, with $d_0 = \delta^0$.
    \item $E_1^{p,q} = H^q(K^{p,\bullet}, \delta^0) \cong \bigoplus_{|s|=p} \operatorname{Kh}^q(F \circ j_s)$, where $j_s: \cube^{n-k}\hookrightarrow\cube^n$ is the canonical inclusion $t\mapsto (s,t)$, and $d_1$ is induced by the horizontal differential $\delta^1$.
\end{itemize}
\end{theorem}

\begin{proof}
The $E_0$ page is given by the associated graded pieces of the filtration. The zeroth differential $d_0$ is induced by $D$ on the quotient $I^p / I^{p+1}$, which preserves the filtration degree $p$, so $d_0 = \delta^0$. Taking homology with respect to $\delta^0$ yields the $E_1$ page, which is the direct sum of the fiber Khovanov homologies for each fixed base resolution $s$. The first differential $d_1: E_1^{p,q} \to E_1^{p+1,q}$ is induced by the horizontal differential $\delta^1$, and $E_2$ is the homology of the $E_1$ page.

By Theorem~\ref{theorem:iterated_decomposition}, there is a canonical isomorphism of chain complexes $\operatorname{Tot}(K) \cong \mathcal{C}(F)$. Hence the homology of $\operatorname{Tot}(K)$ coincides with the Khovanov homology of $F$. Since the filtration is finite and bounded, the spectral sequence converges to the homology of the total complex, with
\[
E_\infty^{p,q} \cong \operatorname{Gr}_p \operatorname{Kh}^{p+q}(F),
\]
where $\operatorname{Gr}_p$ denotes the $p$-th associated graded component with respect to the base-degree filtration.

The $r$-th differential $d_r$ has bidegree $(r, 1-r)$, acting as $d_r: E_r^{p,q} \to E_r^{p+r, q-r+1}$. Since the horizontal grading $p$ is bounded above by $k$ (the dimension of the base cube), we have $K^{a,\bullet} = 0$ for all $a > k$. Consequently, for all $r \ge k+1$, the target horizontal degree $p+r$ exceeds $k$, forcing $d_r$ to vanish identically. This boundary constraint implies the collapse $E_{k+1} \cong E_{\infty}$.
\end{proof}

\begin{example}
When the base resolution space is one-dimensional with cube $\cube^1_{\mathrm{base}}$, the spectral sequence recovers the classical mapping-cone decomposition of the Khovanov complex. Specifically, the total complex is quasi-isomorphic to the mapping cone of the saddle morphism:
\[
    \mathcal{C}(F) \simeq \operatorname{Cone}\left( \delta^1: \mathcal{C}(F_0) \to \mathcal{C}(F_1)[1] \right).
\]
The $E_1$ page consists of the Khovanov homologies of the two resolutions: $E_1^{0,q} \cong \operatorname{Kh}^q(F_0)$ and $E_1^{1,q} \cong \operatorname{Kh}^q(F_1)$. The $d_1$ differential is precisely the map induced by the saddle cobordism on homology. The $E_2$ page, being the last non-trivial page, computes the kernel and cokernel of this induced map, whose extensions reconstruct the total homology $\operatorname{Kh}(F)$.
\end{example}

\subsection{Induced maps on pro-tangle spectral sequences}

In this section, we establish the functoriality of the spectral sequence associated with the decomposition of pro-tangles. We show that morphisms in $\fun_B$ that respect the base-fiber structure induce well-defined morphisms between the corresponding spectral sequences.

\begin{definition}
Consider a decomposition of the resolution spaces as $\mathbb{B}^n \cong \mathbb{B}^k_{\mathrm{base}} \times \mathbb{B}^{n-k}_{\mathrm{fiber}}$. 
A morphism $p: \mathbb{B}^n \to \mathbb{B}^m$ is said to be \textit{fibered over $\mathbb{B}^k_{\mathrm{base}}$} if the map $p$ admits a product decomposition
\begin{equation*}
    p = \mathrm{id}_{\mathbb{B}^k} \times p_{\mathrm{fiber}} : \mathbb{B}^k_{\mathrm{base}} \times \mathbb{B}^{n-k}_{\mathrm{fiber}} \to \mathbb{B}^k \times \mathbb{B}^{m-k}_{\mathrm{fiber}},
\end{equation*}
where $p_{\mathrm{fiber}}: \mathbb{B}^{n-k}_{\mathrm{fiber}} \to \mathbb{B}^{m-k}_{\mathrm{fiber}}$ is a functor between the fiber components.
\end{definition}

Briefly, the morphism $p: \mathbb{B}^n \to \mathbb{B}^m$ admits a product decomposition $p = \mathrm{id}_{\mathbb{B}^k} \times p_{\mathrm{fiber}}$ up to permutation.

\begin{definition}
Let $F: \cube^n \to \cob(B)$ and $G: \cube^m \to \cob(B)$ be pro-tangles. Consider a decomposition of their resolution spaces as $\cube^n \cong \cube^k_{\mathrm{base}} \times \cube^{n-k}_{\mathrm{fiber}}$. A morphism $\Phi = (p, \eta): F \to G$ in $\fun_B$ is said to be \textit{fibered} over $\cube^k_{\mathrm{base}}$ if the map $p$ is fibered over $\cube^k_{\mathrm{base}}$.
\end{definition}

Suppose $F: \cube^n \to \cob(B)$ and $G: \cube^m \to \cob(B)$ are pro-tangles with $k<\min\{m,n\}$. Set $\cube^n \cong \cube^k_{\mathrm{base}} \times \cube^{n-k}_{\mathrm{fiber}}$.

A fibered morphism $\Phi = (p, \eta): F \to G$ over the base cube $\mathbb{B}^k_{\mathrm{base}}$ induces a map $f: \mathrm{Tot}(K_F) \to \mathrm{Tot}(K_G)$ between the associated total complexes. Specifically, for any element $x \in \mathrm{Tot}(K_F)$, its image is determined vertex-wise by
\begin{equation*}
    f(x) = \sum_{|s|=p, |t|=q} \eta_{(s,t)}(x_{s,t}),
\end{equation*}
where $x_{s,t} \in F^{\Theta}(s,t)$ denotes the component of $x$ at the resolution vertex $(s,t) \in \mathbb{B}^k_{\mathrm{base}} \times \mathbb{B}^{n-k}_{\mathrm{fiber}}$. 

\begin{lemma}\label{lemma:filtration}
Let $\Phi = (p, \eta): F \to G$ be a fibered morphism over $\cube^k_{\mathrm{base}}$. The induced map $f: \mathrm{Tot}(K_F) \to \mathrm{Tot}(K_G)$ is a filtered map with respect to the base-degree filtration $I^\bullet$, satisfying
\begin{equation*}
    f(I^p \mathrm{Tot}(K_F)) \subseteq I^p \mathrm{Tot}(K_G) \quad \text{for all } p \ge 0.
\end{equation*}
\end{lemma}

\begin{proof}
Recall that $\mathrm{Tot}(K_F) = \bigoplus_{p,q} K_F^{p,q}$, where $K_F^{p,q} = \bigoplus_{|s|=p, |t|=q} F^{\Theta}(s,t)$. The map $f$ is defined vertex-wise by the natural transformation $\eta_{(s,t)}: F(s,t) \to G(p(s,t))$. 

Let $x \in I^p \mathrm{Tot}(K_F)$. Decomposing $x = \sum_{i \ge p} x_i$ with $x_i \in K_F^{i, \bullet}$, it suffices to consider a homogeneous element $x_{s,t} \in F^{\Theta}(s,t)$ with $|s|=i$. Its image under $f$ is $\eta_{(s,t)}(x_{s,t}) \in G(s, p_{\mathrm{fiber}}(t))$. Since the base coordinate $s$ is preserved by the product structure $p = \mathrm{id} \times p_{\mathrm{fiber}}$, the base degree of the image is precisely $|s|=i$. Consequently, $f(x_{s,t}) \in K_G^{i, \bullet} \subseteq I^i \mathrm{Tot}(K_G)$. As $i \ge p$, we have $f(x) \in I^p \mathrm{Tot}(K_G)$, proving that $f$ is a filtered chain map of degree zero.
\end{proof}

Crucially, while the natural transformation $\eta_{(s,t)}$ is required to preserve the base coordinate $s$, it may induce a non-uniform shift in the fiber degree depending on the specific vertex $(s,t)$. Consequently, $f$ maps the $(p,q)$-th bigraded component $K_F^{p,q}$ into a direct sum of components in $K_G$ that share the same base degree $p$ but may span multiple fiber degrees $q'$. Formally, this yields the inclusion
\begin{equation*}
    f(K_F^{p,q}) \subseteq \bigoplus_{q'} K_G^{p,q'}.
\end{equation*}
Despite this potential non-homogeneity in the second grading, $f$ behaves like a chain map. The naturality of $\eta$ ensures commutativity with both the vertical differential $\delta^0$ and the horizontal differential $\delta^1$.

\begin{theorem}\label{theorem:alegbraic_ss}
Any fibered morphism $\Phi = (p, \eta) \in \fun_B$ induces a filtered map between spectral sequences $\{f_r\}_{r \ge 0}$, acting as $f_r^{p}: E_r^{p,\bullet}(F) \to E_r^{p,\bullet}(G)$ for each base filtration degree $p$. This morphism satisfies:
\begin{enumerate}[label=$(\roman*)$]
    \item For each $r \ge 0$, the induced maps commute with the spectral sequence differentials, i.e., $d_r^G \circ f_r = f_r \circ d_r^F$.
    
    \item The morphism $f_1$ on the $E_1$ page is the direct sum of maps induced by the natural transformation $\eta$ on the local fiber homologies:
    \begin{equation*}
        f_1^{p}: \bigoplus_{|s|=p} Kh(F|_{s\times \mathbb{B}^{n-k}}) \longrightarrow \bigoplus_{|s|=p} Kh(G|_{s\times \mathbb{B}^{m-k}}).
    \end{equation*}
    Here, $f_1^p$ is not necessarily homogeneous with respect to the fiber grading $q$, but decomposes into components $f_1^{(p, q, \delta)}: E_1^{p,q}(F) \to E_1^{p,q+\delta}(G)$ determined by the degree shifts of $\eta_{(s,t)}$.

    \item As the spectral sequences collapse at the $E_{k+1}$ page (i.e., $d_r = 0$ for $r \ge k+1$), the map $f_{k+1} = f_\infty$ coincides with the associated graded morphism $\mathrm{Gr}_p(\Phi_*)$ induced by the map $\Phi_*: Kh(F) \to Kh(G)$ on the total Khovanov homology with respect to the base filtration $I^\bullet$.
\end{enumerate}
\end{theorem}

\begin{proof}
The existence of the filtered map $\{f_r\}_{r \ge 0}$ follows from the general mapping theorem for filtered complexes. Since $f$ is a filtered map satisfying $f(I^p \mathrm{Tot}(K_F)) \subseteq I^p \mathrm{Tot}(K_G)$, it induces a collection of maps on the associated graded objects $E_0^{p,\bullet} \cong I^p/I^{p+1}$. The naturality of the transformation $\eta$ ensures that $f$ commutes with the differentials, $d_r^G \circ f_r = f_r \circ d_r^F$.

On the $E_1$ page, the differential is purely vertical ($d_1 = \delta^0$). Since $f$ is defined vertex-wise by $\eta_{(s,t)}$, the map $f_1^p$ on each column is the direct sum of maps induced by $\eta$ on the fiber homologies $Kh(F|_{s \times \mathbb{B}})$. Although the fiber grading $q$ may shift, the stability of the base coordinate $s$ ensures that $f_1$ maps the $p$-th column of $E_1(F)$ into the $p$-th column of $E_1(G)$.

The collapse at $E_{k+1}$ is a structural property of the pro-tangle double complex. By the comparison theorem for spectral sequences, since $f$ is a filtered map, the morphism $f_\infty$ on the limit page must coincide with the associated graded morphism $\mathrm{Gr}_p(\Phi_*)$ induced by the total filtered map $f$ on the homology $Kh(F) \to Kh(G)$. 
\end{proof}

\begin{definition}
Let $F,G\in \fun^{(n)}_B$. A morphism $\Phi = (p, \eta): F \to G$ in $\fun_B$ is said to be a \textit{fibered matching} over $\cube^k_{\mathrm{base}}$ if the map $p$ admits a product decomposition
\begin{equation*}
    p = \mathrm{id}_{\cube^k_{\mathrm{base}}} \times \mathrm{id}_{\cube^{n-k}_{\mathrm{fiber}}} : \cube^k_{\mathrm{base}} \times \cube^{n-k}_{\mathrm{fiber}} \to \cube^k_{\mathrm{base}} \times \cube^{n-k}_{\mathrm{fiber}}.
\end{equation*}
\end{definition}

\begin{theorem}\label{theorem:fibered_matching}
Let $\Phi = (p, \eta): F \to G$ be a fibered matching over $\mathbb{B}^k_{\mathrm{base}}$. Then the induced map $f: \mathrm{Tot}(K_F) \to \mathrm{Tot}(K_G)$ induces a morphism of spectral sequences $\{f_r\}_{r \ge 0}$ with $f_r^{p,q}: E_r^{p,q}(F) \longrightarrow E_r^{p,q}(G)$ satisfying
\begin{enumerate}[label=$(\roman*)$]
    \item This morphism intertwines the spectral sequence differentials, $d_r^G \circ f_r = f_r \circ d_r^F$.

    \item The morphism $f_1$ on the $E_1$ page is the direct sum of maps induced by $\eta$ on the local fiber homologies:
    \begin{equation*}
        f_1^{p, q}: \bigoplus_{|s|=p} Kh^q(F|_{s\times \mathbb{B}^{n-k}}) \longrightarrow \bigoplus_{|s|=p} Kh^q(G|_{s\times \mathbb{B}^{n-k}}).
    \end{equation*}
    \item As the spectral sequences collapse at the $E_{k+1}$ page, the map $f_{k+1} = f_\infty$ coincides with the associated graded morphism $\mathrm{Gr}_{p,q}(\Phi_*)$ on the total Khovanov homology
    \begin{equation*}
        f_\infty^{p,q}: \mathrm{Gr}_{p,q} Kh(F) \longrightarrow \mathrm{Gr}_{p,q} Kh(G).
    \end{equation*}
\end{enumerate}
\end{theorem}

\begin{proof}
By the definition of a fibered matching, there exists a decomposition $p = \mathrm{id}_{\mathbb{B}^k_{\mathrm{base}}} \times \mathrm{id}_{\mathbb{B}^{n-k}_{\mathrm{fiber}}}$. Since any coordinate permutation preserves the Hamming weight of resolution vertices, it follows that $|p(s, t)| = |s| + |t|$. Consequently, the induced chain map $f$ maps each $(p,q)$-bigraded component $K_F^{p,q}$ directly into $K_G^{p,q}$. By Theorem \ref{theorem:alegbraic_ss}, this yields a spectral sequence morphism $\{f_r\}_{r \ge 0}$ 
\end{proof}

\begin{definition}
A fibered matching $\Phi = (p, \eta): F \to G$ in $\fun_B$ is said to be a \textit{fibered equivalence} if $\eta: F \Rightarrow G \circ p$ is a natural isomorphism. 
\end{definition}

\begin{corollary}
If $\Phi = (p, \eta)$ is a fibered equivalence in $\fun_B$, then the induced map 
\begin{equation*}
    f_r^{p}: E_r^{p,\bullet}(F) \longrightarrow E_r^{p,\bullet}(G)
\end{equation*}
is an isomorphism for all $r \ge 1$. Consequently, $\Phi$ induces an isomorphism on the total Khovanov homology
\begin{equation*}
    \Phi_*: \kh(F) \xrightarrow{\cong} \kh(G).
\end{equation*}
\end{corollary}
\begin{proof}
Since $\Phi$ is a fibered equivalence, $\eta_s$ induces an isomorphism on local fiber homologies for each $s$, making $f_1$ an isomorphism on the $E_1$ page. By the convergence of spectral sequences, this isomorphism leads to all higher pages $r \ge 1$ and eventually to the total Khovanov homology.
\end{proof}

\subsection{Reidemeister moves}

This section is devoted to proving the Reidemeister invariance of the spectral sequence. Our approach leverages the homotopy equivalence of Bar-Natan complexes and the algebraic reduction of contractible subcomplexes.

\subsubsection{Reidemeister I invariance}

\begin{definition}[Reidemeister I Move]
A pro-tangle $F \in \fun_B^{(n)}$ is said to admit a \textit{Reidemeister I (R1) move} at the $i$-th crossing if there exists a natural isomorphism
\begin{equation*}
    \Psi: F^{(i)}_0 \xrightarrow{\cong} F^{(i)}_1 \sqcup S^1 \quad (\text{or } F^{(i)}_1 \xrightarrow{\cong} F^{(i)}_0 \sqcup S^1),
\end{equation*}
where $S^1\in \fun_{\emptyset}^{(0)}$ denotes a disjoint circle component. In this configuration, the crossing map $\Phi_i:F^{(i)}_0 \to F^{(i)}_1$ locally corresponds to the birth or death cobordism of the circle $S^1$.
\end{definition}

In what follows, we consider the case $F^{(i)}_0 \xrightarrow{\cong} F^{(i)}_1 \sqcup S^1$.

\begin{proposition}[R1 Invariance]
If $F \in \fun_B^{(n)}$ admits an R1 move at the $i$-th crossing, then the associated $k=1$ spectral sequence collapses at the $E_2$ page, yielding
\begin{equation*}
    E_2^{p,q}(F) \cong 
    \begin{cases} 
    \kh^q(G), & \text{if } (p,q) \text{ corresponds to the reduced index}; \\
    0, & \text{otherwise,}
    \end{cases}
\end{equation*}
where $G=F^{(i)}_1 \in \fun_B^{(n-1)}$ is the simplified pro-tangle. Consequently, $\kh(F) \cong \kh(G)$ in $\mathcal{A}$.
\end{proposition}

\begin{proof}
Bar-Natan's delooping lemma \cite{bar2007fast} guarantees that any object containing a closed circle $S^1$ is splittably isomorphic to a direct sum of two components without that circle up to grading shifts. Formally, there exists a canonical cobordism splitting isomorphism:
\begin{equation*}
    \iota = \begin{pmatrix} \iota_1 & \iota_x \end{pmatrix}: G\{1\} \oplus G\{-1\} \xrightarrow{\cong} G \sqcup S^1,
\end{equation*}
where $\iota_1$ and $\iota_x$ are induced by the birth (cap) cobordisms without and with an $x$-label respectively. 

Given the pro-tangle natural isomorphism $F^{(i)}_0 \cong G \sqcup S^1$ where $G = F^{(i)}_1$, the additivity of the TQFT functor $\overline{\Theta}$ translates this geometric split into the $E_1$ page. The differential $d_1: E_1^{0, \bullet} \to E_1^{1, \bullet}$ is specified by the crossing map $\overline{\Theta}(\Phi_i)$, where $\Phi_i: G \sqcup S^1 \to G$ is the elementary saddle morphic clearance of the circle. 

To determine the precise matrix representation of the differential, we pre-compose $d_1$ with the splitting isomorphism $\overline{\Theta}(\iota)$. Under our TQFT construction, evaluating the saddle composite on the respective direct summands yields
\begin{equation*}
    \overline{\Theta}(\Phi_i \circ \iota_1) = \mathrm{id}_{\overline{\Theta}(G)}, \quad \overline{\Theta}(\Phi_i \circ \iota_x) = 0.
\end{equation*}
Consequently, the differential $d_1$ is explicitly expressed by the row matrix:
\begin{equation*}
    d_1 \circ \overline{\Theta}(\iota) = \begin{pmatrix} \mathrm{id}_{\overline{\Theta}(G)} & 0 \end{pmatrix} : \overline{\Theta}(G)\{1\} \oplus \overline{\Theta}(G)\{-1\} \to \overline{\Theta}(G).
\end{equation*}

In the Abelian category $\mathcal{A}$, the sub-complex $[\overline{\Theta}(G)\{1\} \xrightarrow{\mathrm{id}} \overline{\Theta}(G)]$ forms a contractible pair whose homology vanishes identically. The remaining summand $\overline{\Theta}(G)\{-1\}$ corresponds precisely to the kernel component $\ker(d_1) / \mathrm{im}(0)$, which survives to the $E_2$ page. Since the base resolution space $\cube^1_{\mathrm{base}}$ is one-dimensional, the spectral sequence collapses at $E_2$, yielding $E_2^{\bullet,\bullet}(F) \cong \kh(G)$ up to grading shift. This confirms the R1 invariance.
\end{proof}

\subsubsection{Reidemeister II invariance}

From now on, for a pro-tangle $F \in \fun_B^{(n)}$ and choice of resolutions $\epsilon_i, \epsilon_j \in \{0, 1\}$ at distinct indices $i < j$, we denote the \textit{double smoothing} as the pro-tangle functor
\begin{equation*}
    F^{(i,j)}_{\epsilon_i,\epsilon_j} :  \cube^{n-2} \longrightarrow \cob(B),
\end{equation*}
which is defined by the composition $F \circ \delta_{i,j; \,\epsilon_i,\epsilon_j}$. Here, $\delta_{i,j; \,\epsilon_i,\epsilon_j} :  \cube^{n-2} \to \cube^n$ is the joint face embedding functor that assigns the fixed resolution values $\epsilon_i$ and $\epsilon_j$ to the $i$-th and $j$-th coordinates respectively:
\begin{equation*}
    \delta_{i,j; \,\epsilon_i,\epsilon_j}(x_1, \dots, x_{n-2}) = (x_1, \dots, \epsilon_i, \dots, \epsilon_j, \dots, x_{n-2}).
\end{equation*}

\begin{definition}[Reidemeister II Move]
A pro-tangle $F \in \fun_B^{(n)}$ is said to admit a \textit{Reidemeister II (R2) move} at crossings $i, j$ if there exists a pro-tangle $G \in \fun_B^{(n-2)}$ such that the local resolutions at these crossings satisfy the following natural isomorphisms:
\begin{enumerate}[label=$(\roman*)$]
    \item $F^{(i,j)}_{0,0} \cong G$ and $F^{(i,j)}_{1,1} \cong G$.
    \item $F^{(i,j)}_{0,1} \cong G \sqcup S^1$, where $S^1\in \fun_{\emptyset}^{(0)}$ is a disjoint circle component.
\end{enumerate}
In this configuration, the pro-tangle $F$ is said to be R2-equivalent to the resolution $F^{(i,j)}_{1,0}$.
\end{definition}

The Reidemeister II move involves the simultaneous resolution of two adjacent crossings, denoted as $i$ and $j$. In our setting, this corresponds to a spectral sequence fibered over the square base $\cube^2_{\mathrm{base}}$.

\begin{proposition}
Suppose the pro-tangle $F$ admits an R2 move at crossings $i, j$. The spectral sequence associated with the base $\cube^2_{\mathrm{base}}$ collapses at the $E_2$ page, yielding a natural isomorphism of total homology
\begin{equation*}
    \kh(F) \cong \kh(F^{(i,j)}_{1,0}).
\end{equation*}
\end{proposition}

\begin{proof}
The $E_1$ page is a square complex in the Abelian category $\mathcal{A}$ formed by the objectwise homologies at the four vertices of $\cube^2_{\mathrm{base}}$. By Bar-Natan's delooping lemma \cite{bar2007fast}, the object at vertex $(0,1)$ admits a canonical splitting isomorphism $\iota = \begin{pmatrix} \iota_1 & \iota_x \end{pmatrix} : \overline{\Theta}(G)\{1\} \oplus \overline{\Theta}(G)\{-1\} \xrightarrow{\cong} \overline{\Theta}(F^{(i,j)}_{0,1})$. Substituting this decomposition yields the following grid on the $E_1$ page:
\begin{equation*}
\begin{tikzcd}[column sep=large, row sep=large]
    \overline{\Theta}(G) \arrow[d, "d_1^{v,0}"'] \arrow[r, "d_1^{h,0}"] & \overline{\Theta}(F^{(i,j)}_{1,0}) \arrow[d, "d_1^{v,1}"] \\
    \overline{\Theta}(G)\{1\} \oplus \overline{\Theta}(G)\{-1\} \arrow[r, "d_1^{h,1}"'] & \overline{\Theta}(G),
\end{tikzcd}
\end{equation*}
where the internal components of the differentials are governed by elementary birth and death saddles. 

To determine the precise algebraic behavior without ambiguity, we analyze the differentials by composing them with the splitting isomorphism $\iota$. For the vertical differential $d_1^{v,0}$ (the birth of $S^1$), post-composing with the projection $\iota^{-1}$ or evaluating the standard cobordism matrix shows it acts as a split monomorphism:
\begin{equation*}
    \overline{\Theta}(\iota)^{-1} \circ d_1^{v,0} = \begin{pmatrix} 0 \\ \mathrm{id}_{\overline{\Theta}(G)} \end{pmatrix} : \overline{\Theta}(G) \to \overline{\Theta}(G)\{1\} \oplus \overline{\Theta}(G)\{-1\}.
\end{equation*}
Conversely, for the horizontal differential $d_1^{h,1}$ (the death of $S^1$), pre-composing with the splitting map $\iota$ yields a row matrix representing a split epimorphism:
\begin{equation*}
    d_1^{h,1} \circ \overline{\Theta}(\iota) = \begin{pmatrix} \mathrm{id}_{\overline{\Theta}(G)} & 0 \end{pmatrix} : \overline{\Theta}(G)\{1\} \oplus \overline{\Theta}(G)\{-1\} \to \overline{\Theta}(G),
\end{equation*}
where the vanishing component is forced by the boundary-relative ideal annihilation.

In the Abelian category $\mathcal{A}$, these explicit matrix representations induce a pairwise chain contraction across the grid via the following matching mechanisms:
\begin{itemize}
    \item \textit{Horizontal cancellation:} The horizontal differential $d_1^{h,1} \circ \overline{\Theta}(\iota) = \begin{pmatrix} \mathrm{id} & 0 \end{pmatrix}$ restricts to an isomorphism on the first direct summand $\overline{\Theta}(G)\{1\} \subset \overline{\Theta}(F^{(i,j)}_{0,1})$. This maps it bijectively onto the target vertex $(1,1)$, which is $\overline{\Theta}(G)$. Consequently, the sub-complex $[\overline{\Theta}(G)\{1\} \xrightarrow{\cong} \overline{\Theta}(G)]$ forms a contractible pair, yielding zero homology at both vertex $(1,1)$ and the upper summand of vertex $(0,1)$.
    \item \textit{Vertical cancellation:} The vertical differential $\overline{\Theta}(\iota)^{-1} \circ d_1^{v,0} = \begin{pmatrix} 0 \\ \mathrm{id} \end{pmatrix}$ embeds the source vertex $(0,0)$, which is $\overline{\Theta}(G)$, injectively into the second direct summand $\overline{\Theta}(G)\{-1\} \subset \overline{\Theta}(F^{(i,j)}_{0,1})$. Thus, the sub-complex $[\overline{\Theta}(G) \xrightarrow{\cong} \overline{\Theta}(G)\{-1\}]$ forms another independent contractible pair, forcing the mutual annihilation of vertex $(0,0)$ and the lower summand of vertex $(0,1)$.
\end{itemize}
Through these algebraic cancellations, the homological contributions of the three vertices $(0,0)$, $(0,1)$, and $(1,1)$ mutually annihilate on the $E_2$ page. The only surviving un-paired component is $\overline{\Theta}(F^{(i,j)}_{1,0})$ situated at vertex $(1,0)$. Thus, the $E_2$ page is concentrated entirely on the single column $p=1$:
\begin{equation*}
    E_2^{p,q} \cong \begin{cases} \kh^q(F^{(i,j)}_{1,0}), & \text{if } p=1; \\ 0, & \text{otherwise.} \end{cases}
\end{equation*}
Consequently, all higher page differentials $d_r$ ($r \ge 2$) vanish identically due to grading and support restrictions, and the spectral sequence collapses at $E_2$. We conclude that $\kh(F) \cong \kh(F^{(i,j)}_{1,0})$ up to appropriate topological grading shifts, completing the proof of R2 invariance.
\end{proof}

\subsubsection{Reidemeister III invariance}

\begin{definition}
A pro-tangle $F \in \fun_B^{(3)}$ is said to be of \textit{R3 configuration} if it is naturally isomorphic to the pro-tangle associated with a tangle of three pairwise intersecting strands, such that one strand is an over-strand relative to the other two.
\end{definition}

\begin{definition}
A pro-tangle $F \in \fun_B$ \textit{admits an R3 move} if there exists a contractible disk $\mathbb{D}$ such that the restriction of the functor
\begin{equation*}
    F|_{\mathbb{D}}: \cube^3 \to \cob(B),\quad F|_{\mathbb{D}}(v) = F(v)\cap \mathbb{D}
\end{equation*}
is of R3 configuration.
\end{definition}

\begin{definition}[Reidemeister III Move]
Suppose pro-tangle $F \in \fun_B^{(n)}$ admits an R3 move at crossings $\{i, j, k\}$ in $\mathbb{D}$. We say that $F$ is \textit{R3-equivalent} to a pro-tangle $G \in \fun_B^{(n)}$ if the local connectivity of $G$ within $\mathbb{D}$ differs from that of $F$ only by an antipodal permutation on the six boundary labels of $\partial \mathbb{D}$.
\end{definition}

\begin{figure}[H]
  \centering
  \includegraphics[width=0.4\textwidth]{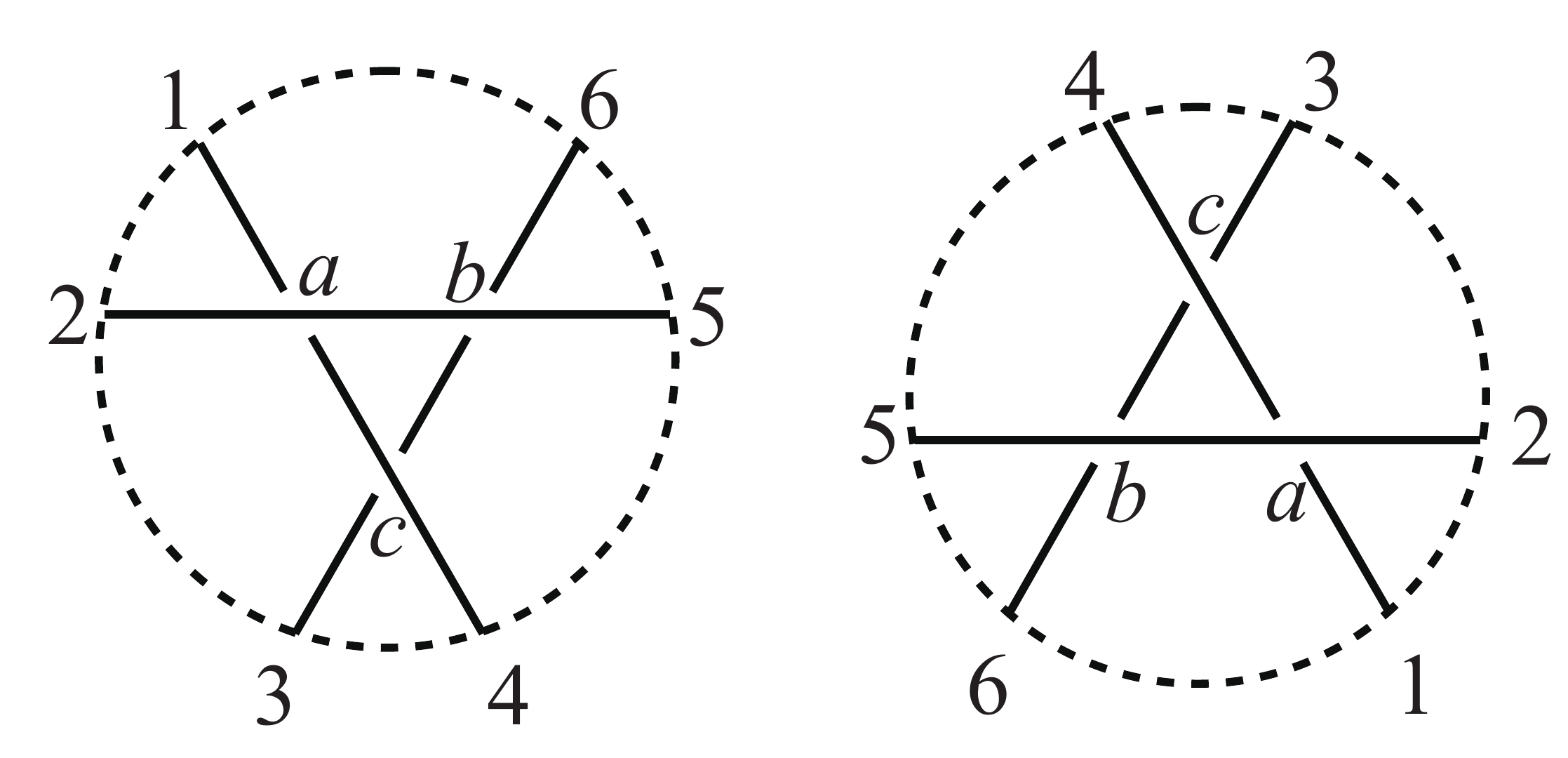}
  \caption{The antipodal permutation on the six boundary labels and the corresponding crossings.}\label{figure:antipodal}
\end{figure}

The Reidemeister III move is a local transformation involving three strands and three crossings, corresponding to a $3$-dimensional cube of resolutions $\cube^3_{\mathrm{base}}$. By the construction of the Reidemeister III move, there is a fibered matching $\Phi = (p, \eta): F \to G$, with the base map $f_{\mathrm{base}}: \cube^3_{\mathrm{base}} \to \cube^3_{\mathrm{base}}$ defined by $a \mapsto a$, $b \mapsto b$, and $c \mapsto c$, as illustrated in Figure \ref{figure:antipodal}. In this case, the map $p$ can be taken to be the identity, by choosing an appropriate matching of the crossings.

\begin{proposition}
Suppose the pro-tangles $F$ and $G$ in $\fun_B^{(n)}$ are R3-equivalent. Then there exists a morphism of spectral sequences $f_r: E_r(F) \to E_r(G)$ which becomes an isomorphism for $r \geq 2$. In particular, we have $\kh(F)\cong \kh(G)$.
\end{proposition}

\begin{proof}
The proof proceeds by analyzing the induced morphism across the eight resolution states of the base cube $\cube^3_{\mathrm{base}}$. For each $s \in \cube^3_{\mathrm{base}}$, let $(\mathcal{C}(F|_s), d_{\mathrm{fiber}})$ and $(\mathcal{C}(G|_{s}), d_{\mathrm{fiber}})$ be the fiber complexes. The fibered matching $\Phi$ induces a chain map between these local complexes
\begin{equation*}
    \Phi|_s: (\mathcal{C}(F|_s), d_{\mathrm{fiber}}) \longrightarrow (\mathcal{C}(G|_{s}), d_{\mathrm{fiber}}).
\end{equation*}
By passing to the fiber homology, we obtain a bi-homogeneous morphism on the $E_1$ page
\begin{equation*}
    f_1^{p,q}: \bigoplus_{|s|=p} Kh^q(F|_s) \longrightarrow \bigoplus_{|s|=p} Kh^q(G|_{s}).
\end{equation*}
At this stage, although $f_1$ may not be a vertex-wise isomorphism for all $s$, the total complex of the $3$-cube allows for a Gaussian elimination procedure analogous to the standard simplification of the Bar-Natan complex for R3 (see \cite[§5]{bar2005khovanov}; \cite[§4]{clark2009fixing}). Specifically, the contractible subcomplexes arise from pairs of resolution states $(s, s') \in \cube^3_{\mathrm{base}} \times \cube^3_{\mathrm{base}}$ that differ by a single crossing resolution, forming canceling pairs of acyclic subcomplexes which are systematically eliminated, leaving the remaining homological terms isomorphic. This reduction ensures that the induced map on the second page
\begin{equation*}
    f_2: E_2(F) \xrightarrow{\cong} E_2(G)
\end{equation*}
is a strict isomorphism. By the comparison theorem for spectral sequences, the isomorphism $f_2$ propagates to all higher pages $r \ge 2$ and eventually to the total Khovanov homology.
\end{proof}

\section{TQFT and Khovanov homology of pro-tangles}\label{section:pro-tangle_tqft}

This section is devoted to the TQFT formulation of pro-tangles. Furthermore, we explore several concrete applications that leverage this construction to facilitate the computation of Khovanov homology for some tangles.
Unless otherwise specified, all pro-tangles discussed herein are assumed to be oriented with a prescribed orientation.

\subsection{TQFT of pro-tangles}

Let $\mathbb{F}$ be the base field. From now on, unless otherwise specified, all tensor products $\otimes$ are assumed to be taken over $\mathbb{F}$.

\subsubsection{TQFT construction}

Let $V = \mathbb{F}[x]/(x^2)$ be the Frobenius algebra over $\mathbb{F}$. We define a representation of tangles via a functorial construction that generalizes the traditional $(1+1)$-dimensional TQFT to a relative setting involving both closed and open components. The categorification of tangles is formally realized through a symmetric monoidal functor $\mathcal{Z}: \cob(B) \to \mathbf{Mod}_V$, which maps the topological category of relative cobordisms into the category of finitely generated graded $V$-modules. 

On the level of objects, let $M$ denote a general $1$-manifold embedded in a surface with fixed boundary $B = \partial M$. Such an $M$ represents a resolution consisting of disjoint unions of closed circles and open arcs anchored on $B$. The functor assigns to each constituent component a $V$-module: a closed circle $S^1$ is mapped to the regular module $\mathcal{Z}(S^1) = V = \mathbb{F}[x]/(x^2)$, while an open arc $A$ is mapped to the augmentation ideal $W = xV$. For a general object $M$, the construction $\mathcal{Z}(M)$ is defined as the tensor product over $\mathbb{F}$ of the modules corresponding to its components, resulting in a module of the form $V^{\otimes n} \otimes W^{\otimes m}$, which inherits a natural $V$-module structure.

The morphism assignment of $\mathcal{Z}$ maps topological cobordisms to $V$-linear homomorphisms. The local saddle moves are encoded by the Frobenius structure of $V$ and its action on $W$. A merge saddle $S^1 \sqcup S^1 \to S^1$ corresponds to the multiplication
\begin{equation*}
  m: V \otimes V \to V,\quad a\otimes b \mapsto ab,
\end{equation*}
while the interaction between a circle and an arc $S^1 \sqcup A \to A$ is sent to the module action
\begin{equation*}
  \mu: V \otimes W \to W,\quad a\otimes b \mapsto ab.
\end{equation*}
Conversely, the split saddles are described by the comultiplication
\begin{equation*}
  \Delta: V \to V \otimes V,\quad \left\{\begin{array}{ll}
  1 \mapsto 1 \otimes x + x \otimes 1, \\
  x \mapsto x\otimes x
\end{array}\right.
\end{equation*}
and the mixed co-action
\begin{equation*}
  \Delta_W: W \to V \otimes W,\quad x\mapsto x\otimes x.
\end{equation*}
The unit $\eta: \mathbb{F} \to V$ maps $1$ to $1$, and the counit $\epsilon: V \to \mathbb{F}$ is given by
\begin{equation*}
  \epsilon(1) = 0,\quad \epsilon(x) =1.
\end{equation*}
Besides, a crossing between two arcs corresponds to a saddle cobordism inducing a morphism
\begin{equation*}
  \sigma: W \otimes W \to W \otimes W,\quad \sigma(x\otimes x)=0
\end{equation*}
The vanishing of the arc-crossing morphism, $\sigma(x \otimes x) = 0$, captures a kind of algebraic shadowing of the boundary constraints.

\begin{table}[ht]
\centering
\caption{The TQFT construction $\mathcal{Z}: \cob(B) \to \mathbf{Mod}_V$ for pro-tangles}
\label{tab:tqft_construction}
\begin{tabular}{lll}
\toprule
\textbf{Topological object} & \textbf{Algebraic assignment} & \textbf{Module category description} \\ \midrule
Closed Circle $S^1$ & $\mathcal{Z}(S^1) = V = \mathbb{F}[x]/(x^2)$ & Regular $V$-module (Algebra) \\
Open Arc $A$ & $\mathcal{Z}(A) = W = xV$ & Augmentation ideal (Submodule) \\
General $1$-manifold $M$ & $\mathcal{Z}(M) \cong V^{\otimes n} \otimes W^{\otimes m}$ & Tensor product of constituents \\ \midrule
Merge (Circle-Circle) & $m: V \otimes V \to V$ & Frobenius multiplication \\
Merge (Circle-Arc) & $\mu: V \otimes W \to W$ & $V$-module action on $W$ \\
Split (Circle) & $\Delta: V \to V \otimes V$ & Frobenius comultiplication \\
Split (Arc) & $\Delta_W: W \to V \otimes W$ & Mixed module co-action \\ 
Arc-crossing & $\sigma: W \otimes W \to W \otimes W$ & Zero morphism \\ \midrule
Counit (Cup) & $\epsilon: V \to \mathbb{F}$ & Frobenius counit \\
Unit (Cap) & $\eta: \mathbb{F} \to V$ & Frobenius unit \\
Identity (Cylinder) & $\mathrm{id}_{\mathcal{Z}(M)}$ & Identity map \\ \bottomrule
\end{tabular}
\end{table}

\begin{remark}
The assignment of $V=\mathbb{F}[x]/(x^2)$ to circles is standard \cite{khovanov2000categorification}. 
Our assignment of $W=xV$ to open arcs is a natural adaptation: it is the minimal $V$-submodule of $V$ corresponding to the augmentation ideal. 
While \cite{khovanov2002functor} captures boundary constraints globally by viewing tangle complexes as functors taking values in bimodules over a non-commutative boundary algebra, this
construction localizes this topological enforcement by directly assigning the simple $V$-submodule $W$ to each individual open arc. Crucially, the Frobenius module structure $(W,\mu,\Delta_W)$ on open arcs departs fundamentally from the standard open-closed TQFT construction of \cite{lauda2008open}. 
While \cite{lauda2008open} models open and closed state spaces as distinct Frobenius algebras linked by transition morphisms, our open arc space $W=xV$ is a strict ideal of the Frobenius algebra $V$.
\end{remark}

\begin{lemma}\label{lemma:compatible}
The triple $(W, \mu, \Delta_W)$, where $\mu: V \otimes W \to W$ is the module action and $\Delta_W: W \to V \otimes W$ is the mixed co-action, forms a boundary-compatible Frobenius module over $V$. Specifically, the following identities hold:
\begin{enumerate}[label=$(\roman*)$]
    \item Unitality: $\mu \circ (\eta \otimes \mathrm{id}_W) = \mathrm{id}_W$.
    \item Frobenius compatibility: $\Delta_W \circ \mu = (m \otimes \mathrm{id}_W) \circ (\mathrm{id}_V \otimes \Delta_W)$.
    \item Boundary annihilation: $(\epsilon \otimes \mathrm{id}_W) \circ \Delta_W = \mathrm{id}_W$.
\end{enumerate}
\end{lemma}

\begin{proof}
Since $W = xV$ is a one-dimensional vector space over $\mathbb{F}$ spanned uniquely by the single generator $x$, it suffices to verify the proposed identities on this basis element. First, the unitality of the module action follows directly from $\mu(1 \otimes x) = 1 \cdot x = x$, which recovers $\mathrm{id}_W$. 

To establish the Frobenius compatibility, we evaluate both sides of the identity on the domain generators $a \otimes x \in V \otimes W$, where $a \in \{1, x\}$ forms the standard basis of $V$. 
For the identity element $a=1$, the left-hand side yields 
\begin{equation*}
    \Delta_W(\mu(1 \otimes x)) = \Delta_W(x) = x \otimes x \in V \otimes W.
\end{equation*}
Meanwhile, evaluating the right-hand side via the mixed co-action and the regular multiplication $m$ gives:
\begin{equation*}
    \big((m \otimes \mathrm{id}_W) \circ (\mathrm{id}_V \otimes \Delta_W)\big)(1 \otimes x) = (m \otimes \mathrm{id}_W)(1 \otimes x \otimes x) = m(1 \otimes x) \otimes x = x \otimes x.
\end{equation*}
For the generator $a=x$, both sides vanish identically: the left-hand side triggers $\mu(x \otimes x) = 0$ due to the ideal nilpotency $x^2 = 0$ in $V$, while the right-hand side yields $(m \otimes \mathrm{id}_W)(x \otimes x \otimes x) = m(x \otimes x) \otimes x = 0 \otimes x = 0$.

Finally, the boundary annihilation property (iii) is verified by applying the Frobenius counit $\epsilon$ to the left-hand factor of the co-action image. Evaluating on the generator $x \in W$, we obtain
\begin{equation*}
    (\epsilon \otimes \mathrm{id}_W)\Delta_W(x) = \epsilon(x) \cdot x.
\end{equation*}
Given the standard Khovanov evaluation $\epsilon(x)=1$, this expression simplifies directly to $1 \cdot x = x$, thereby recovering the identity map on $W$ and completing the proof.
\end{proof}

\begin{proposition}
The assignment $\mathcal{Z}$ extends to a symmetric monoidal functor
\[
\mathcal{Z}: \cob(B) \to \mathbf{Vec}_{\mathbb{F}}
\]
with values in the category of vector spaces.
\end{proposition}

\begin{proof}
The functoriality requires $\mathcal{Z}$ to respect composition and identities. 
For any object $M$, the identity morphism (the product cylinder $M \times [0,1]$) is mapped to $\mathrm{id}_{\mathcal{Z}(M)}$ due to the unitality of $V$ and $W$ established in Lemma~\ref{lemma:compatible}.

The well-definedness of cobordism composition relies on showing that $\mathcal{Z}$ satisfies the defining local relations of Bar-Natan's cobordism category $\cob(B)$ as follows:
\begin{itemize}
    \item \textit{$S$ and T relations:} On internal closed components, the evaluation of disjoint spheres and tori uniquely involves the closed Frobenius algebra $V$. The invariance under these relations reduces entirely to the standard closed TQFT construction. Specifically, $\epsilon \circ \eta = 0$ and $\epsilon \circ m \circ \Delta \circ \eta = 2 \cdot \mathrm{id}_{\mathbb{F}}$ hold automatically.

    \item \textit{$4Tu$ relation:} For all local configurations, handle-slide invariance holds case-by-case. For purely closed components, the relation reduces to the classical Frobenius identity for $V$. For configurations with one open arc, the relation is guaranteed by the three-term Frobenius module identity:
    \[
    (\mathrm{id}_V \otimes \mu) \circ (\Delta \otimes \mathrm{id}_W) = \Delta_W \circ \mu = (m \otimes \mathrm{id}_W) \circ (\mathrm{id}_V \otimes \Delta_W).
    \]
    For configurations with two or more open arcs, the relation holds trivially because the arc-crossing map $\sigma$ vanishes by definition. This symmetry ensures that reordering elementary saddles across distinct components yields identical linear endomorphisms, satisfying the $4Tu$ relation.
\end{itemize}

Furthermore, the arc-crossing morphism is defined as the zero map $\sigma = 0$, which is consistent with the $V$-module structure and vanishes on the generator $x \otimes x$ of $W \otimes W$.

Since all topological isotopies and local surgery relations are satisfied, the functor $\mathcal{Z}$ is well-defined. The symmetric monoidal property is directly inherited from the standard tensor product structure on $\mathbf{Vec}_{\mathbb{F}}$.
\end{proof}

\subsubsection{Quantum grading}

Let $C = \bigoplus_{i \in \mathbb{Z}} C^i$ be a graded vector space of finite type. The \textit{Poincaré series} is defined as the formal Laurent polynomial
\begin{equation*}
    P(C; q) = \sum_{i \in \mathbb{Z}} \left( \rank C^{i} \right) q^i.
\end{equation*}
In the context of the TQFT functor $\mathcal{Z}$, when the grading corresponds to the intrinsic internal grading of the $V$-modules, the Poincaré series represents the \textit{quantum dimension} or graded dimension of the object. The grading on the base algebra $V = \mathbb{F}[x]/(x^2)$ is defined such that $\deg(1) = 1$ and $\deg(x) = -1$. Consequently, the arc-module $W = xV$ is concentrated in degree $-1$. We have $P(V; q) = q + q^{-1}$ and $P(W; q) = q^{-1}$.

For a bigraded vector space $C = \bigoplus_{i,j \in \mathbb{Z}} C^{i,j}$, where $i$ denotes the height (homological grading) and $j$ denotes the quantum grading. The Poincaré series generalizes to the bivariate form
\begin{equation*}
    P(C; t, q) = \sum_{i \in \mathbb{Z}} \sum_{j \in \mathbb{Z}} \left( \rank C^{i,j} \right) t^i q^j.
\end{equation*}
By setting $t = -1$, the Poincaré series recovers the (unnormalized) Jones polynomial of the associated link, $J(q) = P(\kh; -1, q)$, where $\kh$ is the Khovanov homology and $P(\kh; -1, q)$ is the corresponding Euler characteristic.

\begin{definition}
Let $C = \bigoplus_{i,j \in \mathbb{Z}} C^{i,j}$ be a bigraded vector space. For any $h, k \in \mathbb{Z}$, the \textit{height shift} $[h]$ and the \textit{quantum shift} $\{k\}$ is defined as the endofunctors acting on the bigraded components such that
\begin{equation*}
    (C[h]\{k\})^{i,j} = C^{i-h, \, j-k}.
\end{equation*}
In terms of the Poincaré series, this shift corresponds to the monomial multiplication
\begin{equation*}
    P(C[h]\{k\}; t, q) = t^h q^k P(C; t, q).
\end{equation*}
\end{definition}

For a pro-tangle $F \in \fun^{(n)}_{B}$ and a vertex $v \in \cube^n$, suppose the resolution $F(v)$ consists of $n$ closed circles and $m$ open arcs. The graded dimension of the associated $V$-module $\mathcal{Z}(F(v))$ is given by
\begin{equation*}
    \dim_q \left( \mathcal{Z}(F(v)) \right) = (q + q^{-1})^n q^{-m}.
\end{equation*}
For an oriented pro-tangle $F$ with a fixed orientation, let $n_+$ denote the number of positive crossings and $n_-$ the number of negative crossings. The normalized Khovanov complex is given by
\begin{equation*}
    \mathcal{K}(F) = \mathcal{C}(F)[-n_-]\{n_+ - 2n_-\},
\end{equation*}
where $\mathcal{C}(F)$ is the Khovanov complex. The height shift tracks the number of $1$-resolutions in the hypercube of resolutions. Each differential $d$ in the Khovanov complex is a homogeneous map of degree $(1, 0)$ with respect to the $(h, q)$ bi-grading. 

\begin{remark}
For a generator $z$ in the Khovanov complex $\mathcal{K}^{i}(F)$, its height is $i$, and its quantum grading is given by
\begin{equation*}
Q(z) = i + n_{+} - n_{-} + \theta(z).
\end{equation*}
Here, $\theta: TV \otimes TW \to \mathbb{Z}$ is defined by $\theta(1) = 1$ and $\theta(x) = -1$, and satisfies the property that $\theta(ab) = \theta(a) + \theta(b)$ for all $a, b\in TV \otimes TW$. Equipped with the quantum grading, the Khovanov homology is a bigraded vector space, with the first grading coming from the dimension and the second coming from the quantum grading.
\end{remark}

From a categorical perspective, the quantum grading is formally characterized by the Poincaré series on the associated graded category. The assignment defines a well-behaved additive and multiplicative homomorphism
\begin{equation*}
    \chi_q: K_0(\mathbf{Mod}_V) \longrightarrow \mathbb{F}[q^{\pm 1}],
\end{equation*}
which maps each isomorphism class $[M]$ in the Grothendieck group to its corresponding graded dimension. This homomorphism ensures that the topological operations of disjoint unions and cobordism gluing are faithfully represented by the ring operations in $\mathbb{F}[q^{\pm 1}]$.

\subsection{Hopf clasp}

\subsubsection{Algebraic spectral sequence for Hopf clasp}

The Hopf clasp is the fundamental local clasping unit in low‑dimensional topology, serving as a core building block for constructing knots and links and describing finite-type invariants and clasper calculus.

\begin{definition}[Hopf clasp]
A pro-tangle $F \in \fun_B^{(n)}$ is said to admit a \textit{Hopf clasp} at crossings $i, j$ if there exists a pro-tangle $G \in \fun_B^{(n-2)}$ such that the local resolutions at these crossings satisfy the following natural isomorphisms:
\begin{enumerate}[label=$(\roman*)$]
    \item $F^{(i,j)}_{0,1} \cong F^{(i,j)}_{1,0} \cong G$;
    \item $F^{(i,j)}_{0,0} \cong G \sqcup S^1$ or $F^{(i,j)}_{1,1} \cong G \sqcup S^1$, where $S^1$ denotes a disjoint circle component.
\end{enumerate}
\end{definition}

If $F^{(i,j)}_{0,0} \cong G \sqcup S^1$, we say $F$ admits a Hopf clasp of type \textup{(I)}, and if $F^{(i,j)}_{1,1} \cong G \sqcup S^1$, a Hopf clasp of type \textup{(II)}.

\begin{example}
Figure \ref{figure:local_joints} displays a local region of the tangle, where the connection is characterized by a Hopf clasp. From the perspective of the Hopf clasp, there are typically two types: one where both crossings are undercrossings, and the other where both crossings are overcrossings. We refer to these as type \textup{(I)} and type \textup{(II)}, respectively.
\begin{figure}[H]
  \centering
  \includegraphics[width=0.3\textwidth]{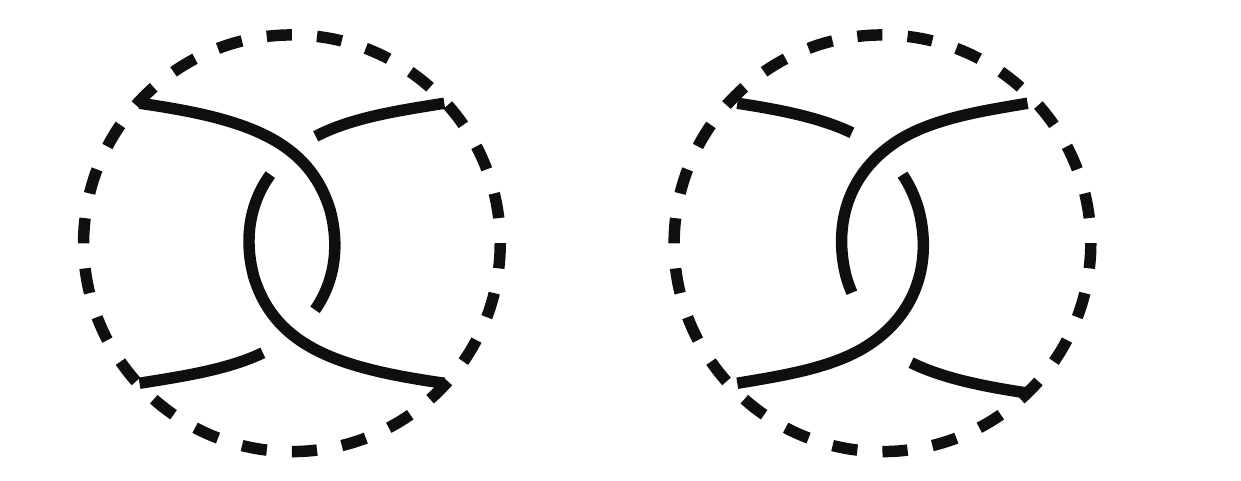}
  \caption{Hopf clasp: type \textup{(I)} on the left, type \textup{(II)} on the right.}\label{figure:local_joints}
\end{figure}
It is worth noting that the additional crossings in a Hopf clasp are always either all left-handed or all right-handed.
\end{example}

\begin{theorem}\label{theorem:hopf_clasp}
Let $F \in \fun_B^{(n)}$ be a pro-tangle admitting a Hopf clasp at crossings $i$ and $j$. Then, there exists a spectral sequence $\{E_r, d_r\}_{r \ge 1}$ that collapses at the $E_3$ page, i.e., $E_3 \cong E_\infty$, and converges to the total Khovanov homology $\kh(F)$. The $E_1$ page is concentrated in columns $p \in \{0, 1, 2\}$, and is given by
\begin{align*}
    E_1^{0, \bullet} &\cong \kh(G) \otimes V, \\
    E_1^{1, \bullet} &\cong \kh(G) \oplus \kh(G), \\
    E_1^{2, \bullet} &\cong \kh(F_{1,1})
\end{align*}
for the Hopf clasp of type \textup{(I)}, and 
\begin{align*}
    E_1^{0, \bullet} &\cong \kh(F_{0,0}), \\
    E_1^{1, \bullet} &\cong \kh(G) \oplus \kh(G), \\
    E_1^{2, \bullet} &\cong \kh(G) \otimes V
\end{align*}
for the Hopf clasp of type \textup{(II)}.
\end{theorem}

\begin{proof}
Consider the induced base-degree filtration $I^\bullet$ associated with the 2-dimensional base cube $\mathbb{B}^2_{\{i,j\}}$ for the crossings of the Hopf clasp. By Theorem \ref{theorem:spectral_sequence}, the spectral sequence collapses at the $E_3$-page, and its $E_1$-page is a direct sum of fiber Khovanov homologies over each fixed base resolution.
\end{proof}

\begin{example}
In this example, we consider the spectral sequence associated with a local Hopf clasp applied to the knot $8_{19}$. This example illustrates that such a spectral sequence does not necessarily collapse at the $E_2$-page. As shown in Figure \ref{figure:8_19_clasp}, we construct the Hopf clasp at the position marked by the dashed box.
\begin{figure}[H]
  \centering
  \includegraphics[width=0.6\textwidth]{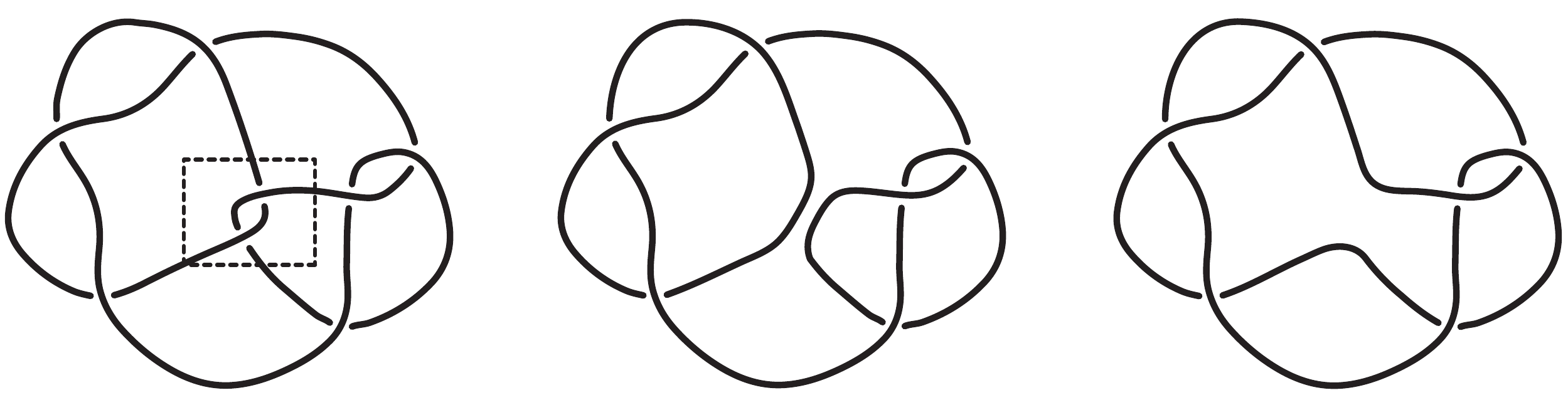}
  \caption{Illustration of the Hopf clasp. The left part shows the Hopf clasp of the $8_{19}$ knot within the dashed box, the center presents its $(0,0)$-smoothing, and the right part shows its $(0,1)$- and $(1,0)$-smoothings.}\label{figure:8_19_clasp}
\end{figure} 

We have computed the dimensions $\dim E_r^{p,q}$ of different pages of this spectral sequence via computer computations, where $\dim E_r^{p,q}$ denotes the dimension over $\mathbb{Q}$ at bidegree $(p,q)$ on the $r$-th page. The computation is summarized in Table \ref{tab:8_19_spectral_sequence}.
\begin{table}[htbp]
\centering
\caption{Dimensions of the spectral sequence pages for the Hopf clasp of $8_{19}$}
\label{tab:8_19_spectral_sequence}
\setlength{\tabcolsep}{3pt}
\renewcommand{\arraystretch}{1.05}
\newlength{\subboxwidth}
\setlength{\subboxwidth}{0.45\textwidth}

\begin{minipage}{\subboxwidth}
\centering
\textbf{$E_0$-page}\\[2pt]
\begin{tabular}{p{0.7cm} *{7}{p{0.7cm}}}
\toprule
$p \setminus q$ & 0 & 1 & 2 & 3 & 4 & 5 & 6 \\
\midrule
0 & 8 & 24 & 66 & 104 & 120 & 96 & 32 \\
1 & 8 & 60 & 192 & 296 & 264 & 144 & 32 \\
2 & 8 & 60 & 192 & 296 & 264 & 144 & 32 \\
\bottomrule
\end{tabular}
\end{minipage}
\qquad
\begin{minipage}{\subboxwidth}
\centering
\textbf{$E_1$-page}\\[2pt]
\begin{tabular}{p{0.7cm} *{7}{p{0.7cm}}}
\toprule
$p \setminus q$ & 0 & 1 & 2 & 3 & 4 & 5 & 6 \\
\midrule
0 & 2 & 0 & 2 & 2 & 1 & 2 & 1 \\
1 & 0 & 0 & 2 & 2 & 0 & 4 & 0 \\
2 & 0 & 0 & 2 & 2 & 0 & 4 & 0 \\
\bottomrule
\end{tabular}
\end{minipage}

\vspace{6pt}
\begin{minipage}{\subboxwidth}
\centering
\textbf{$E_2$-page}\\[2pt]
\begin{tabular}{p{0.7cm} *{7}{p{0.7cm}}}
\toprule
$p \setminus q$ & 0 & 1 & 2 & 3 & 4 & 5 & 6 \\
\midrule
0 & 2 & 0 & 1 & 1 & 1 & 1 & 1 \\
1 & 0 & 0 & 0 & 0 & 0 & 0 & 0 \\
2 & 0 & 0 & 1 & 1 & 0 & 1 & 0 \\
\bottomrule
\end{tabular}
\end{minipage}
\qquad
\begin{minipage}{\subboxwidth}
\centering
\textbf{$E_3$-page}\\[2pt]
\begin{tabular}{p{0.7cm} *{7}{p{0.7cm}}}
\toprule
$p \setminus q$ & 0 & 1 & 2 & 3 & 4 & 5 & 6 \\
\midrule
0 & 2 & 0 & 1 & 1 & 1 & 1 & 0 \\
1 & 0 & 0 & 0 & 0 & 0 & 0 & 0 \\
2 & 0 & 0 & 1 & 1 & 0 & 0 & 0 \\
\bottomrule
\end{tabular}
\end{minipage}
\end{table}

Indeed, the differential $d_2^{0,6}: E_{2}^{0,6} \to E_{2}^{2,5}$ is nonzero; it maps the generator at $(0,6)$ to that at $(2,5)$, resulting in the precise cancellation of these two generators on the $E_3$-page.
\end{example}

\subsubsection{Connected sum and Hopf sum}

\begin{definition}
Let $F \in \fun^{(n)}_B$ be a pro-tangle. A \textit{strand} of $F$ is defined as a functor 
\begin{equation*}
    \alpha: \cube^{n} \longrightarrow \cob(B_0)
\end{equation*}
where $B_0$ is a two-point set such that for each vertex $v \in \cube^n$, the resolution $\alpha(v)$ is a connected component (an open arc) of $F(v)$. Furthermore, for any edge $v \to w$ in $\cube^n$, the morphism $\alpha(v \to w)$ is the identity cylinder cobordism over the arc.
\end{definition}

\begin{definition}
Let $F \in \fun^{(n)}_B$ and $G \in \fun^{(m)}_{B'}$ be pro-tangles with disjoint boundaries. For a given pair of strands $R = (\alpha_F, \alpha_G)$, where $\alpha_F \in F$ and $\alpha_G \in G$, the \textit{connected sum} of $F$ and $G$ along $R$, denoted by $F \#_R G$, is the pro-tangle functor
\begin{equation*}
    F \#_R G: \cube^{n+m} \longrightarrow \cob(B \sqcup B')
\end{equation*}
defined such that it fits into a natural transformation $(\mathrm{id}, \eta): F \#_R G \Rightarrow F \sqcup G$. Here, the morphism $\eta$ corresponds to the saddle cobordism (crossing resolution) effectively joining the local arcs $\alpha_F$ and $\alpha_G$.
\end{definition}
The connected sum depends on the choice of the pair of strands in the corresponding Hopf clasp. When there is no ambiguity, we always use the notation $F_1 \# F_2$ to denote the connected sum associated with a fixed Hopf clasp. Figure \ref{figure:connected_sum} illustrates the connected sum of two tangles.
\begin{figure}[H]
  \centering
  \includegraphics[width=0.5\textwidth]{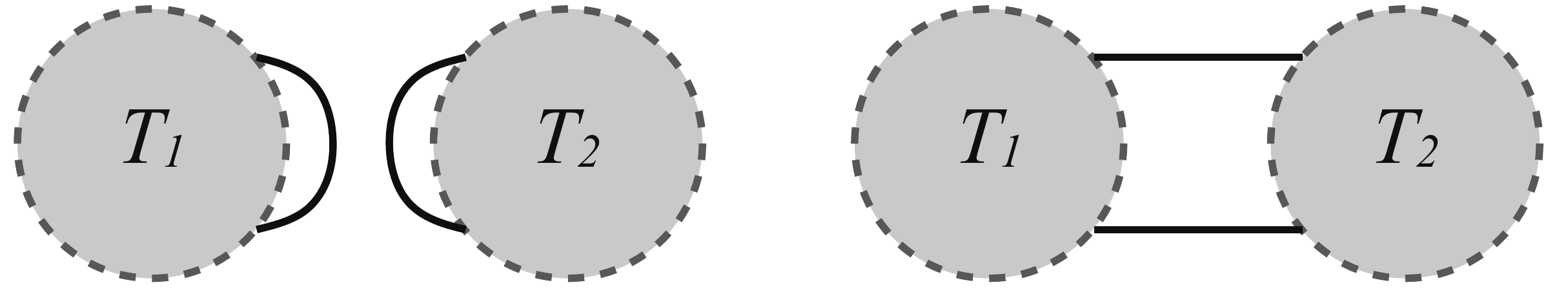}
  \caption{The illustration of the connected sum of tangles $T_1$ and $T_2$.}\label{figure:connected_sum}
\end{figure}

\begin{definition}
Let $F \in \fun^{(n)}_B$ and $G \in \fun^{(m)}_{B'}$ be pro-tangles with disjoint boundaries. Given a pair of strands $R = (\alpha_F, \alpha_G)$ from $F$ and $G$ respectively, the \textit{Hopf sum} along $R$, denoted by $F \boxtimes_R G$, is defined as the pro-tangle functor
\begin{equation*}
    F \boxtimes_{R} G: \cube^{2} \times \cube^{n+m} \longrightarrow \cob(B \sqcup B')
\end{equation*}
equipped with the following natural isomorphisms
\begin{enumerate}[label=$(\roman*)$]
    \item $(F\boxtimes_{R} G)^{(0,1)}_{0,1} \cong (F\boxtimes_{R} G)^{(0,1)}_{1,0} \cong F\#_{R} G$;
    \item In type \textup{(I)}: $(F\boxtimes_{R} G)^{(0,1)}_{0,0} \cong (F\#_{R} G) \sqcup S^1$ and $(F\boxtimes_{R} G)^{(0,1)}_{1,1} \cong F \sqcup G$. In type \textup{(II)}, $(F\boxtimes_{R} G)^{(0,1)}_{0,0} \cong F \sqcup G$ and $(F\boxtimes_{R} G)^{(0,1)}_{1,1} \cong (F\#_{R} G) \sqcup S^1$.
\end{enumerate}
\end{definition}

From the perspective of the Hopf clasp, there are typically two types: one where both crossings are undercrossings, and the other where both crossings are overcrossings. We refer to these as type \textup{(I)} and type \textup{(II)}, respectively (as illustrated in Figure \ref{figure:hopf_sum}).
\begin{figure}[H]
  \centering
  \includegraphics[width=0.5\textwidth]{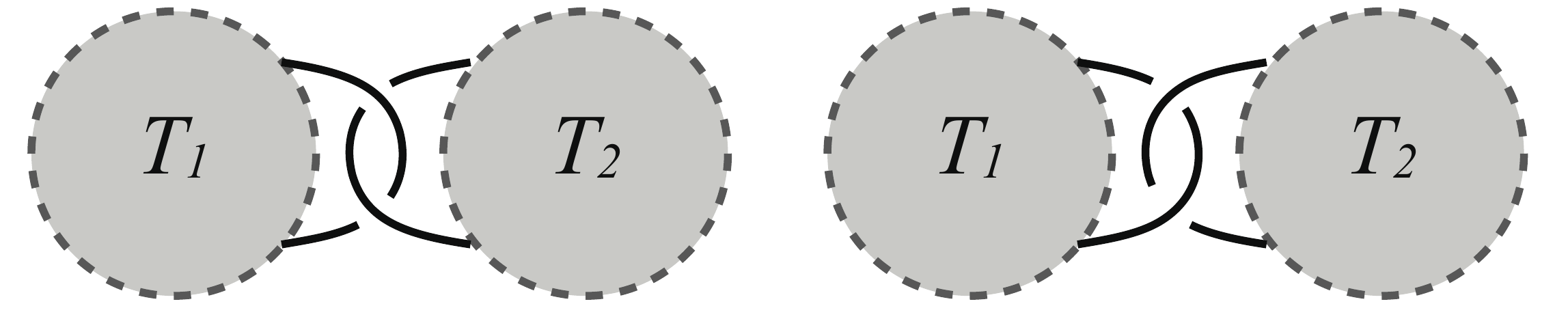}
  \caption{Hopf sum: type \textup{(I)} on the left, type \textup{(II)} on the right.}\label{figure:hopf_sum}
\end{figure}
In fact, the Hopf sum defined above yields a Hopf clasp, which functorially relates the connected sum $F \#_R G$ to the disjoint union $F \sqcup G$.

\begin{theorem}\label{theorem:hopf_sum}
Let $F \boxtimes_R G$ be the Hopf sum of pro-tangles $F \in \fun^{(n)}_B$ and $G \in \fun^{(m)}_{B'}$ along a pair of strands $R$. There exists a spectral sequence $E_r^{p,q}$ converging to the Khovanov homology of $F \boxtimes_R G$, i.e.,
\begin{equation*}
  E_1^{p,q} \Rightarrow \kh^{p+q+\delta}(F \boxtimes_R G),
\end{equation*}
where $\delta\in \{0,2\}$ is the number of left crossings among the two crossings at the joint. Moreover, the spectral sequence collapses at the $E_2$ page with the $E_1$ page given by
    \begin{align*}
        E_{1}^{0,q} &\cong \kh^{q}(F \#_R G) \otimes V, \\
        E_{1}^{1,q} &\cong \kh^{q}(F \#_R G) \oplus \kh^{q}(F \#_R G), \\
        E_{1}^{2,q} &\cong \kh^q(F \sqcup G) \cong \bigoplus_{i+j=q} \kh^{i}(F) \otimes \kh^{j}(G),\\
        E_{1}^{p,q} &= 0,\quad p \notin \{0,1,2\}
    \end{align*}
for type \textup{(I)} joint, and
    \begin{align*}
        E_{1}^{0,q} &\cong \kh^q(F \sqcup G) \cong \bigoplus_{i+j=q} \kh^{i}(F) \otimes \kh^{j}(G), \\
        E_{1}^{1,q} &\cong \kh^{q}(F \#_R G) \oplus \kh^{q}(F \#_R G), \\
        E_{1}^{2,q} &\cong \kh^{q}(F \#_R G) \otimes V,\\
        E_{1}^{p,q} &= 0,\quad p \notin \{0,1,2\}
    \end{align*}
for type \textup{(II)} joint.
\end{theorem}

\begin{proof}
Consider the induced base degree filtration associated with the $2$-dimensional base cube $\mathbb{B}^2$ corresponding to the two joint crossings of the Hopf sum $F \boxtimes_R G$. By construction, the $E_1$ page is concentrated in columns $p \in \{0, 1, 2\}$ and consists of the direct sum of the fiber Khovanov homologies computed at each fixed base resolution.

For a Hopf sum of type \textup{(I)}, evaluating the base resolutions yields:
\begin{itemize}
    \item At $s=(0,0)$, the fiber tangle is naturally isomorphic to $(F \#_R G) \sqcup S^1$. Applying the TQFT functor introduces a free tensor factor of the closed circle module $V$, yielding $E_1^{0, q} \cong \kh^q(F \#_R G) \otimes V$.
    \item At the intermediate vertices $s=(0,1)$ and $s=(1,0)$, the fiber tangles are both isomorphic to the connected sum $F \#_R G$, which provides the column splitting $E_1^{1, q} \cong \kh^q(F \#_R G) \oplus \kh^q(F \#_R G)$.
    \item At $s=(1,1)$, the fiber tangle de-clasps into the disjoint union $F \sqcup G$, yielding $E_1^{2, q} \cong \kh^q(F \sqcup G) \cong \bigoplus_{i+j=q} \kh^i(F) \otimes \kh^j(G)$ by the Künneth property of the Khovanov TQFT.
\end{itemize}
The assignment for type \textup{(II)} follows analogously by the symmetric distribution of the dual resolutions on the Boolean cube.

By Theorem \ref{theorem:hopf_clasp}, all higher differentials $d_r$ of the spectral sequence vanish for every $r \ge 3$. To prove strict collapse at the $E_2$-page, the only possibly non-trivial higher differential is the boundary map $d_2 \colon E_2^{0,q} \to E_2^{2,q-1}$. Let $D = \delta^0 + \delta^1$ be the total differential of the underlying bicomplex. A cohomology class in $E_2^{0,q}$ is represented by a combined chain $z = x_0 + x_1$, where $x_0 \in K^{0,q}$ and $x_1 \in K^{1,q-1}$, satisfying the algebraic cycle conditions $\delta^0(x_0) = 0$ and $\delta^1(x_0) + \delta^0(x_1) = 0$. The $E_2$ page differential is defined via the zigzag map $d_2([x_0]) = [\delta^1(x_1)] \in E_2^{2,q-1}$.

Applying $\delta^1$ to the cycle condition, the anti-commutativity of the bicomplex ($\delta^1\delta^0 = -\delta^0\delta^1$) implies that $\delta^0(\delta^1(x_1)) = 0$, confirming that $\delta^1(x_1)$ is a well-defined $\delta^0$-cycle representing a valid class in $E_1^{2,q-1}$. To prove that $[\delta^1(x_1)]$ vanishes in the second page quotient, we decompose the intermediate chain $x_1$ along the direct sum splitting of the middle layer, writing $x_1 = (x_{01}, x_{10}) \in K^{(0,1),q-1} \oplus K^{(1,0),q-1}$. Under this splitting, the differential evaluates to $\delta^1(x_1) = \xi(x_{01}) - \xi(x_{10})$, where $\xi$ is the chain map induced by the saddle cobordism.

Crucially, because the Hopf sum is defined along a designated pair of strands $R = (\alpha_F, \alpha_G)$, the disjoint circle $S^1$ at $s=(0,0)$ is positioned symmetrically with respect to both joint crossings. The edge cobordisms $(0,0) \to (0,1)$ and $(0,0) \to (1,0)$ merge this circle into the exact same local arcs of the connected sum $F \#_R G$. This geometric regularization guarantees that the components of the horizontal differential $\delta^1(x_0) = (\delta^1_i(x_0), \delta^1_j(x_0))$ are identical at the chain level, i.e., $\delta^1_i(x_0) = \delta^1_j(x_0)$. In view of the relation $\delta^1(x_0) + \delta^0(x_1) = 0$, we obtain the component-wise identifications:
\begin{equation*}
    \delta^0(x_{01}) = -\delta^1_i(x_0) = -\delta^1_j(x_0) = \delta^0(x_{10}).
\end{equation*}
By the crossing-reversal symmetry of the Hopf sum configuration, there exists a canonical chain isomorphism $\psi: K^{(0,1),\bullet} \xrightarrow{\cong} K^{(1,0),\bullet}$ that maps the element $x_{01}$ to $x_{10}$. We can therefore define an element $y \in K^{1,q-1}$ via the ordered splitting
\begin{equation*}
    y = (x_{01} - x_{10}, \, 0) \in K^{(0,1),q-1} \oplus K^{(1,0),q-1}.
\end{equation*}
Evaluating the vertical differential $\delta^0$ on $y$ yields
\begin{equation*}
    \delta^0(y) = (\delta^0(x_{01}) - \delta^0(x_{10}), \, 0) = (-\delta^1_i(x_0) - (-\delta^1_i(x_0)), \, 0) = (0,0).
\end{equation*}
Hence, $y \in \ker(\delta^0)$ represents a valid cohomology class $[y] \in E_1^{1,q-1}$.

Finally, evaluating the differential $d_1$ on $[y]$ gives
\begin{equation*}
    d_1([y]) = d_1([(x_{01} - x_{10}, 0)]) = [\xi(x_{01} - x_{10}) - \xi(0)] = [\xi(x_{01}) - \xi(x_{10})] = [\delta^1(x_1)].
\end{equation*}
This identity shows that the cycle $\delta^1(x_1)$ is a strict $d_1$-coboundary in $E_1^{2,q-1}$. Passing to the second-page cohomology quotient, we conclude that $d_2([x_0]) = [\delta^1(x_1)] = 0 \in E_2^{2,q-1}$. Since all higher differentials vanish, the spectral sequence collapses at the $E_2$ page, completing the proof.
\end{proof}

\begin{corollary}\label{corollary:hopf_sum}
Let $F \boxtimes_R G$ be the Hopf sum of pro-tangles $F \in \fun^{(n)}_B$ and $G \in \fun^{(m)}_{B'}$ along a pair of strands $R$. Suppose $\delta\in \{0,2\}$ is the number of left crossings among the two crossings at the joint. Then we have
\begin{equation*}
  \kh^{k}(F \boxtimes_R G) \cong E_2^{0,k+\delta} \oplus E_2^{1,k+\delta-1} \oplus E_2^{2,k+\delta-2}.
\end{equation*}
\end{corollary}

\begin{proof}
By the $E_2$-collapse of the spectral sequence established in Theorem \ref{theorem:hopf_sum}, we have $E_\infty^{p,q} \cong E_2^{p,q}$. The convergence to the filtered Khovanov homology implies
\begin{equation*}
  \kh^n(F \boxtimes_R G) \cong \bigoplus_{p} E_\infty^{p, n-p}.
\end{equation*}
Setting the total degree $n = k+\delta$ and noting that the support of $E_r^{p,q}$ is restricted to $p \in \{0, 1, 2\}$, the homology decomposes into the direct sum
\begin{equation*}
    \kh^k(F \boxtimes_R G) \cong \bigoplus_{p=0}^2 E_2^{p, k+\delta-p} = E_2^{0,k+\delta} \oplus E_2^{1,k+\delta-1} \oplus E_2^{2,k+\delta-2}.
\end{equation*}
This completes the proof.
\end{proof}

\begin{example}
Consider the pro-tangle $F \cong S^1$ representing a single unknotted circle. Let $T' = F \boxtimes_R F$ be the Hopf sum of type \textup{(I)}, which topologically yields a Hopf link. In this case, the connected sum $F \#_R F$ remains an unknotted circle. According to Theorem \ref{theorem:hopf_sum}, the $E_1$ page of the associated spectral sequence is concentrated in $q=0$ and is given by
\begin{equation*}
    E_1^{0,0} \cong \kh(F \#_R G) \otimes V \cong V \otimes V, \quad
    E_1^{1,0} \cong V \oplus V, \quad
    E_1^{2,0} \cong \kh(F \sqcup G) \cong V \otimes V.
\end{equation*}
The $d_1$ differentials are induced by the saddle cobordisms within the $\cube^2$ joint cube
\begin{equation*}
    \xymatrix{
    V \otimes V \ar[rr]^-{d_1^{0,0}=(m, m)} && V \oplus V \ar[rr]^-{d_1^{1,0}=\binom{\Delta}{-\Delta}} && V \otimes V,
    }
\end{equation*}
where $m: V \otimes V \to V$ and $\Delta: V \to V \otimes V$ are the multiplication and comultiplication of the Frobenius algebra $V = \mathbb{F}[x]/(x^2)$. Under the basis $\{1, x\}$, the calculations for the kernels and images are as follows:
\begin{align*}
    \ker(m, m) &= \mathbb{F}\{1 \otimes x - x \otimes 1, x \otimes x\}, \\
    \im(m, m) &= \mathbb{F}\{(1, 1), (x, x)\}, \\
    \ker\textstyle{\binom{\Delta}{-\Delta}} &= \mathbb{F}\{(1, 1), (x, x)\}, \\
    \im\textstyle{\binom{\Delta}{-\Delta}} &= \mathbb{F}\{1 \otimes x + x \otimes 1, x \otimes x\}.
\end{align*}
Consequently, we find $E_2^{1,q} = 0$ for all $q$. The non-zero $E_2$ terms are
\begin{equation*}
    E_2^{0,0} \cong \mathbb{F}\{1 \otimes x - x \otimes 1, x \otimes x\}, \quad
    E_2^{2,0} \cong (V \otimes V) / \im\textstyle{\binom{\Delta}{-\Delta}} \cong \mathbb{F}\{1 \otimes 1, 1 \otimes x\}.
\end{equation*}
Applying Corollary \ref{corollary:hopf_sum} with $\delta=2$, the Khovanov homology $\kh^*(T')$ is generated by four elements. Specifically, $E_2^{0,0}$ contributes to $\kh^{-2}(T')$ with quantum degrees $\{-4, -6\}$, and $E_2^{2,0}$ contributes to $\kh^{0}(T')$ with quantum degrees $\{0, -2\}$. This result is consistent with the standard Khovanov homology of the Hopf link.
\end{example}

\subsubsection{Hopf twist}

The Hopf twist is a dual construction to the Hopf sum, representing a different local configuration of the Hopf clasp where the intermediate states correspond to the disjoint union of pro-tangles.

\begin{definition}
Let $F \in \fun^{(n)}_B$ and $G \in \fun^{(m)}_{B'}$ be pro-tangles with disjoint boundaries. Given a pair of strands $R = (\alpha_F, \alpha_G)$ from $F$ and $G$ respectively, the \textit{Hopf twist} along $R$, denoted by $F \boxplus_R G$, is defined as the pro-tangle functor
\begin{equation*}
    F \boxplus_{R} G: \cube^{2} \times \cube^{n+m} \longrightarrow \cob(B \sqcup B')
\end{equation*}
equipped with the following natural isomorphisms:
\begin{enumerate}[label=$(\roman*)$]
    \item $(F\boxplus_{R} G)^{(0,1)}_{0,1} \cong (F\boxplus_{R} G)^{(0,1)}_{1,0} \cong F \sqcup G$;
    \item In type \textup{(I)}: $(F\boxplus_{R} G)^{(0,1)}_{0,0} \cong F \sqcup G \sqcup S^1$ and $(F\boxplus_{R} G)^{(0,1)}_{1,1} \cong F \#_{R} G$. In type \textup{(II)}: $(F\boxplus_{R} G)^{(0,1)}_{0,0} \cong F \#_{R} G$ and $(F\boxplus_{R} G)^{(0,1)}_{1,1} \cong F \sqcup G \sqcup S^1$.
\end{enumerate}
\end{definition}

Figure \ref{figure:hopf_twist} illustrates the Hopf twist of two tangles.
\begin{figure}[H]
  \centering
  \includegraphics[width=0.3\textwidth]{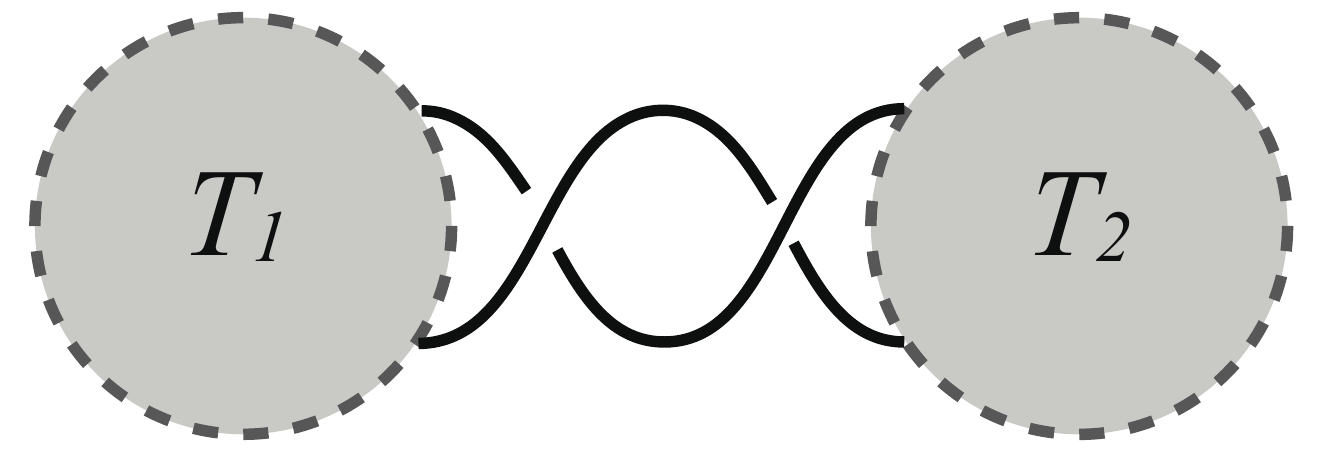}
  \caption{Illustration of the Hopf twist of type \textup{(I)} between two tangles $T_1$ and $T_2$.}\label{figure:hopf_twist}
\end{figure}

Similar to the Hopf sum, the Hopf twist admits a spectral sequence that relates the homology of the twist to the homology of its constituent parts and their connected sum.

\begin{theorem}\label{theorem:hopf_twist}
Let $F \boxplus_R G$ be the Hopf twist of pro-tangles $F \in \fun^{(n)}_B$ and $G \in \fun^{(m)}_{B'}$ along a pair of strands $R$. There exists a spectral sequence $E_r^{p,q}$ converging to the Khovanov homology of the Hopf twist
\begin{equation*}
  E_1^{p,q} \Rightarrow \kh^{p+q+\delta}(F \boxplus_R G),
\end{equation*}
where $\delta \in \{0,2\}$ is the number of left crossings at the joint. The spectral sequence collapses at the $E_3$ page, with the $E_1$ page given by
    \begin{align*}
        E_{1}^{0,q} &\cong \kh^q(F \sqcup G) \otimes V, \\
        E_{1}^{1,q} &\cong \kh^{q}(F \sqcup G) \oplus \kh^{q}(F \sqcup G), \\
        E_{1}^{2,q} &\cong \kh^q(F \#_R G),\\
        E_{1}^{p,q} &\cong 0, \quad p\notin \{0,1,2\}.
    \end{align*}
for type \textup{(I)} joint, and
    \begin{align*}
        E_{1}^{0,q} &\cong \kh^{q}(F \#_R G), \\
        E_{1}^{1,q} &\cong \kh^{q}(F \sqcup G) \oplus \kh^{q}(F \sqcup G), \\
        E_{1}^{2,q} &\cong \kh^q(F \sqcup G) \otimes V,\\
        E_{1}^{p,q} &\cong 0, \quad p\notin \{0,1,2\}.
    \end{align*}
for type \textup{(II)} joint. Here, $\kh^q(F \sqcup G)\cong\bigoplus_{i+j=q}\kh^i(F)\otimes Kh^j(G)$.
\end{theorem}

\begin{proof}
Consider the filtration on the complex $\mathcal{C}(F \boxplus_R G)$ associated with the joint subcube $\cube^2$. The $E_1$ terms are determined by the homology of the pro-tangle resolutions at each vertex $(i,j) \in \cube^2$. It suffices to prove the theorem for the case of a Hopf twist of type \textup{(I)}, as the argument for type \textup{(II)} is parallel.

For type \textup{(I)}, the vertex $(0,0)$ resolves to the connected sum $F \#_R G$, yielding the term $E_1^{0,q} \cong \kh^q(F\#_R G)$. The intermediate vertices $(0,1)$ and $(1,0)$ both resolve to the disjoint union $F \sqcup G$. By the K\"{u}nneth formula, $\kh^q(F\sqcup G) \cong \bigoplus_{i+j=q} \kh^i(F) \otimes \kh^j(G)$. Since both $(0,1)$ and $(1,0)$ contribute to this column, we obtain $E_1^{1,q} \cong \kh^q(F\sqcup G) \oplus \kh^q(F\sqcup G)$. The vertex $(1,1)$ yields the disjoint union $F \sqcup G \sqcup S^1$, which leads to the isomorphism $E_1^{2,q} \cong \kh^q(F\sqcup G) \otimes V$.

The collapse at the $E_3$ page and the convergence to $\kh^*(F\boxplus_R G)$ are guaranteed by Theorem \ref{theorem:hopf_clasp}. The grading shift by $\delta \in \{0,2\}$ accounts for the left-handed crossings at the joint.
\end{proof}

\begin{example}
In this example, we consider the Hopf twist of two trefoils. As shown in Figure~\ref{figure:trefoil_twist}, the type~\textup{II} Hopf twist is constructed by linking two trefoil knots.
\begin{figure}[H]
  \centering
  \includegraphics[width=0.4\textwidth]{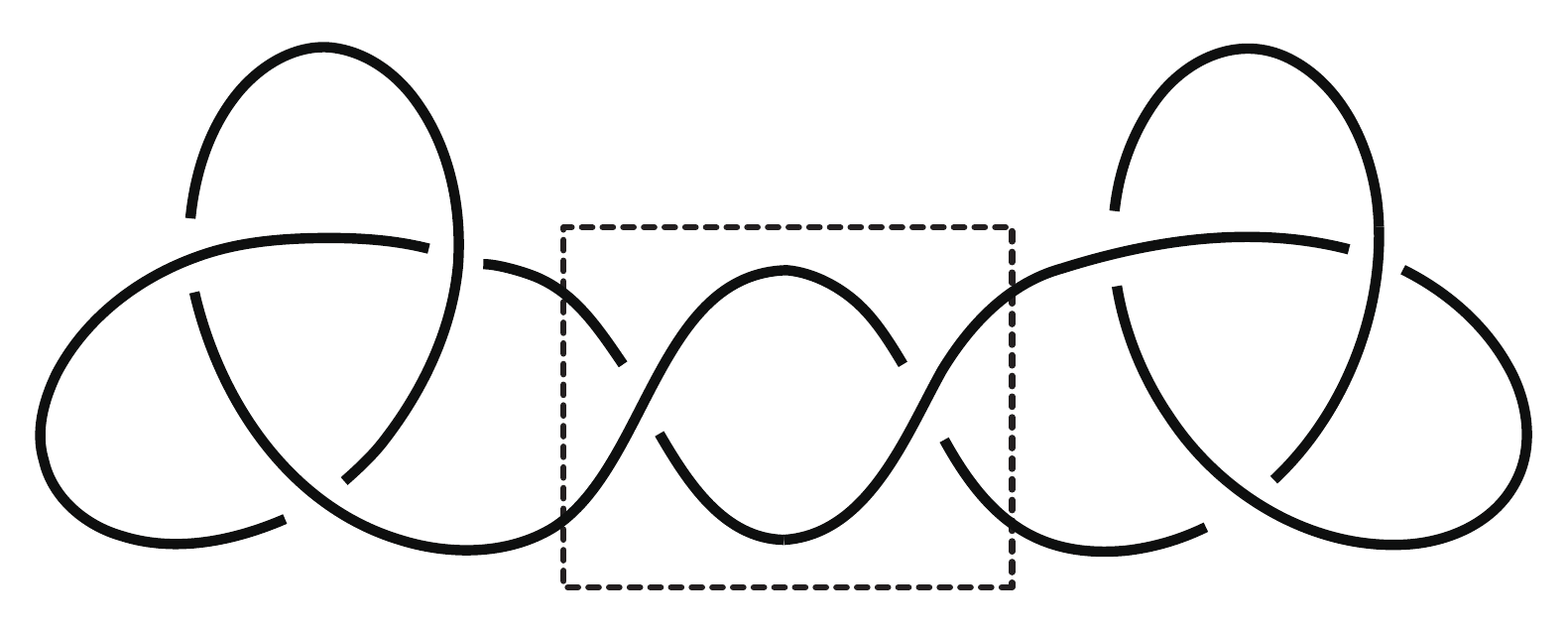}
  \caption{Illustration of the Hopf twist of two trefoils within the dashed box.}\label{figure:trefoil_twist}
\end{figure}

The dimensions $\dim E_r^{p,q}$ at each bidegree $(p,q)$ for the successive pages of this spectral sequence, obtained via computer computation, are summarized in Table~\ref{tab:spectral_sequence_new}. 
The $E_2$-page differentials $d_2^{0,5}: E_{2}^{0,5} \to E_{2}^{2,4}$ and $d_2^{0,6}: E_{2}^{0,6} \to E_{2}^{2,5}$ are non-trivial, which prevents four generators from surviving to the $E_3$-page.

\begin{table}[htbp]
\centering
\caption{Dimensions of the spectral sequence pages for Hopf twist of two trefoils.}
\label{tab:spectral_sequence_new}
\setlength{\tabcolsep}{3pt}
\renewcommand{\arraystretch}{1.05}

\begin{minipage}{0.45\textwidth}
\centering
\textbf{$E_0$-page}\\[2pt]
\begin{tabular}{p{0.7cm} *{7}{p{0.7cm}}}
\toprule
$p \setminus q$ & 0 & 1 & 2 & 3 & 4 & 5 & 6 \\
\midrule
0 & 8 & 24 & 66 & 104 & 120 & 96 & 32 \\
1 & 32 & 96 & 264 & 416 & 480 & 384 & 128 \\
2 & 32 & 96 & 264 & 416 & 480 & 384 & 128 \\
\bottomrule
\end{tabular}
\end{minipage}
\qquad
\begin{minipage}{0.45\textwidth}
\centering
\textbf{$E_1$-page}\\[2pt]
\begin{tabular}{p{0.7cm} *{7}{p{0.7cm}}}
\toprule
$p \setminus q$ & 0 & 1 & 2 & 3 & 4 & 5 & 6 \\
\midrule
0 & 2 & 0 & 2 & 2 & 1 & 2 & 1 \\
1 & 8 & 0 & 8 & 8 & 2 & 4 & 2 \\
2 & 8 & 0 & 8 & 8 & 2 & 4 & 2 \\
\bottomrule
\end{tabular}
\end{minipage}

\vspace{6pt}
\begin{minipage}{0.45\textwidth}
\centering
\textbf{$E_2$-page}\\[2pt]
\begin{tabular}{p{0.7cm} *{7}{p{0.7cm}}}
\toprule
$p \setminus q$ & 0 & 1 & 2 & 3 & 4 & 5 & 6 \\
\midrule
0 & 0 & 0 & 0 & 0 & 0 & 1 & 1 \\
1 & 0 & 0 & 0 & 0 & 0 & 1 & 1 \\
2 & 2 & 0 & 2 & 2 & 1 & 2 & 1 \\
\bottomrule
\end{tabular}
\end{minipage}
\qquad
\begin{minipage}{0.45\textwidth}
\centering
\textbf{$E_3$-page}\\[2pt]
\begin{tabular}{p{0.7cm} *{7}{p{0.7cm}}}
\toprule
$p \setminus q$ & 0 & 1 & 2 & 3 & 4 & 5 & 6 \\
\midrule
0 & 0 & 0 & 0 & 0 & 0 & 0 & 0 \\
1 & 0 & 0 & 0 & 0 & 0 & 1 & 1 \\
2 & 2 & 0 & 2 & 2 & 0 & 1 & 1 \\
\bottomrule
\end{tabular}
\end{minipage}
\end{table}

\end{example}

To conclude this section, we provide an overview of the structural and dual relationships between the variant clasping modifications introduced thus far. The internal geometric configurations within the joint cube $\cube^2$, specifically the localized topological resolutions at each state vertex for the Hopf clasp, Hopf sum, and Hopf twist, are systematically contrasted in Table \ref{table:hopf_operations_summary}.

\begin{table}[ht]
\centering
\caption{Geometric configurations of Hopf clasp operations}\label{table:hopf_operations_summary}
\begin{tabular}{lcccc}
\toprule
\textbf{Operation} & \textbf{Type} & \textbf{State $(0,0)$} & \textbf{State $(0,1)/(1,0)$} & \textbf{State $(1,1)$} \\ 
\midrule
\multirow{2}{*}{Hopf clasp $F$} 
  & \textup{(I)}  & $G \sqcup S^1$ & $G$ & $F_{1,1}$ \\
  & \textup{(II)} & $F_{0,0}$      & $G$ & $G \sqcup S^1$ \\ 
\midrule
\multirow{2}{*}{Hopf sum $F \boxtimes_R G$} 
  & \textup{(I)}  & $(F\#_R G) \sqcup S^1$ & $F\#_R G$ & $F \sqcup G$ \\
  & \textup{(II)} & $F \sqcup G$ & $F\#_R G$ & $(F\#_R G) \sqcup S^1$ \\ 
\midrule
\multirow{2}{*}{Hopf twist $F \boxplus_R G$} 
  & \textup{(I)}  & $F\#_R G$ & $F \sqcup G$ & $(F \sqcup G) \sqcup S^1$ \\
  & \textup{(II)} & $(F \sqcup G) \sqcup S^1$ & $F \sqcup G$ & $F\#_R G$ \\ 
\bottomrule
\end{tabular}
\end{table}

\subsection{Applications}\label{section:applications}

In this section, we present the application of the Hopf sum between tangles and a single circle or arc. This is particularly useful in computing the Poincar\'{e} polynomial in certain cases.

\subsubsection{Hopf sums with a circle}

We now consider the behavior of the Khovanov homology under the Hopf sum of a pro-link $L$ and the standard circle $S^1$.

\begin{theorem}\label{theorem:links_hopf}
Let $L$ be a pro-link, and let $L' = L \boxtimes_R S^1$ be the Hopf sum of $L$ and $S^1$ along a component $R$. 
\begin{itemize}
    \item If the joint crossings are left-handed, the Poincar\'{e} polynomial of $L'$ is given by
    \begin{equation*}
        P_{L'}(t,q) = P_{L}(t,q) (q^{-1} + t^{-2} q^{-5}).
    \end{equation*}
    \item If the joint crossings are right-handed, the Poincar\'{e} polynomial of $L'$ is given by
    \begin{equation*}
        P_{L'}(t,q) = P_{L}(t,q) (q + t^2 q^5).
    \end{equation*}
\end{itemize}
\end{theorem}

\begin{proof}
We provide the proof for a Hopf sum of type \textup{(I)}; the type \textup{(II)} case follows by symmetry.

We arrange the two crossings of the Hopf sum at the first two positions of the state cube for $L'$. The resolutions of $L'$ are indexed as $(i, j, s)$, where $(i, j) \in \{0,1\}^2$ corresponds to the joint crossings and $s \in \{0,1\}^n$ to the crossings of $L$. The resulting links for the joint resolutions are
\begin{align*}
    L'_{(0,0,s)} &= L \sqcup S^1, \\
    L'_{(0,1,s)} &= L, \\
    L'_{(1,0,s)} &= L, \\
    L'_{(1,1,s)} &= L \sqcup S^1.
\end{align*}

By Theorem \ref{theorem:hopf_sum}, there exists a spectral sequence $E_r^{p,q}$ converging to $\kh^*(L')$. The $E_1$ page is concentrated in the first three columns ($p=0,1,2$):
\begin{equation*}
    \xymatrix{
    \kh^{\ast}(L) \otimes V \ar[rr]^-{(\widetilde{m},\widetilde{m})} && \kh^{\ast}(L) \oplus \kh^{\ast}(L) \ar[rr]^-{\binom{\widetilde{\Delta}}{-\widetilde{\Delta}}} && \kh^{\ast}(L) \otimes V,
    }
\end{equation*}
where $\widetilde{m}$ and $\widetilde{\Delta}$ are induced by the multiplication and comultiplication of the Frobenius algebra $V = \mathbb{F}[x]/(x^2)$. 

Since $1 \in V$ is the unit, $\widetilde{m}(\alpha \otimes 1) = \alpha$ for any $\alpha \in \kh^*(L)$, which implies
\begin{equation*}
    \im (\widetilde{m}, \widetilde{m}) \cong \kh^*(L),
\end{equation*}
and we have $\ker(\widetilde{m}, \widetilde{m}) \cong \kh^*(L) \otimes x$. Furthermore, $\im\binom{\widetilde{\Delta}}{-\widetilde{\Delta}} \cong \widetilde{\Delta}(\kh^*(L)) \cong \kh^*(L) \otimes x$. It follows that $E_2^{1,q} = 0$, while the remaining non-zero terms are
\begin{equation*}
    E_2^{0,q} \cong \kh^{q}(L) \otimes x, \quad E_2^{2,q} \cong \kh^{q}(L) \otimes 1.
\end{equation*}

Applying Corollary \ref{corollary:hopf_sum}, we obtain the decomposition:
\begin{equation*}
    \kh^{k}(L') \cong (\kh^{k+\delta}(L) \otimes x) \oplus (\kh^{k+\delta-2}(L) \otimes 1).
\end{equation*}
In our Frobenius algebra $V$, $\deg(1)=1$ and $\deg(x)=-1$. Accounting for the internal shifts and the writhe compensation $\delta$, the bigraded homology satisfies
\begin{equation*}
    \kh^{k,l}(L') \cong \kh^{k+\delta,l+3\delta-1}(L) \oplus \kh^{k+\delta-2,l+3\delta-5}(L).
\end{equation*}
Summing over the gradings yields the polynomial relation
\begin{equation*}
    P_{L'}(t,q) = P_{L}(t,q) (t^{-\delta}q^{-3\delta+1} + t^{-\delta+2} q^{-3\delta+5}).
\end{equation*}
The theorem follows by substituting $\delta=2$ for left-handed joints and $\delta=0$ for right-handed joints.
\end{proof}

\begin{proposition}
Let $L_n$ be a link obtained by the iterative Hopf sum of $(n+1)$ circles $S^1$ in a linear chain. If $j$ of the joints are left-handed and $n-j$ are right-handed, the Poincar\'{e} polynomial is
\begin{equation*}
    P_{L_n}(t,q) = (q + q^{-1}) (q^{-1} + t^{-2} q^{-5})^{j}(q + t^2 q^5)^{n-j}.
\end{equation*}
\end{proposition}
\begin{proof}
The result follows by induction on $n$ applying Theorem \ref{theorem:links_hopf} at each step, with the base case $P_{S^1}(t,q) = q + q^{-1}$.
\end{proof}

\subsubsection{Hopf sums with an arc}

\begin{definition}
Let $T$ be a tangle. An \textit{arc} of $T$ is called \textit{pure} in $T$ if it does not form part of the boundary of any closed region in $T$.
\end{definition}

\begin{definition}
Let $T = T_1 \boxtimes_R T_2$ be the Hopf sum of tangles $T_1$ and $T_2$ along the joined arcs $\alpha_1 \subset T_1$ and $\alpha_2 \subset T_2$. If $\alpha_1$ and $\alpha_2$ are pure arcs in $T_1$ and $T_2$ respectively, the Hopf sum is said to be \textit{pure}.
\end{definition}

\begin{definition}
A \textit{cut tangle} $\widehat{T}$ of a tangle $T$ is obtained by cutting a segment of an arc or loop in $T$ and connecting the two resulting endpoints to the boundary without introducing additional crossings.
\begin{figure}[h]
  \centering
  \includegraphics[width=0.35\textwidth]{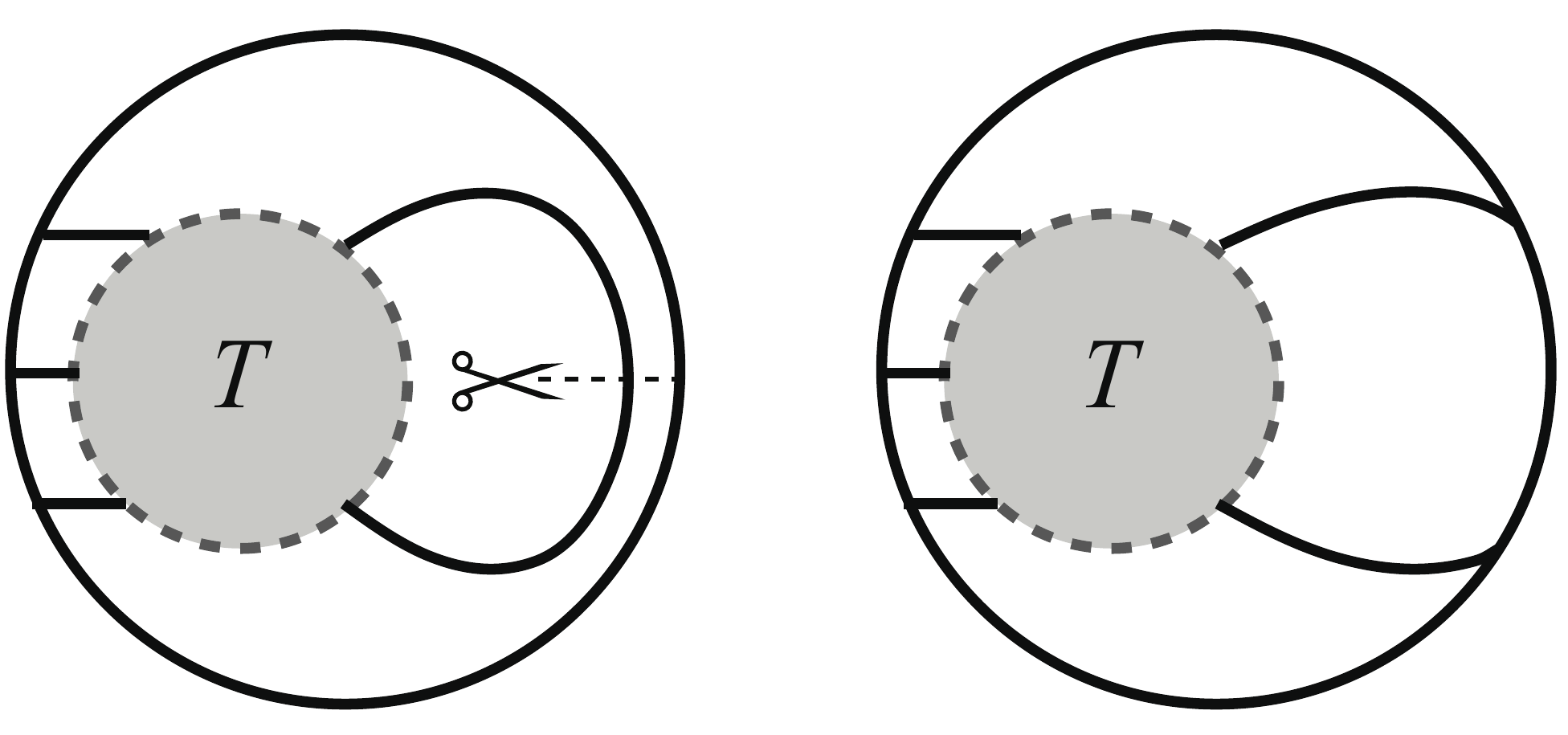}
  \caption{Illustration of the cut tangle.}
\end{figure}
\end{definition}

\begin{theorem}\label{theorem:hopf_joint_arc}
Let $T$ be a tangle, and let $T' = T \boxtimes_R A$ be the pure Hopf sum of $T$ and an arc $A$ along $\alpha \subseteq T$. Let $\widehat{T}$ denote the cut tangle of $T$ at the joint position. 
\begin{itemize}
\item For a Hopf sum of type \textup{(I)}:
\begin{itemize}
    \item If the joint crossings are right-handed ($\delta=0$):
        \begin{equation*}
           \kh^{k,l}(T') \cong \kh^{k,l-1}(\widehat{T}) \oplus \kh^{k-1,l-3}(\widehat{T}) \oplus \kh^{k-2,l-3}(T).
        \end{equation*}
    \item If the joint crossings are left-handed ($\delta=2$):
        \begin{equation*}
           \kh^{k,l}(T') \cong \kh^{k+2,l+5}(\widehat{T}) \oplus \kh^{k+1,l+3}(\widehat{T}) \oplus \kh^{k,l+3}(T).
        \end{equation*}
\end{itemize}
\item For a Hopf sum of type \textup{(II)}:
\begin{itemize}
    \item If the joint crossings are right-handed ($\delta=0$):
        \begin{equation*}
           \kh^{k,l}(T') \cong \kh^{k+2,l+5}(T) \oplus \kh^{k+1,l+3}(\widehat{T}) \oplus \kh^{k,l+1}(\widehat{T}).
        \end{equation*}
    \item If the joint crossings are left-handed ($\delta=2$):
        \begin{equation*}
           \kh^{k,l}(T') \cong \kh^{k,l-1}(T) \oplus \kh^{k-1,l-3}(\widehat{T}) \oplus \kh^{k-2,l-5}(\widehat{T}).
        \end{equation*}
\end{itemize}
\end{itemize}
\end{theorem}

\begin{proof}
The tangle $T'$ consists of $n+2$ crossings, where $n$ is the number of crossings in $T$. Let $\delta \in \{0, 2\}$ be the number of left-handed crossings at the joint. We arrange the joint crossings in the first two positions of the state cube $\cube^{n+2}$.

\textbf{(i) Type \textup{(I)} Hopf sum.}
The resolutions of $T'$ at the joint vertices $(i,j)$ for a given state $s$ of $T$ are given by
\begin{equation*}
    T'_{(0,0,s)} = \widehat{T}_{s} \sqcup S^1, \quad T'_{(0,1,s)} = \widehat{T}_{s}, \quad T'_{(1,0,s)} = \widehat{T}_{s}, \quad T'_{(1,1,s)} = T_{s} \sqcup A.
\end{equation*}
By Theorem \ref{theorem:hopf_sum}, the $E_1$ page of the spectral sequence is
\begin{equation*}
    \xymatrix{
    \kh^{\ast}(\widehat{T}) \otimes V \ar[rr]^-{(\mu, \mu)} && \kh^{\ast}(\widehat{T}) \oplus \kh^{\ast}(\widehat{T}) \ar[rr]^-{\binom{\nu}{-\nu}} && \kh^{\ast}(T) \otimes W,
    }
\end{equation*}
where $V = \mathbb{F}[x]/(x^2)$ and $W \cong xV$ is the module associated with the arc $A$. For a pure Hopf sum, the induced saddle maps $\nu: \kh^*(\widehat{T}) \to \kh^*(T) \otimes W$ vanish due to the topological configuration of the pure arc. The map $\mu$ satisfies $\mu(\alpha \otimes 1) = \alpha$ and $\mu(\alpha \otimes x) = 0$. Thus, the $E_2$ terms are
\begin{equation*}
    E_2^{0,q} \cong \kh^q(\widehat{T}) \otimes x, \quad E_2^{1,q} \cong \kh^q(\widehat{T}), \quad E_2^{2,q} \cong \kh^q(T) \otimes x.
\end{equation*}
To determine the quantum grading $\Phi(z)$, we use the formula $\Phi(z) = \theta(z) + k + w(T')$, where $\theta(z)$ is the internal grading and $w(T') = w(\widehat{T}) + 2 - 2\delta$ is the writhe.
\begin{enumerate}
    \item[(a)] For $z \in E_2^{0, k+\delta}$, let $z = \alpha \otimes x$ with $\alpha \in \kh^{k+\delta}(\widehat{T})$. Then $\theta(z) = \theta(\alpha) - 1$.
    \begin{align*}
        \Phi(z) &= (\theta(\alpha) - 1) + k + (w(\widehat{T}) + 2 - 2\delta) \\
        &= (\theta(\alpha) + (k+\delta) + w(\widehat{T})) - \delta - 1 + 2 - 2\delta \\
        &= \Phi(\alpha) - 3\delta + 1.
    \end{align*}
    \item[(b)] For $z \in E_2^{1, k+\delta-1}$, $z$ corresponds to $\alpha \in \kh^{k+\delta-1}(\widehat{T})$ where $\theta(z) = \theta(\alpha)$.
    \begin{align*}
        \Phi(z) &= \theta(\alpha) + k + (w(\widehat{T}) + 2 - 2\delta) \\
        &= (\theta(\alpha) + (k+\delta-1) + w(\widehat{T})) - \delta + 1 + 2 - 2\delta \\
        &= \Phi(\alpha) - 3\delta + 3.
    \end{align*}
    \item[(c)] For $z \in E_2^{2, k+\delta-2}$, let $z = \alpha \otimes x$ with $\alpha \in \kh^{k+\delta-2}(T)$. Then $\theta(z) = \theta(\alpha) - 1$.
    \begin{align*}
        \Phi(z) &= (\theta(\alpha) - 1) + k + (w(T) + 2 - 2\delta) \\
        &= (\theta(\alpha) + (k+\delta-2) + w(T)) - \delta + 2 - 1 + 2 - 2\delta \\
        &= \Phi(\alpha) - 3\delta + 3.
    \end{align*}
\end{enumerate}
Accounting for the canonical normalization of the arc $A$, these shifts yield the bigraded isomorphisms for type \textup{(I)}.

\textbf{(ii) Type \textup{(II)} Hopf sum.}
The resolutions are $T'_{(0,0,s)} = T_s \sqcup A$, $T'_{(0,1,s)} = \widehat{T}_s$, $T'_{(1,0,s)} = \widehat{T}_s$, and $T'_{(1,1,s)} = \widehat{T}_s \sqcup S^1$. The $E_1$ page is
\begin{equation*}
    \xymatrix{
    \kh^{\ast}(T) \otimes W \ar[rr]^-{(\mu', \mu')} && \kh^{\ast}(\widehat{T}) \oplus \kh^{\ast}(\widehat{T}) \ar[rr]^-{\binom{\nu'}{-\nu'}} && \kh^{\ast}(\widehat{T}) \otimes V,
    }
\end{equation*}
where $\mu'=0$ as the mapping components $W \otimes W \to 0$. The map $\nu'$ is induced by the component change $W \to W \otimes V$ given by $x \mapsto x \otimes x$. Thus $\nu'(\alpha) = \alpha \otimes x$. It follows that
\begin{equation*}
    \ker \binom{\nu'}{-\nu'} = \{ (\alpha, \alpha) \mid \alpha \in \kh^*(\widehat{T}) \}, \quad \im \binom{\nu'}{-\nu'} = \kh^*(\widehat{T}) \otimes x.
\end{equation*}
The $E_2$ terms are
\begin{equation*}
    E_2^{0,q} \cong \kh^q(T) \otimes x, \quad E_2^{1,q} \cong \kh^q(\widehat{T}), \quad E_2^{2,q} \cong \kh^q(\widehat{T}) \otimes 1.
\end{equation*}
Here, note that $E_2^{2,q}$ is the quotient $(\kh^* \otimes V) / (\kh^* \otimes x)$, which is isomorphic to $\kh^* \otimes 1$.
The quantum grading analysis for type \textup{(II)} is as follows
\begin{enumerate}
    \item[(a)] If $z \in E_2^{0, k+\delta}$ lies in $\kh^{k+\delta}(T) \otimes x$, then $z = \alpha \otimes x$ and $\theta(\alpha) = \theta(z) + 1$.
    \begin{equation*}
        \Phi(\alpha) = k + \delta + w(T) + \theta(\alpha) = \Phi(z) + 3\delta - 1.
    \end{equation*}
    \item[(b)] If $z \in E_2^{1, k+\delta-1} \cong \kh^{k+\delta-1}(\widehat{T})$, then $\theta(\alpha) = \theta(z)$.
    \begin{equation*}
        \Phi(\alpha) = k + \delta - 1 + w(T) + \theta(z) = \Phi(z) + 3\delta - 3.
    \end{equation*}
    \item[(c)] If $z \in E_2^{2, k+\delta-2}$ lies in $\kh^{k+\delta-2}(\widehat{T}) \otimes 1$, we represent it by the class of $\alpha \otimes 1$. Then $\theta(\alpha) = \theta(z) - 1$.
    \begin{equation*}
        \Phi(\alpha) = k + \delta - 2 + w(T) + \theta(z) = \Phi(z) + 3\delta - 5.
    \end{equation*}
\end{enumerate}
Combining these computations, we obtain the desired result.
\end{proof}

\begin{proof}
The tangle $T'$ consists of $n+2$ crossings, where $n$ is the number of crossings in $T$. Let $\delta \in \{0, 2\}$ be the number of left-handed crossings at the joint. We arrange the joint crossings in the first two positions of the state cube $\cube^{n+2}$.

\medskip
\noindent \textbf{(i) Type \textup{(I)} Hopf sum.}
The resolutions of $T'$ at the joint vertices $(i,j)$ for a given state $s$ of $T$ are given by
\begin{equation*}
    T'_{(0,0,s)} = \widehat{T}_{s} \sqcup S^1, \quad T'_{(0,1,s)} = \widehat{T}_{s}, \quad T'_{(1,0,s)} = \widehat{T}_{s}, \quad T'_{(1,1,s)} = T_{s} \sqcup A.
\end{equation*}
By Theorem \ref{theorem:hopf_sum}, the $E_1$ page of the spectral sequence is
\begin{equation*}
    \xymatrix{
    \kh^{\ast}(\widehat{T}) \otimes V \ar[rr]^-{(\mu, \mu)} && \kh^{\ast}(\widehat{T}) \oplus \kh^{\ast}(\widehat{T}) \ar[rr]^-{\binom{\nu}{-\nu}} && \kh^{\ast}(T) \otimes W,
    }
\end{equation*}
where $V = \mathbb{F}[x]/(x^2)$ and $W \cong xV$ is the module associated with the arc $A$. For a pure Hopf sum, the induced saddle maps $\nu: \kh^*(\widehat{T}) \to \kh^*(T) \otimes W$ vanish due to the topological configuration of the pure arc. The map $\mu$ satisfies $\mu(\alpha \otimes 1) = \alpha$ and $\mu(\alpha \otimes x) = 0$. Thus, the $E_2$ terms are:
\begin{equation*}
    E_2^{0,q} \cong \kh^q(\widehat{T}) \otimes x, \quad E_2^{1,q} \cong \kh^q(\widehat{T}), \quad E_2^{2,q} \cong \kh^q(T) \otimes x.
\end{equation*}

To determine the quantum grading $\Phi(z)$, we use the formula $\Phi(z) = \theta(z) + k + w(T')$, where $\theta(z)$ is the internal grading and $w(T') = w(\widehat{T}) + 2 - 2\delta$ is the writhe.
\begin{enumerate}
    \item[(a)] For $z \in E_2^{0, k+\delta}$, let $z = y \otimes x$ with $x \in \kh^{k+\delta}(\widehat{T})$. Then $\theta(z) = \theta(x) - 1$.
    \begin{align*}
        \Phi(z) &= (\theta(x) - 1) + k + (w(\widehat{T}) + 2 - 2\delta) \\
        &= (\theta(x) + (k+\delta) + w(\widehat{T})) - \delta - 1 + 2 - 2\delta \\
        &= \Phi(x) - 3\delta + 1.
    \end{align*}
    \item[(b)] For $z \in E_2^{1, k+\delta-1}$, $z$ corresponds to $x \in \kh^{k+\delta-1}(\widehat{T})$ where $\theta(z) = \theta(x)$.
    \begin{align*}
        \Phi(z) &= \theta(x) + k + (w(\widehat{T}) + 2 - 2\delta) \\
        &= (\theta(x) + (k+\delta-1) + w(\widehat{T})) - \delta + 1 + 2 - 2\delta \\
        &= \Phi(x) - 3\delta + 3.
    \end{align*}
    \item[(c)] For $z \in E_2^{2, k+\delta-2}$, let $z = x \otimes x$ with $x \in \kh^{k+\delta-2}(T)$. Then $\theta(z) = \theta(x) - 1$.
    \begin{align*}
        \Phi(z) &= (\theta(x) - 1) + k + (w(T) + 2 - 2\delta) \\
        &= (\theta(x) + (k+\delta-2) + w(T)) - \delta + 2 - 1 + 2 - 2\delta \\
        &= \Phi(x) - 3\delta + 3.
    \end{align*}
\end{enumerate}
Accounting for the canonical normalization of the arc $A$, these shifts yield the bigraded isomorphisms for type \textup{(I)}.

\medskip
\noindent \textbf{(ii) Type \textup{(II)} Hopf sum.}
The resolutions are $T'_{(0,0,s)} = T_s \sqcup A$, $T'_{(0,1,s)} = \widehat{T}_s$, $T'_{(1,0,s)} = \widehat{T}_s$, and $T'_{(1,1,s)} = \widehat{T}_s \sqcup S^1$. The $E_1$ page is
\begin{equation*}
    \xymatrix{
    \kh^{\ast}(T) \otimes W \ar[rr]^-{(\mu', \mu')} && \kh^{\ast}(\widehat{T}) \oplus \kh^{\ast}(\widehat{T}) \ar[rr]^-{\binom{\nu'}{-\nu'}} && \kh^{\ast}(\widehat{T}) \otimes V,
    }
\end{equation*}
where $\mu'=0$ as the mapping components $W \otimes W \to 0$. The map $\nu'$ is induced by the component change $W \to W \otimes V$ given by $x \mapsto x \otimes x$. Thus $\nu'(\alpha) = \alpha \otimes x$. It follows that
\begin{equation*}
    \ker \binom{\nu'}{-\nu'} = \{ (\alpha, \alpha) \mid \alpha \in \kh^*(\widehat{T}) \}, \quad \im \binom{\nu'}{-\nu'} = \kh^*(\widehat{T}) \otimes x.
\end{equation*}
The $E_2$ terms are
\begin{equation*}
    E_2^{0,q} \cong \kh^q(T) \otimes x, \quad E_2^{1,q} \cong \kh^q(\widehat{T}), \quad E_2^{2,q} \cong \kh^q(\widehat{T}) \otimes 1.
\end{equation*}
Here, note that $E_2^{2,q}$ is the quotient $(\kh^* \otimes V) / (\kh^* \otimes x)$, which is isomorphic to $\kh^* \otimes 1$.

The quantum grading analysis for type \textup{(II)} is as follows
\begin{enumerate}
    \item[(a)] If $z \in E_2^{0, k+\delta}$ lies in $\kh^{k+\delta}(T) \otimes x$, then $z = \alpha \otimes x$ and $\theta(\alpha) = \theta(z) + 1$.
    \begin{equation*}
        \Phi(\alpha) = k + \delta + w(T) + \theta(\alpha) = \Phi(z) + 3\delta - 1.
    \end{equation*}
    \item[(b)] If $z \in E_2^{1, k+\delta-1} \cong \kh^{k+\delta-1}(\widehat{T})$, then $\theta(\alpha) = \theta(z)$.
    \begin{equation*}
        \Phi(\alpha) = k + \delta - 1 + w(T) + \theta(z) = \Phi(z) + 3\delta - 3.
    \end{equation*}
    \item[(c)] If $z \in E_2^{2, k+\delta-2}$ lies in $\kh^{k+\delta-2}(\widehat{T}) \otimes 1$, we represent it by the class of $\alpha \otimes 1$. Then $\theta(\alpha) = \theta(z) - 1$.
    \begin{equation*}
        \Phi(\alpha) = k + \delta - 2 + w(T) + \theta(z) = \Phi(z) + 3\delta - 5.
    \end{equation*}
\end{enumerate}
Combining these computations, we obtain the desired result.
\end{proof}

\begin{corollary}\label{cor:poincare_hopf_arc}
Let $T$ be a tangle and $T' = T \boxtimes_R A$ be its pure Hopf sum with an arc $A$. The Poincaré polynomial $P_{T'}(t, q)$ of the resulting tangle is determined by the type and chirality of the joint as follows:

\begin{enumerate}[label=$(\roman*)$]
    \item For a Hopf sum of type \textup{(I)}:
    \begin{itemize}
        \item If the joint consists of right-handed crossings ($\delta=0$),
        \begin{equation*} P_{T'}(t, q) = P_{\widehat{T}}(t, q)(q + tq^3) + P_T(t, q)t^2 q^3. \end{equation*}
        \item If the joint consists of left-handed crossings ($\delta=2$),
        \begin{equation*} P_{T'}(t, q) = P_{\widehat{T}}(t, q)(t^{-2} q^{-5} + t^{-1} q^{-3}) + P_T(t, q)t^{-1} q^{-3}. \end{equation*}
    \end{itemize}
    
    \item For a Hopf sum of type \textup{(II)}:
    \begin{itemize}
        \item If the joint consists of right-handed crossings ($\delta=0$),
        \begin{equation*} P_{T'}(t, q) = P_T(t, q)t^{-2} q^{-5} + P_{\widehat{T}}(t, q)(t^{-1} q^{-3} + q^{-1}). \end{equation*}
        \item If the joint consists of left-handed crossings ($\delta=2$),
        \begin{equation*} P_{T'}(t, q) = P_T(t, q)q + P_{\widehat{T}}(t, q)(tq^3 + t^2 q^5). \end{equation*}
    \end{itemize}
\end{enumerate}
Here, $P_{\widehat{T}}(t, q)$ denotes the Poincaré polynomial of the cut tangle associated with the joint position.
\end{corollary}

\subsubsection{A spectral sequence for arc joint}

Let $T$ be a tangle with two endpoints. Denote by $T'$ the tangle obtained from $T$ as follows: an embedded arc $\beta$ intersects each of the two boundary arcs of $T$ corresponding to its endpoints transversely at exactly one point. The new tangle $T'$ is obtained by replacing the portions of $T$ between these intersection points with the arc $\beta$, as illustrated in Figure~\ref{figure:arc_reduction}. Hence, $T'$ is derived from $T$ by a local modification along $\beta$, and remains a tangle.

\begin{figure}[h]
  \centering
  \includegraphics[width=0.6\textwidth]{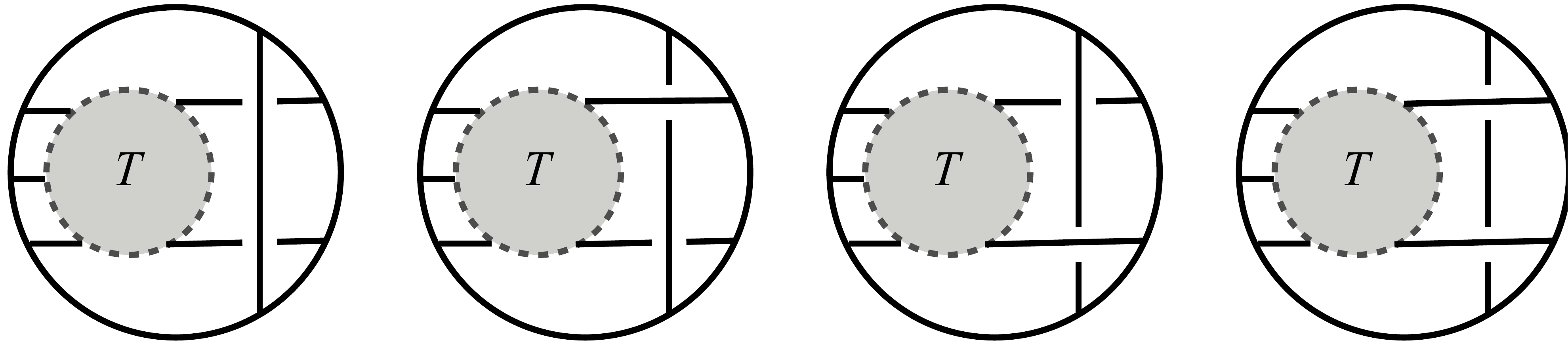} 
  \caption{Illustration of the four possible configurations for constructing the tangle $T'$.}\label{figure:arc_reduction}
\end{figure}

\begin{example}
As shown in Figure~\ref{figure:example_tangles}, we present the tangle $T$ together with the four possible configurations of $T'$ obtained by joining $T$ with an arc as described above. It is noteworthy that, unlike the Hopf clasp, the arc tangle in this section cannot be simplified by Reidemeister moves due to the interlacing of the tangle components.
\begin{figure}[h]
  \centering
  \includegraphics[width=0.7\textwidth]{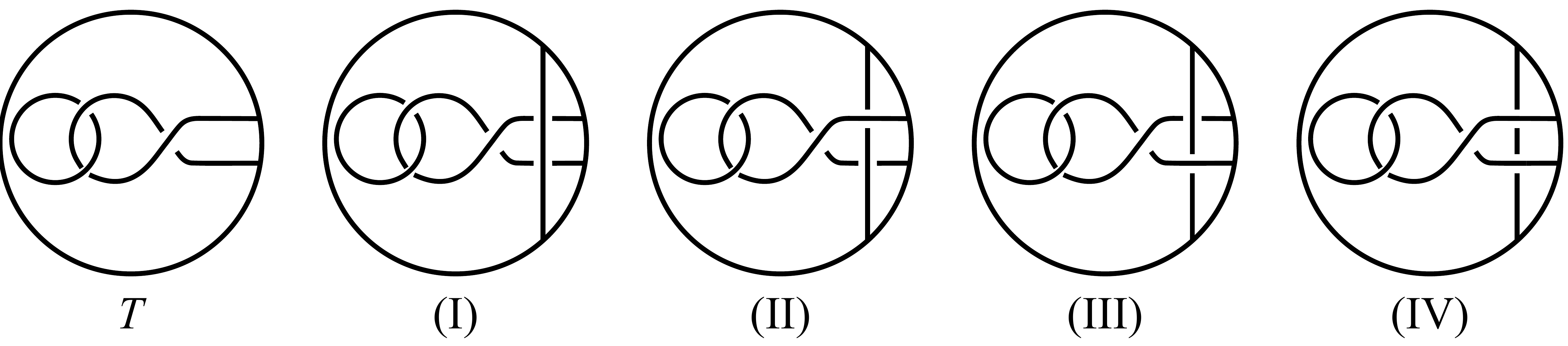} 
  \caption{Illustration of the tangle $T$ and its four possible configurations of $T'$ obtained by joining with an arc.}\label{figure:example_tangles}
\end{figure}
\end{example}

Although $T'$ is constructed by joining the tangle $T$ with an arc, in practice, by establishing a relationship between the Khovanov homology of $T$ and $T'$, we can characterize the invariants of the more complex $T'$ via those of the simpler $T$. The fundamental objective is to reduce a complex tangle to a more tractable one. For instance, the type \textup{(I)} tangle $T'$ can be reduced to $T$, which may itself be further decomposed. Such an approach transforms the computational challenge of complex tangles into a systematic study of simpler constituents.

\begin{definition}
Let $T$ be a tangle with two endpoints. By identifying the two endpoints of $T$, we obtain a closed curve, which forms a link denoted by $\widetilde{T}$. This link is called the \textit{glued link} of $T$.
\begin{figure}[h]
  \centering
  \includegraphics[width=0.35\textwidth]{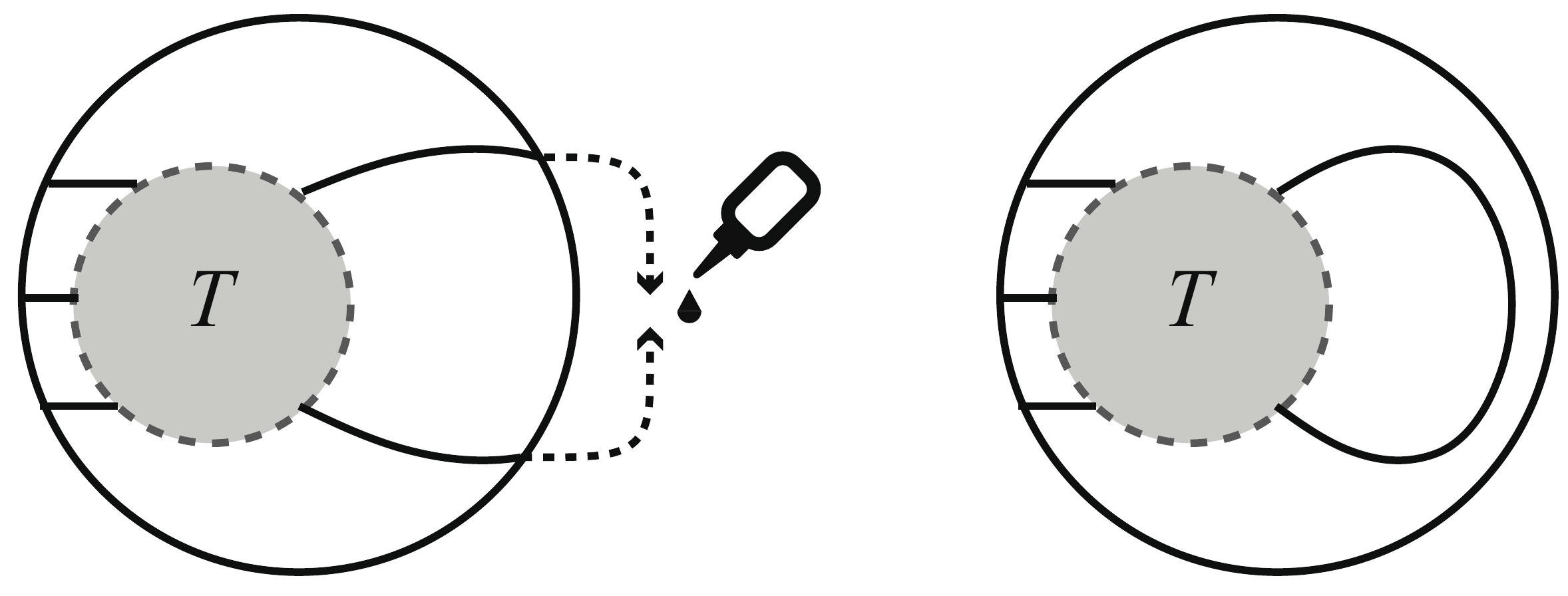} 
  \caption{Illustration of the glued link constructed from a tangle.}\label{figure:gluing_tangle}
\end{figure}
\end{definition}

\begin{theorem}\label{theorem:spectral_sequence_refined}
There exists a spectral sequence of the form
\begin{equation*}
  E_1^{p,q} \Rightarrow \kh^{p+q+\delta}(T'),
\end{equation*}
converging to the Khovanov homology of $T'$, where $\delta$ denotes the number of left-handed crossings at the joint. The $E_1$-terms for each configuration type are summarized as follows:
\begin{equation*}
\renewcommand{\arraystretch}{1.5}
\begin{tabular}{c|c|c|c|c}
  \hline
  $E_1$ page & Type \textup{(I)} & Type \textup{(II)} & Type \textup{(III)} & Type \textup{(IV)} \\
  \hline
  $E_1^{0,q}$ & $\kh^q(T)$ & $\kh^q(T)$ & $\kh^q(\widetilde{T})$ & $\kh^q(T)$ \\
  $E_1^{1,q}$ & $\kh^q(\widetilde{T}) \oplus \kh^q(T)$  & $\kh^q(T)\oplus \kh^q(T)$ & $\kh^q(T)\oplus \kh^q(T)$ & $\kh^q(\widetilde{T}) \oplus \kh^q(T)$ \\
  $E_1^{2,q}$ & $\kh^q(T)$ & $\kh^q(\widetilde{T})$ & $\kh^q(T)$ & $\kh^q(T)$ \\
  \hline
\end{tabular}
\end{equation*}
\noindent More precisely, the spectral sequence collapses at the $E_3$ page, exhibiting two distinct structural behaviors:
\begin{enumerate}[label=$(\roman*)$]
    \item \textit{Types \textup{(I)} and \textup{(IV)}:} At the $E_1$ page, the row complexes decouple into a reduced core and a trivial summand $\kh^*(T)$.
    \item \textit{Types \textup{(II)} and \textup{(III)}:} The spectral sequence collapses at the $E_2$ page, yielding the isomorphism of vector spaces
    \begin{equation*}
        \kh^k(T') \cong E_2^{0,k+\delta} \oplus E_2^{1,k+\delta-1} \oplus E_2^{2,k+\delta-2}.
    \end{equation*}
\end{enumerate}
\end{theorem}

\begin{proof}
Consider the tangle $T'$ obtained by an arc-reduction modification on $T$ at two specific crossings. We assign the coordinates corresponding to these two joint crossings to the first two positions of the state cube $\mathcal{C}(T')$. Consequently, any state of $T'$ can be uniquely represented as $(i, j, s)$, where $i, j \in \{0, 1\}$ and $s \in \{0, 1\}^n$ denotes a state of the original tangle $T$. Let $\delta \in \{0, 1, 2\}$ be the number of left-handed crossings at the joint.

\begin{figure}[h]
  \centering
  \includegraphics[width=0.6\textwidth]{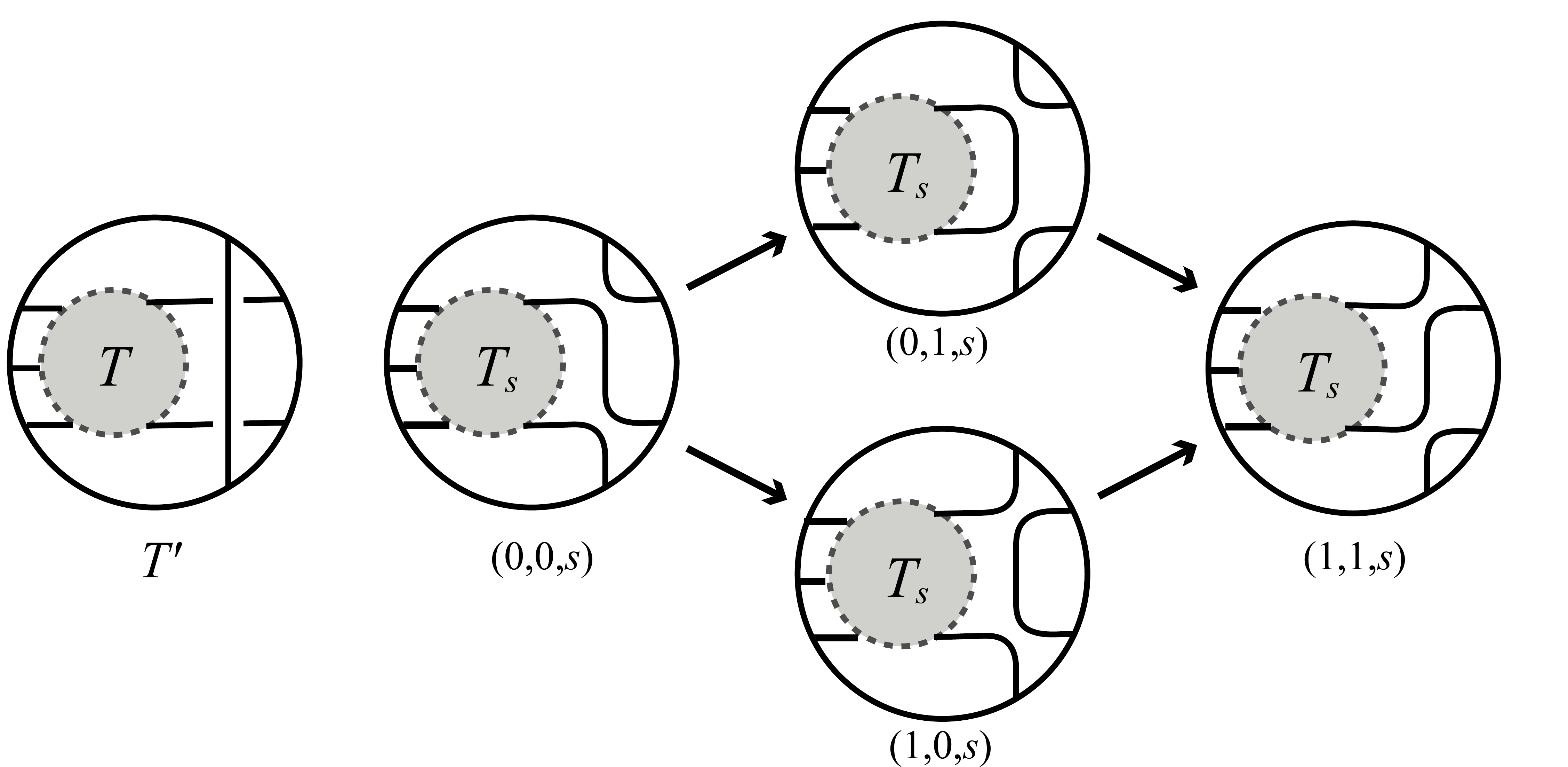}
  \caption{Illustration of the tangle $T'$ and its state cube for the first two coordinates.}\label{figure:arc_reduction1}
\end{figure}
Without loss of generality, we perform the detailed derivation for type \textup{(I)} (as illustrated in Figure~\ref{figure:arc_reduction1}). The resolutions of $T'$ at the joint vertices $(i,j)$ for a fixed state $s \in \mathcal{C}(T)$ are given by
\begin{align*}
    T'_{(0,0,s)} &= T_{s} \sqcup A, \\
    T'_{(0,1,s)} &= \widetilde{T}_{s} \sqcup A \sqcup A, \\
    T'_{(1,0,s)} &= T_{s} \sqcup A, \\
    T'_{(1,1,s)} &= T_{s} \sqcup A,
\end{align*}
where $A$ denotes the embedded arc and $\widetilde{T}$ is the glued link. Let $V = \mathbb{F}[x]/(x^2)$ be the Frobenius algebra associated with the Khovanov invariant, with basis $\{1, x\}$ and internal gradings $\deg(1)=1, \deg(x)=-1$.

\medskip
Let $C = \mathcal{C}(T')$ be the total Khovanov chain complex. We define a descending filtration
\begin{equation*} F^p C = \bigoplus_{i+j \ge p, s} \mathcal{G}(T'_{(i,j,s)}). \end{equation*}
This filtration induces a spectral sequence $\{E_r, d_r\}$. The $E_1$ page is computed as the homology with respect to the vertical differentials, which yields
\begin{equation*}
    \xymatrix{
    \kh^*(T) \ar[rr]^-{d_1^{0,q}} && \kh^*(\widetilde{T}) \oplus \kh^*(T) \ar[rr]^-{d_1^{1,q}} && \kh^*(T).
    }
\end{equation*}
The horizontal differentials $d_1^{p,q}$ are induced by the saddle maps between joint resolutions:
\begin{itemize}
    \item $d_1^{0,q} = (\widetilde{\Delta}, 0)^T$, where $\Delta^0: \kh^*(T) \to \kh^*(\widetilde{T})\cong \kh^*(\widetilde{T})\otimes W$ is the map induced by the co-multiplication $\Delta: W \to V \otimes W$, specifically sending $x \mapsto x \otimes x$ in the junction component.
    \item $d_1^{1,q} = (\nu, 0)$, where $\nu: \kh^*(\widetilde{T}) \cong \kh^*(\widetilde{T})\otimes W\to \kh^*(T)$ is the multiplication map $\mu: V \otimes W \to W$, defined by $1 \otimes x \mapsto x$ and $x \otimes x \mapsto 0$.
\end{itemize}
By Theorem~\ref{theorem:spectral_sequence}, the spectral sequence boundedly collapses at the $E_3$ page. Note that the horizontal differentials originating from and terminating at the direct summand $\kh^*(T)$ within the intermediate layer $E_1^{1,\bullet} = \kh^*(\widetilde{T}) \oplus \kh^*(T)$ vanish identically. Consequently, by the pre-additivity of the category of filtered complexes, this spectral sequence splits into a direct sum of two independent spectral sequences, one of which is purely concentrated on the trivial complex of $\kh^*(T)$.

For the cases of types \textup{(II)} and \textup{(III)}, Theorem~\ref{theorem:spectral_sequence} guarantees that the spectral sequence collapses at the $E_3$ page and converges to $\kh^*(T')$. To establish the sharper collapse at the $E_2$ page, one can adapt the argument from the proof of Theorem~\ref{theorem:hopf_clasp} with necessary modifications. Structurally, the early collapse at the $E_2$ page for these two types is driven by the fact that $E_1^{1,\bullet} \cong \kh^*(T) \oplus \kh^*(T)$ consists of identical isomorphic summands, whose associated row differentials agree up to a sign. Finally, combining this result with the grading isomorphism $\kh^k(T') \cong \bigoplus_p E_2^{p, k+\delta-p}$ completes the proof for types \textup{(II)} and \textup{(III)}.
\end{proof}

The construction of the Hopf sum and arc joint allows us to simplify the computation of Khovanov homology for tangles. By reducing these local interlacing structures, we can lower the computational complexity of Khovanov invariants.

\subsubsection{Khovanov homology of simple tangles}

In \cite{shen2025computing}, a tangle $T$ is defined as \textit{simple} if every arc is pure. It is shown therein that simple tangles possess a particularly rigid homological structure. In this section, we provide a more direct and conceptual proof of this result by employing the tangle spectral sequence.

\begin{definition}
  A pro-tangle $F: \cube^n \to \cob(B)$ is said to be \textit{simple} if the resolution $F(v)$ contains no closed loops for any vertex $v \in \cube^n$.
\end{definition}

It is evident that the definition of a simple tangle in \cite{shen2025computing} is consistent with the notion of a simple pro-tangle in the functorial sense.

\begin{theorem}
Let $F \in \fun^{(n)}_B$ be a simple pro-tangle. Then the Khovanov homology of $F$ is given by
\begin{equation*}
    \kh^{p+\delta}(F) \cong \mathbb{F}^{\binom{n}{p}},
\end{equation*}
where $\delta$ denotes the number of left-handed crossings in the pro-tangle.
\end{theorem}

\begin{proof}
Consider the filtration on the complex $\mathcal{C}(F)$ induced by the weight of the vertices in the hypercube $\cube^n$. By Theorem \ref{theorem:spectral_sequence}, there exists a spectral sequence $E_r^{p,q}$ converging to the Khovanov homology of $F$, with the $E_1$ page given by
\begin{equation*}
    E_{1}^{p,q} \cong \bigoplus_{|s|=p} \kh^{q}(F_s),
\end{equation*}
where $s \in \cube^n$ are the vertices of the cube. 

Since $F$ is a simple pro-tangle, each resolution $F_s$ consists solely of disjoint open arcs. In the Frobenius algebra $V = \mathbb{F}[x]/(x^2)$, the homology of a single arc is associated with the subspace $W = xV$. For a resolution $F_s$ containing $k_s$ arcs, the homology is $\kh^*(F_s) \cong W^{\otimes k_s}$. Since $\dim(W) = 1$, we have $\kh^q(F_s) \cong \mathbb{F}$ for $q=0$ and vanishes otherwise.

The $E_1$ differentials are induced by cobordisms between these arc-only resolutions. Any such cobordism involves either the map $\sigma: W \otimes W \to W\otimes W$
\begin{equation*}
  \sigma: W\otimes W \to W\otimes W,\quad x\otimes x \mapsto 0.
\end{equation*}
Consequently, all $E_1$ differentials vanish, i.e., $d_1 = 0$. This implies that the spectral sequence collapses at the $E_1$ page, yielding $E_{\infty}^{p,q} \cong E_1^{p,q}$. The $E_1$ page is concentrated in row $q=0$, and we have
\begin{equation*}
    E_{1}^{p,0} \cong \bigoplus_{\binom{n}{p}} \mathbb{F} \cong \mathbb{F}^{\binom{n}{p}}.
\end{equation*}
Accounting for the total homological degree $n = p + q + \delta$, we conclude that $\kh^{p+\delta}(F) \cong \mathbb{F}^{\binom{n}{p}}$.
\end{proof}

\section{Planar connected sum}\label{section:pro-tangle_connected_sum}

The calculation of Khovanov homology for complex tangles becomes increasingly formidable as the crossing number grows. Given that connected sums are a natural way to decompose tangles, this section explores the algebraic structure of the resulting Khovanov complexes and establishes several formal descriptions of their properties.

\subsection{Khovanov complexes of connected sums}

\subsubsection{The $V$-module structure of the Khovanov complex}

Let $F: \cube^n \to \cob(B)$ be a pro-tangle. The Khovanov complex $\mathcal{C}(F)$ of $F$ can be regarded as a differential graded module over $V$.

Indeed, for each vertex $v \in \cube^n$, the object $F(v)$ is a 1-manifold consisting of disjoint arcs and circles. Fix a base strand $\alpha$ of the pro-tangle. Recall that the TQFT functor $\mathcal{Z}$ assigns
\begin{equation*}
    \mathcal{Z}(F(v)) = \bigotimes_{i} W_i \otimes \bigotimes_{j} V_j,
\end{equation*}
where $W_i$ and $V_j$ denote the vector spaces associated with arcs and circles, respectively (with $V_j \cong V$ as the defining Frobenius algebra). The choice of the base strand $\alpha$ uniquely singles out a tensor factor $Z \in \{W_i, V_j\}$ in $\mathcal{Z}(F(v))$. 

The $V$-action on the chain complex $\mathcal{C}(F)$ is defined level-wise. For any element $a \in V$, the endomorphism $\mu_a: \mathcal{Z}(F(v)) \to \mathcal{Z}(F(v))$ acts exclusively on the designated factor
\begin{equation*}
    \mu_a (\xi_1 \otimes \dots \otimes z \otimes \dots \otimes \xi_k) = \xi_1 \otimes \dots \otimes (z \cdot a) \otimes \dots \otimes \xi_k,
\end{equation*}
where $z \in Z$, and $z \cdot a$ is given by the natural $V$-module structure on $Z$. Geometrically, when $a$ corresponds to a generator of $V$, this algebraic action is induced via the TQFT functor by a saddle cobordism in $\cob(B)$ that merges an auxiliary unknot into the base strand $\alpha$.

\begin{proposition}
Let $F \in \fun^{(n)}_B$ be a pro-tangle with a designated base strand $\alpha$. Then the Khovanov chain complex $\mathcal{C}(F)$ inherits the structure of a differential graded $V$-module.
\end{proposition}

\begin{proof}
The global Khovanov differential $d_F$ is defined as the sum of the linear maps induced by the local elementary cobordisms in $\cob(B)$. To establish that $\mathcal{C}(F)$ is a differential graded $V$-module, it suffices to show that the $V$-action $\mu_a$ commutes with each edge differential $d_{v \to v'}$ for adjacent vertices $v, v' \in \cube^n$. That is, for any $\xi \in \mathcal{Z}(F(v))$ and $a \in V$, we must verify
\begin{equation*}
    d_{v\to v'}(\mu_a(\xi)) = \mu_a(d_{v\to v'}(\xi)).
\end{equation*}
We analyze the geometric support of the local cobordism underlying $d_{v \to v'}$ via two distinct cases:
\begin{enumerate}[label=$(\roman*)$]
    \item \textit{Disjoint support:} If the geometric support of the local cobordism $d_{v \to v'}$ is disjoint from the base strand $\alpha$, the map $d_{v \to v'}$ and the $V$-action $\mu_a$ operate on distinct tensor factors of $\mathcal{Z}(F(v))$. Consequently, the commutativity $d_{v \to v'} \circ \mu_a = \mu_a \circ d_{v \to v'}$ follows directly from the functoriality of the tensor product.
    
    \item \textit{Intersecting support:} If $d_{v \to v'}$ involves the base strand $\alpha$, the local cobordism represents either a multiplication (merging a circle into $\alpha$) or a comultiplication (splitting a circle from $\alpha$). Since $V$ is a commutative Frobenius algebra and the assignment $Z$ on $\alpha$ is a well-defined compatible $V$-module, this compatibility translates precisely to the associativity and coassociativity of the underlying Frobenius structure.
\end{enumerate}
In both configurations, the local compatibility ensures the commutativity of the following diagram at the chain level:
\begin{equation*}
    \begin{tikzcd}
    \mathcal{Z}(F(v)) \arrow[r, "d_{v \to v'}"] \arrow[d, "\mu_a"'] & \mathcal{Z}(F(v')) \arrow[d, "\mu_a"] \\
    \mathcal{Z}(F(v)) \arrow[r, "d_{v \to v'}"'] & \mathcal{Z}(F(v'))
    \end{tikzcd}
\end{equation*}
Summing over all edges of the hypercube $\cube^n$, we obtain $d_F \circ \mu_a = \mu_a \circ d_F$, which establishes the $V$-linearity of the global differential and completes the proof.
\end{proof}

\subsubsection{Tensor product formula for connected sums}

In this section, we formalize the relationship between the connected sum of pro-objects and the tensor product over the Frobenius algebra $V$. This result allows for the decomposition of complex pro-tangles.

Let $F \in \fun^{(n)}_B$ be a pro-tangle and $L \in \fun^{(m)}_{\emptyset}$ be a pro-link. Suppose $R = (\alpha, \beta)$ is a pair consisting of a strand $\alpha \subset F$ and a component $\beta \subset L$ along which the connected sum $F \#_R L$ is formed.

\begin{theorem}\label{theorem:connected_sum_isomorphism}
There exists a natural isomorphism of chain complexes
\begin{equation*}\label{eq:tensor_iso}
    \mathcal{C}(F \#_R L) \cong \mathcal{C}(F) \otimes_V \mathcal{C}(L),
\end{equation*}
where $\otimes_V$ denotes the tensor product over the Frobenius algebra $V = \mathbb{F}[x]/(x^2)$.
\end{theorem}

\begin{proof}
The proof establishes a degree-by-degree identification of the state spaces and verifies the strict commutativity of the differentials.

First, consider the underlying graded vector spaces. For any multi-index $(v, w) \in \cube^n \times \cube^m$, the state space of the disjoint union $F \sqcup L$ is the $\mathbb{F}$-tensor product $\mathcal{Z}(F(v)) \otimes  \mathcal{Z}(L(w))$. In the Khovanov TQFT, the connected sum at $R$ corresponds to the application of the multiplication map $m: V \otimes V \to V$ (if $R$ joins two circles) or the module action $\mu: W \otimes V \to W$ (if $R$ joins a strand and a circle). At each vertex $(v, w)$, the surgery at $R$ collapses the two independent $V$-factors into a single shared factor. This precisely mimics the defining relation of the tensor product over $V$. Thus, we obtain a canonical isomorphism of graded vector spaces
\begin{equation*}
    \kappa_{v,w}: \mathcal{Z}((F \#_R L)(v, w)) \xrightarrow{\cong} \mathcal{Z}(F(v)) \otimes_V \mathcal{Z}(L(w)).
\end{equation*}
Summing over all vertices $(v, w) \in \cube^n \times \cube^m$, this induces a canonical grading-preserving isomorphism of chain groups
\begin{equation*}
    \kappa = \bigoplus_{(v,w)} \kappa_{v,w}: \mathcal{C}(F \#_R L) \xrightarrow{\cong} \mathcal{C}(F) \otimes_V \mathcal{C}(L).
\end{equation*}

Next, we demonstrate that $\kappa$ is a chain map with respect to the total differentials. Let $\mathcal{C}(F) \otimes_V \mathcal{C}(L)$ be equipped with the standard differential graded (DG) tensor product differential $d_{\otimes}$, which respects the Koszul sign rule
\begin{equation*}
    d_{\otimes}(x \otimes_{V} y) = d_F(x) \otimes_{V} y + (-1)^{|x|} x \otimes_{V} d_L(y)
\end{equation*}
for any homogeneous element $x \in \mathcal{C}(F)$ with homological degree $|x|$. 

The total differential $d_{F \# L}$ on the connected sum is composed of per-edge cobordisms. For an edge corresponding to an internal differential of $F$ or $L$ not involving the joint $R$, the strict commutativity with $\kappa$ follows from the functoriality of the tensor product, where the sign $(-1)^{|x|}$ precisely tracks the cubical ordering of the product hypercube $\cube^n \times \cube^m$. For an elementary edge that acts directly via a saddle cobordism at the joint $R$, the associativity and coassociativity of the commutative Frobenius algebra $V$ guarantee that the algebraic actions match. Consequently, we obtain the exact relation
\begin{equation*}
    d_{F \# L} \circ \kappa = \kappa \circ d_{\otimes}.
\end{equation*}
Since $\kappa$ is a bijection that commutes with the total differentials, it is an isomorphism of differential graded $V$-modules.
\end{proof}

\begin{remark}
The isomorphism in Theorem \ref{theorem:connected_sum_isomorphism} implies that the homological complexity of connected sums is entirely governed by the local $V$-module structures at the joint. In addition, this modularity is essential for the convergence of the spectral sequences in Theorems \ref{theorem:hopf_sum} and \ref{theorem:hopf_twist}.
\end{remark}

\subsubsection{K\"{u}nneth spectral sequence of connected sums}

Let $F \in \fun^{(n)}_B$ be a pro-tangle and $L \in \fun^{(m)}_{\emptyset}$ be a pro-link. Suppose $F \#_R L$ is the connected sum along a pair $R = (\alpha, \beta)$ of strands. The isomorphism $\mathcal{C}(F \#_R L) \cong \mathcal{C}(F) \otimes_V \mathcal{C}(L)$ leads to a spectral sequence.

\begin{theorem}\label{theorem:kunneth}
Let $F \in \fun^{(n)}_B$ be a pro-tangle and $L \in \fun^{(m)}_{\emptyset}$ be a pro-link. There exists a spectral sequence converging to the Khovanov homology of the connected sum
\begin{equation*}
    E_2^{p,q} = \bigoplus_{i+j=q} \mathrm{Tor}_p^{V} \left( \kh^i(F), \kh^j(L) \right) \implies \kh^{p+q}(F \#_R L).
\end{equation*}
\end{theorem}

\begin{proof}
By Theorem~\ref{theorem:connected_sum_isomorphism}, the total chain complex $\mathcal{C}(F \#_R L)$ is isomorphic to the tensor product of DG $V$-modules $\mathcal{C}(F) \otimes_V \mathcal{C}(L)$. We construct the spectral sequence by replacing $\mathcal{C}(L)$ with a semi-projective resolution $\mathcal{P}_\bullet \xrightarrow{\sim} \mathcal{C}(L)$ in the category of DG $V$-modules.

Note that since $F$ and $L$ are objects in the functor categories over the finite Boolean cubes $\cube^n$ and $\cube^m$ respectively, their associated Khovanov complexes $\mathcal{C}(F)$ and $\mathcal{C}(L)$ are intrinsically bounded below and finitely generated as $V$-modules at each homological degree. This structural finiteness ensures that the standard filtration on the total complex, induced by the resolution degree $p$ of $\mathcal{P}_\bullet$, is bounded below and exhaustive.

Taking cohomology with respect to the internal differentials $d_F$ and $d_L$ identifies the $E_1$ page as $E_1^{p,q} = \bigoplus_{i+j=q} (\kh^i(F) \otimes_V \mathcal{P}_p)_j$. Subsequently, taking homology with respect to the induced differential from the resolution yields the $E_2$ page as the algebraic $\mathrm{Tor}$ groups over $V$:
\begin{equation*}
    E_2^{p,q} = \bigoplus_{i+j=q} \mathrm{Tor}_p^{V} \big( \kh^i(F), \kh^j(L) \big).
\end{equation*}
The desired result follows.
\end{proof}

If $\kh^*(L)$ is a projective (or free) $V$-module, the spectral sequence collapses at the $E_2$ page. In this case, the homology admits a direct decomposition
\begin{equation*}
    \kh^*(F \#_R L) \cong \kh^*(F) \otimes_{V} \kh^*(L).
\end{equation*}
When $L = S^1$, the connected sum $F \#_R L$ is isomorphic to $F$, and the above decomposition reduces to the identity isomorphism.

\begin{example}
Consider the case where the pro-tangle $F$ is a single arc and the pro-link $L$ is a right-handed Hopf link. The connected sum $F \#_R L$ at a strand $R$ is illustrated as in Figure \ref{figure:hopf_arc_sum}. 
\begin{figure}[h]
  \centering
  \includegraphics[width=0.5\textwidth]{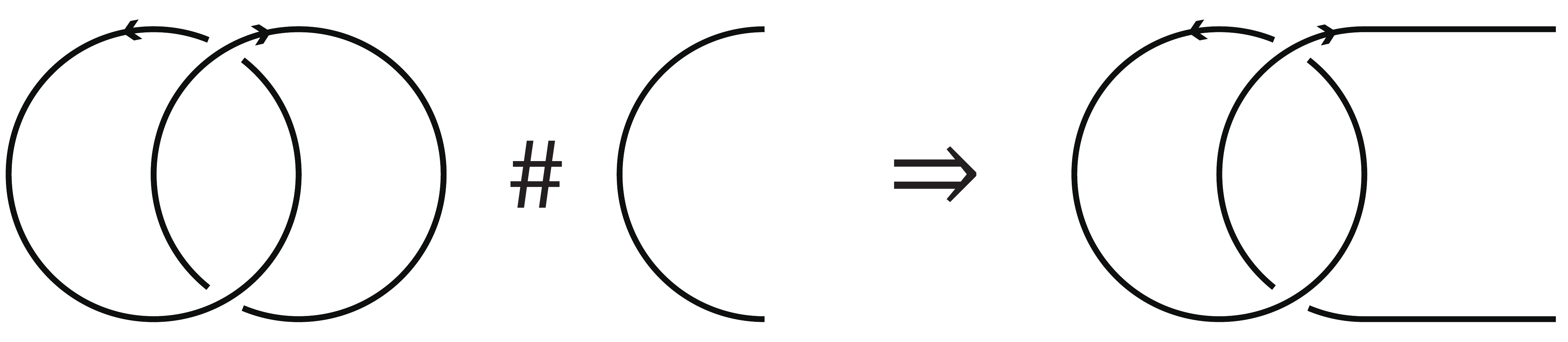} 
  \caption{Illustration of the connected sum of an arc and a right-handed Hopf link.}\label{figure:hopf_arc_sum}
\end{figure}
The Khovanov homology of the arc $F$ is
\begin{equation*}
  \kh^p(F) \cong 
  \begin{cases*}
    \spann\{ [x] \}, & if $p=0$, \\
    0, & otherwise.
  \end{cases*}
\end{equation*}
Note that $\kh^0(F)$ is isomorphic to the ideal $W = xV \subset V$, which is not a projective $V$-module. For the right-handed Hopf link $L$, the Khovanov homology is given by
\begin{equation*}
  \kh^p(L) \cong 
  \begin{cases*}
    \spann\{ [x \otimes x], [x \otimes 1 - 1 \otimes x] \}, & if $p=0$, \\
    \spann\{ [1 \otimes x], [1 \otimes 1] \}, & if $p=2$, \\
    0 & otherwise.
  \end{cases*}
\end{equation*}

We now compute the tensor product $\kh^*(F) \otimes_V \kh^*(L)$. Specifically, we have
\begin{align*}
  [x] \otimes_V [x \otimes x]  &= 0, \\
  [x] \otimes_V [x \otimes 1 - 1 \otimes x] &= - [x \otimes_V 1 \otimes x].
\end{align*}
Similar calculations for $p=2$ show that only $[x \otimes_V 1 \otimes 1]$ survives the contraction. We obtain
\begin{equation*}
    \kh^*(F) \otimes_{V} \kh^*(L) \cong 
    \begin{cases*}
      \spann\{ [x \otimes_V 1 \otimes x] \}, & if $p=0$, \\
      \spann\{ [x \otimes_V 1 \otimes 1] \}, & if $p=2$, \\
      0 & otherwise.
    \end{cases*}
\end{equation*}
Direct computation of the connected sum $F \#_R L$  confirms that its Khovanov homology is indeed generated by $[x \otimes x]$ in degree $0$ and $[1 \otimes x]$ in degree $2$. Remarkably, we have the isomorphism
\begin{equation*}
    \kh^*(F \#_R L) \cong \kh^*(F) \otimes_{V} \kh^*(L).
\end{equation*}
Taking the quantum grading $Q$ into account, we have
\begin{align*}
    Q([x \otimes x]) &= 0 + 2 - 0 - 2 = 0, \\
    Q([1 \otimes x]) &= 2 + 2 - 0 + 0 = 4.
\end{align*}
Consequently, the Poincaré polynomial of the connected sum is
\begin{equation*}
    P(F \#_R L; t, q) = 1 + t^2 q^4.
\end{equation*}
This result is consistent with the computation
\begin{equation*}
    P(F \#_R L; t, q) = q^{-1}(q + t^2 q^5) = 1 + t^2 q^4
\end{equation*}
derived in Section \ref{section:applications}.
\end{example}

Note that a cut tangle of a link $L$ along an arc can be identified with the connected sum $L\# A$ of $L$ and an arc $A$. By Theorem \ref{theorem:kunneth}, we have a spectral sequence convergent to the Khovanov homology of the cut tangle.

\begin{theorem}
Let $L$ be any link. Then there is a spectral sequence convergent to the Khovanov homology of the cut tangle $\kh^*(\widehat{L})$ with the $E_2$ page given by
\begin{equation*}
    E_2^{p,q} \cong 
    \begin{cases} 
       \kh^q(L) \big/ x\kh^q(L), & p = 0, \\
      \ker\big(x: \kh^q(L) \to \kh^q(L)\big) \big/ x\kh^q(L), & p > 0.
    \end{cases}
\end{equation*}
In particular, if $\kh^*(L)$ is a free $V$-module, the spectral sequence collapses at $E_2$ to the $p=0$ column.
\end{theorem}

\begin{proof}
Note that a cut tangle of a link $L$ along an arc can be identified with the connected sum $L\# A$ of $L$ and an arc $A$. By Theorem~\ref{theorem:kunneth}, we have a spectral sequence convergent to the Khovanov homology of the cut tangle.

Since $A$ is a single arc, its Khovanov homology is concentrated in degree $0$, with $\kh^0(A) \cong W$. To compute the $E_2$ page, we use the periodic free resolution of $W$ over $V$ as follows
\begin{equation*}
    \mathcal{P}_\bullet: \cdots \xrightarrow{\cdot x} V \xrightarrow{\cdot x} V \xrightarrow{\cdot x} V \xrightarrow{\iota} W \to 0,
\end{equation*}
where $\iota(1) = x$. This sequence is exact as $\ker(\cdot x) = \operatorname{im}(\cdot x) = xV$.

Applying the tensor product $(-) \otimes_V \kh^*(L)$ to the resolution $\mathcal{P}$, we obtain the following complex of graded vector spaces
\begin{equation*}
    \cdots \xrightarrow{\cdot x} \kh^q(L) \xrightarrow{\cdot x} \kh^q(L) \xrightarrow{\cdot x} \kh^q(L).
\end{equation*}
By the definition of the torsion product $\mathrm{Tor}_p^V(W, \kh^q(L))$, the cohomology of this complex at any stage $p > 0$ is the quotient $\ker(x)/\operatorname{im}(x)$. Thus, we have
\begin{equation*}
  E_2^{p,q} = \ker\big(x: \kh^q(L) \to \kh^q(L)\big) \big/ x\kh^q(L).
\end{equation*}
For $p=0$, we have
\begin{equation*}
  E_2^{0,q} = W \otimes_V \kh^{q}(L) \cong \kh^{q}(L) \big/ x\kh^{q}(L).
\end{equation*}
This completes the proof.
\end{proof}

This result establishes a fundamental bridge between the local data of a tangle, obtained by cutting a link, and the global Khovanov homology of the original link. This provides a structural view for understanding how the global topological information of a link is distributed across its local constituents and how these components are integrated via the derived tensor product.

\subsubsection{Algebraic fibration of connected sums}

Let $L$ be a pro-link, and let $\alpha$ be a base strand. The augmentation map $\epsilon_L: \mathcal{C}(L) \to V$ is constructed by applying the Frobenius counit $\iota: V \to \mathbb{F}$ to all circle components of the fiber resolution $L(w)$ except for the one involved in the surgery at $\alpha$. Formally, for a state $\xi = z_\alpha \otimes (\bigotimes_{i} z_i) \in \mathcal{C}(L)$, where $z_\alpha$ resides on the joint component, we have
\begin{equation*}
    \epsilon_L(\xi) = \left( \prod_{i} \iota(z_i) \right) z_\alpha \in V.
\end{equation*}

Let $F \in \fun^{(n)}_B$ be a pro-tangle, and let $L \in \fun^{(m)}_{\emptyset}$ be a pro-link. Consider the connected sum $F \#_R L$ along a pair of strands $R = (\alpha, \beta)$.
The Khovanov complex of connected sum $\mathcal{C}(F \#_R L)$ gives an algebraic fibration. More precisely, we consider the fibration of pro-tangles as a bundle $\mathfrak{E} = (\mathcal{C}(F \#_R L), \pi, \mathcal{C}(F), \mathcal{C}(L))$, where
\begin{equation*}
  \mathcal{E}=\mathcal{C}(F \#_R L) \cong\mathcal{C}(F) \otimes_V \mathcal{C}(L)
\end{equation*}
is the total space, and $\mathcal{B} =\mathcal{C}(F)$ is the base space. 
The projection $\pi: \mathcal{E} \to \mathcal{B}$ is defined as the chain map induced by the contraction of the fiber factor over the Frobenius algebra $V$. Specifically, for any tensor element $c_F \otimes_V c_L \in \mathcal{C}(F) \otimes_V \mathcal{C}(L)$, we define
\begin{equation*}
    \pi(c_F \otimes_V c_L) = \mu_F(c_F \otimes \epsilon_L(c_L)),
\end{equation*}
where $\mu_F: \mathcal{C}(F) \otimes V \to \mathcal{C}(F)$ is the canonical $V$-module action at the joint $R$, and $\epsilon_L: \mathcal{C}(L) \to V$ is the augmentation map at $\alpha$.

The algebraic fiber $K$ of the fibration $\mathfrak{E}$ is the sub-complex of $\mathcal{E}$ defined by the kernel of the projection map
\begin{equation*}
    K =\ker \left( \pi: \mathcal{C}(F) \otimes_V \mathcal{C}(L) \to \mathcal{C}(F) \right).
\end{equation*}

Let $\widetilde{\mathcal{C}}(L, \alpha) = \ker(\epsilon_L: \mathcal{C}(L) \to V)$ be the reduced Khovanov complex of $L$ relative to the surgery strand $\alpha$.
The following proposition characterizes $K$ in terms of the reduced Khovanov complex of the pro-link $L$.

\begin{proposition}
The algebraic fiber $K$ is isomorphic to the tensor product of the base tangle complex and the reduced graded module of the link over $V$. Specifically, there is a canonical isomorphism of graded $V$-modules
\[
    K \cong \mathcal{C}(F) \otimes_V \widetilde{\mathcal{C}}(L, \alpha).
\]
\end{proposition}

\begin{proof}
Define the reduced graded module $\widetilde{\mathcal{C}}(L, \alpha) = \ker(\epsilon_L: \mathcal{C}(L) \to V)$. By construction, there is a short exact sequence of graded $V$-modules
\[
    0 \longrightarrow \widetilde{\mathcal{C}}(L, \alpha) \xrightarrow{\iota} \mathcal{C}(L) \xrightarrow{\epsilon_L} V \longrightarrow 0.
\]
We note that $\epsilon_L$ is not in general a chain map, so $\widetilde{\mathcal{C}}(L, \alpha)$ does not inherit a natural differential graded structure.

Applying the right exact functor $\mathcal{C}(F) \otimes_V (-)$ to the above short exact sequence, and noting that $\mathrm{Tor}_1^V(\mathcal{C}(F), V) = 0$ since $V$ is a free $V$-module, we obtain a short exact sequence of graded $V$-modules
\[
    0 \longrightarrow \mathcal{C}(F) \otimes_V \widetilde{\mathcal{C}}(L, \alpha) \longrightarrow \mathcal{C}(F) \otimes_V \mathcal{C}(L) \xrightarrow{\mathrm{id} \otimes \epsilon_L} \mathcal{C}(F) \otimes_V V \longrightarrow 0.
\]
Under the canonical identification $\mathcal{C}(F) \otimes_V V \cong \mathcal{C}(F)$ given by the module action, the map $\mathrm{id} \otimes \epsilon_L$ corresponds exactly to the projection $\pi$. Therefore, $\ker \pi \cong \mathcal{C}(F) \otimes_V \widetilde{\mathcal{C}}(L, \alpha)$ as graded $V$-modules, as claimed.
\end{proof}

The choice of $\mathcal{C}(F)$ as the base space is motivated by the fact that the pro-link $L$, being a closed topological object, admits a canonical augmentation $\epsilon_L$ towards the Frobenius algebra $V$. This allows for a natural contraction of the fiber factor while preserving the skeletal boundary information of the pro-tangle $F$ in the base.

\subsection{Multiple connected sums}

The structural realization of connected sums can be extended to cases where multiple pairs of strands are joined simultaneously. To establish the algebraic approach for such simultaneous operations, we formalize how the associated chain complexes inherit a higher-rank module structure from multiple boundary strands.

\subsubsection{The algebraic isomorphism over $V^{\otimes k}$-module}
Let $F \in \fun^{(n)}_B$ be a pro-tangle, and let $\boldsymbol{\alpha} = \{\alpha_1, \dots, \alpha_k\}$ be an ordered collection of $k$ distinct designated base strands in $F$. For each vertex $v \in \cube^n$, the TQFT functor $\mathcal{Z}$ assigns a tensor product of vector spaces corresponding to arcs and circles. The choice of the multi-strand $\boldsymbol{\alpha}$ uniquely singles out $k$ distinct tensor factors $Z_1, \dots, Z_k$ in $\mathcal{Z}(F(v))$. 

The $V^{\otimes k}$-module action on the Khovanov complex $\mathcal{C}(F)$ is defined level-wise. For any pure tensor $a_1 \otimes \dots \otimes a_k \in V^{\otimes k}$, the endomorphism $\mu_{\mathbf{a}}: \mathcal{Z}(F(v)) \to \mathcal{Z}(F(v))$ acts simultaneously and independently on the designated factors:
\begin{equation*}
    \mu_{\mathbf{a}} (\xi_1 \otimes \dots \otimes z_1 \otimes \dots \otimes z_k \otimes \dots \otimes \xi_m) = \xi_1 \otimes \dots \otimes (z_1 \cdot a_1) \otimes \dots \otimes (z_k \cdot a_k) \otimes \dots \otimes \xi_m,
\end{equation*}
where $z_i \in Z_i$, and $z_i \cdot a_i$ is given by the natural $V$-module structure on each component. Geometrically, this algebraic action is induced by $k$ disjoint saddle cobordisms that merge $k$ auxiliary unknots into the respective base strands $\alpha_1, \dots, \alpha_k$. Since these geometric operations are localized and disjoint, their corresponding actions commute, equipping $\mathcal{C}(F)$ with the structure of a differential graded module over the tensor algebra $V^{\otimes k}$.

\begin{theorem}\label{theorem:connected_sum_isomorphism2}
Let $F \in \fun^{(n)}_B$ be a pro-tangle, and let $L = \bigsqcup_{i=1}^k L_i$ be a pro-link consisting of $k$ disjoint connected components $L_i \in \fun^{(m_i)}_\emptyset$ with $\sum m_i = m$. Let $\mathbf{R} = \{R_1, \dots, R_k\}$ be a collection of $k$ distinct junction pairs, where each $R_i = (\alpha_i, \beta_i)$ connects a distinct strand $\alpha_i \subset F$ with a strand $\beta_i \subset L_i$ belonging to the $i$-th component of $L$. Then there is a natural isomorphism of complexes
\begin{equation*}
    \mathcal{C}(F \#_{\mathbf{R}} L) \cong \mathcal{C}(F) \otimes_{V^{\otimes k}} \mathcal{C}(L),
\end{equation*}
where the $V^{\otimes k}$-action on $\mathcal{C}(L)$ is given by the component-wise action of each $V$-factor on the respective $\mathcal{C}(L_i)$ factor in $\mathcal{C}(L) \cong \bigotimes_{i=1}^k \mathcal{C}(L_i)$.
\end{theorem}

\begin{proof}
The proof proceeds by induction on the number of junctions $k$, reducing the assertion to the single-junction case established in Theorem \ref{theorem:connected_sum_isomorphism}.

For the base case $k=1$, the statement holds by Theorem \ref{theorem:connected_sum_isomorphism}. Now, assume the natural isomorphism holds for any collection of $k-1$ junctions. Since $L$ consists of $k$ disjoint connected components, we can isolate the $k$-th component $L_k$ and decompose the simultaneous connected sum as a two-step iterative process
\begin{equation*}
    F \#_{\mathbf{R}} L \cong \left( F \#_{\mathbf{R} \setminus \{R_k\}} \bigsqcup_{i=1}^{k-1} L_i \right) \#_{R_k} L_k.
\end{equation*}
By applying the induction hypothesis to the first $k-1$ disjoint components, we obtain a natural isomorphism of chain complexes
\begin{equation*}
    \mathcal{C}\left( F \#_{\mathbf{R} \setminus \{R_k\}} \bigsqcup_{i=1}^{k-1} L_i \right) \cong \mathcal{C}(F) \otimes_{V^{\otimes (k-1)}} \left( \bigotimes_{i=1}^{k-1} \mathcal{C}(L_i) \right).
\end{equation*}
Next, treating the entire left-hand side as a single pro-tangle and applying Theorem \ref{theorem:connected_sum_isomorphism} for the remaining single junction $R_k$ between $F$ and $L_k$, the total chain complex decomposes further
\begin{align*}
    \mathcal{C}(F \#_{\mathbf{R}} L) &\cong \left[ \mathcal{C}(F) \otimes_{V^{\otimes (k-1)}} \left( \bigotimes_{i=1}^{k-1} \mathcal{C}(L_i) \right) \right] \otimes_V \mathcal{C}(L_k) \\
    &\cong \mathcal{C}(F) \otimes_{V^{\otimes k}} \left( \bigotimes_{i=1}^k \mathcal{C}(L_i) \right) \\
    &\cong \mathcal{C}(F) \otimes_{V^{\otimes k}} \mathcal{C}(L).
\end{align*}
Since each step preserves the total differential and respects the Koszul sign rule by the functoriality of the nested tensor products, the combined map $\kappa$ defines an isomorphism of chain complexes.
\end{proof}

Theorem \ref{theorem:connected_sum_isomorphism2} provides a systematic methodology to reduce the homological analysis of a complex tangle into simpler, more rigid constituents, where the global topological complexity is entirely captured by the tensor interaction over $V^{\otimes k}$.

\begin{theorem}
Let $F \in \fun^{(n)}_B$ be a pro-tangle, and let $L = \bigsqcup_{i=1}^k L_i$ be a pro-link consisting of $k$ disjoint connected components $L_i \in \fun^{(m_i)}_\emptyset$ with $\sum m_i = m$. Let $\mathbf{R} = \{R_1, \dots, R_k\}$ be a collection of $k$ distinct junction pairs, where each $R_i = (\alpha_i, \beta_i)$ connects a distinct strand $\alpha_i \subset F$ with a strand $\beta_i \subset L_i$ belonging to the $i$-th component of $L$.
There exists a spectral sequence convergent to the Khovanov homology of the multiple connected sum
\begin{equation*}
    E_2^{p,q} = \bigoplus_{i+j=q} \mathrm{Tor}_p^{V^{\otimes k}} \left( \kh^i(F), \kh^j(L) \right) \implies \kh^{p+q}(F \#_{\mathbf{R}} L).
\end{equation*}
\end{theorem}

\begin{proof}
By the generalized tensor isomorphism, we have an isomorphism of complexes
\begin{equation*}
    \mathcal{C}(F \#_{\mathbf{R}} L) \cong \mathcal{C}(F) \otimes_{V^{\otimes k}} \mathcal{C}(L).
\end{equation*}
We resolve $\mathcal{C}(L)$ by a semi-projective resolution $\mathcal{P}_\bullet \to \mathcal{C}(L)$ in the category of differential graded $V^{\otimes k}$-modules. The total complex is then filtered by the resolution degree $p$, leading to the standard spectral sequence for the derived tensor product. 

The $E_1$ page consists of the homologies of the internal differentials $d_F$ and $d_L$ as follows:
\begin{equation*}
    E_1^{p,q} = \kh^q(\mathcal{C}(F) \otimes_{V^{\otimes k}} \mathcal{P}_p).
\end{equation*}
Identifying the homology of the internal complexes as $\kh^*(F)$ and $\kh^*(L)$ yields the $E_2$ page as the $p$-th derived functor of the product
\begin{equation*}
    E_2^{p,q} \cong \bigoplus_{i+j=q} \mathrm{Tor}_p^{V^{\otimes k}} \left( \kh^i(F), \kh^j(L) \right) \implies \kh^{p+q}(F \#_{\mathbf{R}} L).
\end{equation*}
The $E_2$ page is computed over the $k$-fold tensor product of the Frobenius algebra cohomology $V^{\otimes k} \cong \mathbb{F}[x_1, \dots, x_k] / (x_1^2, \dots, x_k^2)$.
The convergence follows from the boundedness of the complexes and the finite dimensionality of $V^{\otimes k}$.
\end{proof}

\begin{remark}
It is crucial to emphasize that the disjointness of the components $L_i$ in Theorem \ref{theorem:connected_sum_isomorphism2} cannot be relaxed. If multiple junctions $\mathbf{R}$ are formed between $F$ and the \textit{same} connected component of a pro-link $L$, the isomorphism fails. This forces multiple factors of the tensor product $V^{\otimes k}$ to act non-trivially on the same single copy of $V$ within $\mathcal{Z}(L)$, collapsing the expected modular freedom of $\otimes_{V^{\otimes k}}$ and preventing the total complex from factoring as a simple tensor product.
\end{remark}

\begin{example}\label{example:counterexample_loop}
To explicitly demonstrate the failure of Theorem \ref{theorem:connected_sum_isomorphism2} when multiple junctions are formed along the same connected component, consider the case where both the pro-tangle $F$ and the pro-link $L$ are single circles, i.e., $F = S^1$ and $L = S^1$. Let $\mathbf{R} = \{R_1, R_2\}$ be a collection of $k=2$ distinct junction pairs connecting $F$ and $L$ along two pairs of strands. We obtain the multiple connected sum at $\mathbf{R} = \{R_1, R_2\}$, as shown in Figure \ref{figure:multiple_sum}.
\begin{figure}[h]
  \centering
  \includegraphics[width=0.4\textwidth]{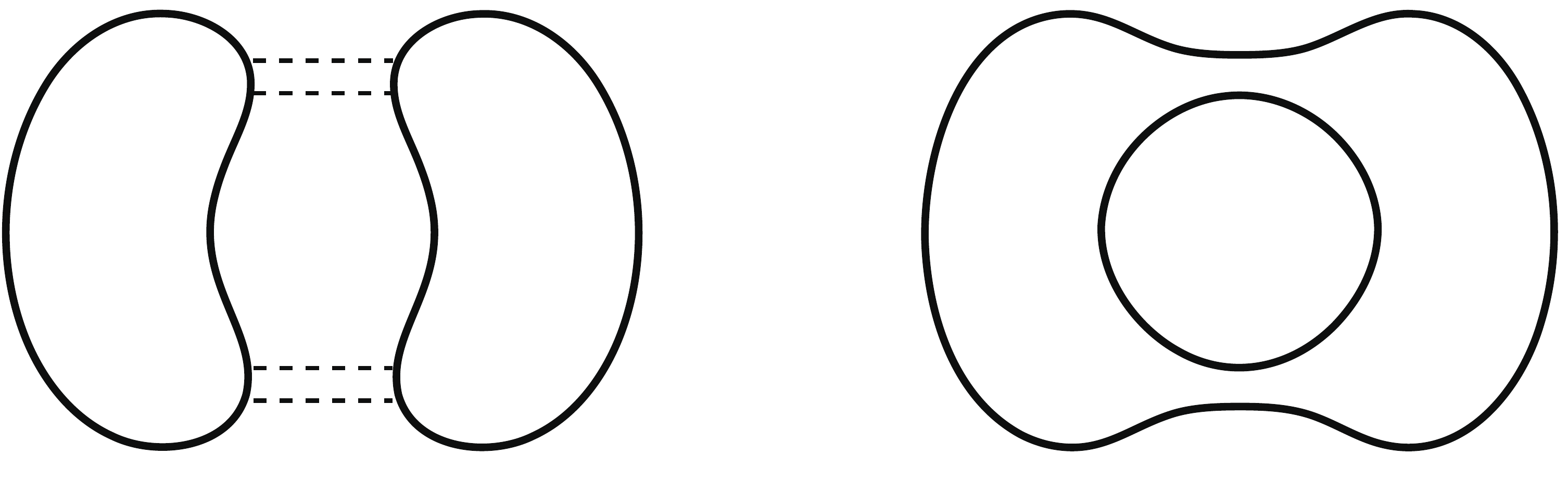} 
  \caption{The multiple connected sum of $S^1$ and $S^1$ at $\mathbf{R} = \{R_1, R_2\}$.}\label{figure:multiple_sum}
\end{figure}

On the left-hand side, performing two independent geometric surgeries between two disjoint circles yields a topological manifold consisting of two disjoint circles. In the Khovanov TQFT, the chain complex of this combined geometric object is simply the standard tensor product of two independent Frobenius algebras over the base field
\begin{equation*}
    \mathcal{C}(F \#_{\mathbf{R}} L) \cong V \otimes  V.
\end{equation*}
The dimension of $\mathcal{C}(F \#_{\mathbf{R}} L)$ over $\mathbb{F}$ is 4.

On the right-hand side, the chain complexes of the individual components are $\mathcal{C}(F) \cong V$ and $\mathcal{C}(L) \cong V$. Here, both complexes are viewed as $V \otimes  V$-modules. Explicitly, the module structure is defined via the canonical multiplication map $V \otimes  V \to V$, which identifies the module with the quotient space
\begin{equation*}
    V \cong (V \otimes  V) / \langle x \otimes 1 - 1 \otimes x \rangle.
\end{equation*}
The left and right actions of an element $(a_1 \otimes a_2) \in V \otimes  V$ on an element $y \in V$ are given by standard modular multiplication: $(a_1 \otimes a_2) \cdot y = a_1 a_2 y$. Now, we compute the balanced tensor product over $V \otimes  V$ as follows:
\begin{align*}
    \mathcal{C}(F) \otimes_{V \otimes V} \mathcal{C}(L) &\cong V \otimes_{V \otimes V} V \\
    &\cong V \otimes_{V \otimes V} \left( (V \otimes  V) / \langle x \otimes 1 - 1 \otimes x \rangle \right) \\
    &\cong V / \langle x \cdot 1 - 1 \cdot x \rangle \\
    &\cong V.
\end{align*}
Evaluating the dimension of this resulting algebraic tensor product, we have
\begin{equation*}
    \dim_{\mathbb{F}} \left( \mathcal{C}(F) \otimes_{V \otimes V} \mathcal{C}(L) \right) = \dim_{\mathbb{F}} V = 2.
\end{equation*}

Comparing the two sides, we observe a strict dimensional mismatch ($4 \neq 2$). Thus, the chain complexes cannot be isomorphic
\begin{equation*}
    \mathcal{C}(F \#_{\mathbf{R}} L) \not\cong \mathcal{C}(F) \otimes_{V \otimes V} \mathcal{C}(L).
\end{equation*}
This divergence arises because the algebraic tensor product overdetermines the relations, collapsing the independent geometric degrees of freedom generated by the two distinct junctions.
\end{example}

To understand and rectify this discrepancy, we observe the precise interplay between the geometric components and the algebraic modules. In the tensor product $V \otimes_{V \otimes V} V \cong V$, the two original copies of $V$ corresponding to $F$ and $L$ are amalgamated into a single $V$. Topologically, this single $V$ represents the underlying circle that actively carries the junction strands and supports the mutual module actions.

Conversely, on the geometric side, the simultaneous connection via two separate junctions creates a closed loop system. This parallel geometric surgery produces an additional, disjoint circle in the state space of $F \#_{\mathbf{R}} L$. This second circle does not participate in the core modular intertwining; rather, it is extraneously generated as a topological consequence of the parallel routing. 

To compensate for this algebraic deficit and capture the degree of freedom represented by this extra circle, we must enrich the right-hand side with a free $V$-factor representing the unlinked component. The corrected isomorphism for this specific configuration reads
\begin{equation*}
    \mathcal{C}(F \#_{\mathbf{R}} L) \cong \left(\mathcal{C}(F) \otimes_{V \otimes V} \mathcal{C}(L)\right) \otimes  V.
\end{equation*}
Evaluating the dimensions now yields $\dim_{\mathbb{F}}(V \otimes_{V \otimes V} V \otimes  V) = 2 \times 2 = 4$, establishing a perfect alignment between the topological state space and the corrected algebraic decomposition.

\subsubsection{Isomorphism corrections for double connected sums}

To systematically capture the topological fluctuations induced by multiple junctions, we formalize the conditional tensor factor associated with loop-generating operations. Let $F \in \fun^{(n)}_B$ and $L \in \fun^{(m)}_\emptyset$, and let $\mathbf{R} = \{R_1, R_2\}$ be a set of two distinct junction pairs, where $R_1=(\alpha_1,\beta_1)$ and $R_2=(\alpha_2,\beta_2)$ are pairs of strands with $\alpha_1,\alpha_2\subset F$ and $\beta_1,\beta_2\subset L$. 

For each cubical vertex $(v, w) \in \cube^n \times \cube^m$, the indicator exponent $I_{\mathbf{R}}(v,w) \in \{0, 1\}$ is defined to be $1$ if $\alpha_1$ and $\alpha_2$ lie on the same connected component of $F(v,w)$ and $\beta_1$ and $\beta_2$ lie on the same connected component of $L(v,w)$. Otherwise, $I_{\mathbf{R}}(v,w)=0$. For example, the indicator exponent of the connected sum from Example \ref{example:counterexample_loop} equals $1$ at its unique state.

\begin{definition}
Let $F \in \fun^{(n)}_B$ and $L \in \fun^{(m)}_\emptyset$, and let $\mathbf{R} = \{R_1, R_2\}$ be a set of two distinct junction pairs. For $(v, w) \in \cube^n \times \cube^m$, we define
\begin{equation*}
    \left(\mathcal{Z}(F(v)) \otimes_{V\otimes V} \mathcal{Z}(L(w))\right) \otimes^{\mathbf{R}} V = \left(\mathcal{Z}(F(v)) \otimes_{V\otimes V} \mathcal{Z}(L(w))\right) \otimes  V^{\otimes I_{\mathbf{R}}(v,w)}.
\end{equation*}
\end{definition}

This local definition extends across the hypercube to define the graded chain complex with
\begin{equation*}
    \left(\mathcal{C}(F) \otimes_{V \otimes V} \mathcal{C}(L)\right) \otimes^{\mathbf{R}} V = \bigoplus\limits_{(v,w)\in \cube^n \times \cube^m} \left(\mathcal{Z}(F(v)) \otimes_{V\otimes V} \mathcal{Z}(L(w))\right) \otimes^{\mathbf{R}} V.
\end{equation*}
The differential $d_{\otimes^{\mathbf{R}}}$ on a component corresponding to an edge $(v,w) \to (v',w')$ in $\cube^n \times \cube^m$ is defined componentwise as follows:
\begin{enumerate}[label=$(\roman*)$]
    \item $I_{\mathbf{R}}(v,w) = I_{\mathbf{R}}(v',w')$: If the indicator remains invariant, the differential acts as the standard tensor differential carrying the external tensor factor via the identity
    \begin{equation*}
        d_{\otimes^{\mathbf{R}}} = (d_F \otimes_{V\otimes V} \mathrm{id}_L \pm \mathrm{id}_F \otimes_{V\otimes V} d_L) \otimes \mathrm{id}_{V^{\otimes I_{\mathbf{R}}}}.
    \end{equation*}
    \item $I_{\mathbf{R}}(v,w) \neq I_{\mathbf{R}}(v',w')$: If the indicator changes, the differential incorporates the local algebraic structures governing the splitting and merging of the loop at the junction-carrying factors. Explicitly, by isolating the state space components as $\mathcal{Z}(F(v)) \otimes_{V\otimes V} \mathcal{Z}(L(w)) \cong \mathcal{Z}_{\mathrm{spec}} \otimes  M_{\mathbf{R}}$, the transition differential is given by:
    \begin{equation*}
        d_{\otimes^{\mathbf{R}}} = \begin{cases}
            \mathrm{id}_{\mathcal{Z}_{\mathrm{spec}}} \otimes \Delta_{\mathbf{R}} & \text{if } I_{\mathbf{R}}(v,w) = 0 \to I_{\mathbf{R}}(v',w') = 1, \\
            \mathrm{id}_{\mathcal{Z}_{\mathrm{spec}}} \otimes m_{\mathbf{R}} & \text{if } I_{\mathbf{R}}(v,w) = 1 \to I_{\mathbf{R}}(v',w') = 0,
        \end{cases}
    \end{equation*}
    where $\mathcal{Z}_{\mathrm{spec}} = \mathcal{Z}(F(v)) \otimes_{V\otimes V} \mathcal{Z}(L(w) \setminus \mathbf{R})$ comprises all spectator components, $M_{\mathbf{R}}$ is the internal structural module (isomorphic to either $V$ or $W$ depending on the local boundary conformation) linked to the junctions, while $\Delta_{\mathbf{R}}\colon M_{\mathbf{R}} \to M_{\mathbf{R}} \otimes  V$ and $m_{\mathbf{R}}\colon M_{\mathbf{R}} \otimes  V \to M_{\mathbf{R}}$ represent the induced local TQFT comultiplication and multiplication actions, respectively.
\end{enumerate}

\begin{remark}\label{remark:tensor_configurations}
We analyze the explicit algebraic form of the tensor product over $V \otimes V$ within the complex $\mathcal{C}(F) \otimes_{V \otimes V} \mathcal{C}(L)$. For a given cubical state $(v,w)$, the local diagrammatic patterns at the junctions yield four possible configurations, as illustrated in Figure \ref{figure:double_sum}.
\begin{figure}[h]
  \centering
  \includegraphics[width=0.5\textwidth]{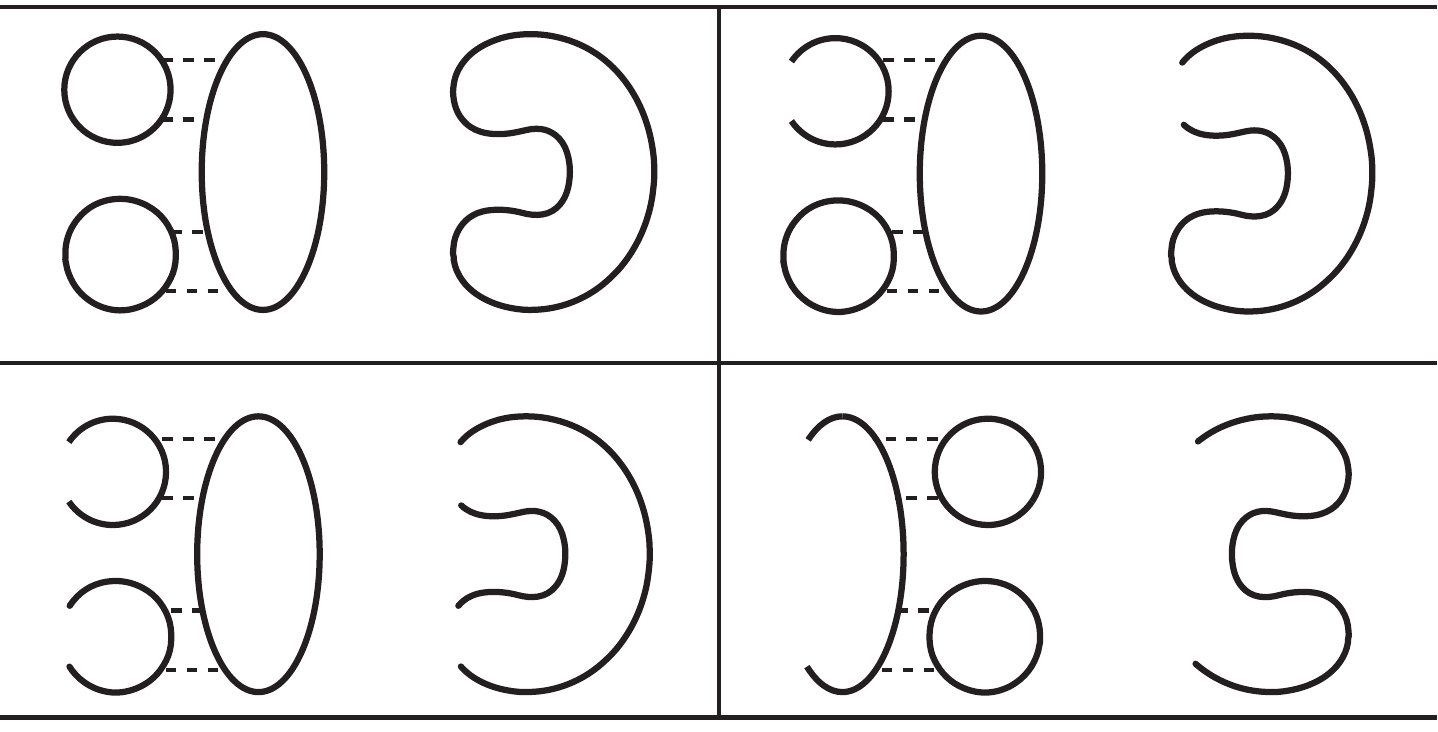} 
  \caption{Illustration of the different local configurations of double connected sum.}\label{figure:double_sum}
\end{figure}

\smallskip
For the first configuration (top-left of Figure \ref{figure:double_sum}), both interacting components are closed circles carrying the full Frobenius algebra, where the product evaluates to
\[
\begin{tikzpicture}[baseline=(X.base)]
    \node (X) [inner sep=2pt] {$(V\otimes V)\otimes_{V\otimes V}V \cong V$.};
    \draw [<->, >=stealth] (X.17) to [out=150, in=30] (X.170);
    \draw [<->, >=stealth] (X.343) to [out=210, in=330] (X.198);
\end{tikzpicture}
\]
Here, the first and the third tensor factors $V$ exchange via the $(V\otimes V)$-module action, and the second and the third factors $V$ exchange as well. The non-trivial generators of $(V\otimes V)\otimes_{V\otimes V}V \cong V$ are given by
\begin{equation*}
  (x\otimes 1)\otimes_{V\otimes V} 1 \quad \text{and} \quad (1\otimes 1)\otimes_{V\otimes V} 1.
\end{equation*}

\smallskip
For the second configuration (top-right of Figure \ref{figure:double_sum}), one crossing resolution reduces to an open arc, and the product yields
\[
\begin{tikzpicture}[baseline=(X.base)]
    \node (X) [inner sep=2pt] {$(W\otimes V)\otimes_{V\otimes V}V \cong W$.};
    \draw [<->, >=stealth] (X.17) to [out=150, in=30] (X.170);
    \draw [<->, >=stealth] (X.343) to [out=210, in=330] (X.199);
\end{tikzpicture}
\]
Here, the third factor $V$ is across the tensor to scale the first factor $W$ and the second factor $V$. The resulting space $(W\otimes V)\otimes_{V\otimes V}V \cong W$ is $1$-dimensional, spanned by the unique non-zero generator
\begin{equation*}
  (x\otimes 1)\otimes_{V\otimes V} 1.
\end{equation*}

\smallskip
For the third configuration (bottom-left of Figure \ref{figure:double_sum}), both resolutions yield open arcs on the same side, corresponding to the tensor form
\[
\begin{tikzpicture}[baseline=(X.base)]
    \node (X) [inner sep=2pt] {$(W\otimes W)\otimes_{V\otimes V}V \cong W\otimes W$};
    \draw [<->, >=stealth] (X.26) to [out=150, in=30] (X.172);
    \draw [<->, >=stealth] (X.334) to [out=210, in=330] (X.192);
\end{tikzpicture}
\]
The action of the external factor $V$ is absorbed by the internal components, and the space $(W\otimes W)\otimes_{V\otimes V}V$ possesses a unique non-trivial generator
\begin{equation*}
  (x\otimes x)\otimes_{V\otimes V} 1.
\end{equation*}

\smallskip
For the fourth configuration (bottom-right of Figure \ref{figure:double_sum}), symmetrically, the open arc conformations appear on the alternative side, yielding
\[
\begin{tikzpicture}[baseline=(X.base)]
    \node (X) [inner sep=2pt] {$W\otimes_{V\otimes V}(V\otimes V) \cong W$,};
    \draw [<->, >=stealth] (X.19) to [out=150, in=30] (X.171);
    \draw [<->, >=stealth] (X.280) to [out=210, in=330] (X.189);
\end{tikzpicture}
\]
where the second and third factors $V$ are shuffled across the tensor product to act on the ideal $W$. The non-zero generator of $W\otimes_{V\otimes V}(V\otimes V)$ is uniquely given by
\begin{equation*}
  x\otimes_{V\otimes V} (1\otimes 1).
\end{equation*}

All these algebraic evaluations are in perfect alignment with the geometric configurations of the local connected sums.
\end{remark}

With the above construction established, the single-junction isomorphism generalizes smoothly to dual-junction configurations.

\begin{theorem}\label{thm:dual_junction_isomorphism}
Let $F \in \fun^{(n)}_B$ be a pro-tangle and $L \in \fun^{(m)}_\emptyset$ be a pro-link. Suppose $\mathbf{R}$ is a collection of two distinct junction pairs between $F$ and $L$ for the connected sum. There exists an isomorphism of chain complexes
\begin{equation*}
    \mathcal{C}(F \#_{\mathbf{R}} L) \cong \left(\mathcal{C}(F) \otimes_{V \otimes V} \mathcal{C}(L)\right) \otimes^{\mathbf{R}} V.
\end{equation*}
\end{theorem}

\begin{proof}
We construct the chain isomorphism by defining a family of localized linear embeddings at each hypercube vertex $(v, w) \in \cube^n \times \cube^m$, obtaining the map
\begin{equation*}
    \kappa_{v,w}: \mathcal{Z}\left((F \#_{\mathbf{R}} L)(v, w)\right) \longrightarrow \left(\mathcal{Z}(F(v)) \otimes_{V\otimes V} \mathcal{Z}(L(w))\right) \otimes^{\mathbf{R}} V.
\end{equation*}
The definition of $\kappa_{v,w}$ proceeds by a case analysis on the local patterns. If $I_{\mathbf{R}}(v,w) = 0$, the junctions connect to distinct components within $L(w)$; the operator $\otimes^{\mathbf{R}} V$ reduces to the identity factor $\otimes  \mathbb{F}$, and $\kappa_{v,w}$ is given by the standard single-junction bijection. If $I_{\mathbf{R}}(v,w) = 1$, the parallel surgery isolates an extra unlinked circle in the state space of $F \#_{\mathbf{R}} L$; here, $\kappa_{v,w}$ acts as the identity on the core spectator components mapping into $\mathcal{Z}_{\mathrm{spec}}$, while sending the newly generated circle to the external factor $V$ through the canonical vector space identification $M_{\mathbf{R}} \otimes  V$.

To establish that $\kappa$ is a chain map, we check its commutativity with the total differential $d_{\otimes^{\mathbf{R}}}$ across both internal and transition edges of the product hypercube $\cube^n \times \cube^m$:

\smallskip
\textit{Case 1: Internal edges ($I_{\mathbf{R}}(v,w) = I_{\mathbf{R}}(v',w')$).} 
On subcubes where the indicator exponent remains invariant, no parallel loops are created or destroyed by the edge maps. According to the definition, the total differential acts as the standard tensor product differential extended via the identity on the external factor:
\begin{equation*}
    d_{\otimes^{\mathbf{R}}} = (d_F \otimes_{V\otimes V} \mathrm{id}_L \pm \mathrm{id}_F \otimes_{V\otimes V} d_L) \otimes \mathrm{id}_{V^{\otimes I_{\mathbf{R}}}}.
\end{equation*}
Commutativity $d_{\otimes^{\mathbf{R}}} \circ \kappa = \kappa \circ d_{F \#_{\mathbf{R}} L}$ in this regime follows immediately from the functoriality of the standard single-junction tensor product established in Theorem \ref{theorem:connected_sum_isomorphism}, where the cubical signs align precisely with the Koszul sign rule.

\smallskip
\textit{Case 2: Transition edges ($I_{\mathbf{R}}(v,w) \neq I_{\mathbf{R}}(v',w')$).} 
On the edges triggering a transition between the two regimes, the map does not engage the internal differentials $d_F$ or $d_L$, but acts locally on the junction-carrying component. 
\begin{itemize}[label=$(\roman*)$]
    \item For a transition of type $0 \to 1$, a parallel loop splits off from the junction-carrying module. The differential evaluates to $d_{\otimes^{\mathbf{R}}} = \mathrm{id}_{\mathcal{Z}_{\mathrm{spec}}} \otimes \Delta_{\mathbf{R}}$. Since $\kappa_{v',w'}$ maps the newly separated loop into the external $V$ factor, commutativity reduces to the requirement that the geometric splitting cobordism aligns with the local TQFT comultiplication $\Delta_{\mathbf{R}} \colon M_{\mathbf{R}} \to M_{\mathbf{R}} \otimes  V$.
    \item For a transition of type $1 \to 0$, the external loop merges back into the junction-carrying module. The differential evaluates to $d_{\otimes^{\mathbf{R}}} = \mathrm{id}_{\mathcal{Z}_{\mathrm{spec}}} \otimes m_{\mathbf{R}}$. The geometric surgery corresponds precisely to the local TQFT multiplication $m_{\mathbf{R}} \colon M_{\mathbf{R}} \otimes  V \to M_{\mathbf{R}}$.
\end{itemize}
Because $V$ is an associative and coassociative Frobenius algebra and $M_{\mathbf{R}}$ inherits a well-defined local Frobenius-type module structure, these algebraic relations commute exactly with $\kappa_{v,w}$ on the nose.

\smallskip
As $\kappa$ preserves the gradings at every vertex and commutes with the differential on both internal and transition edges, it constitutes a strict isomorphism of differential graded complexes.
\end{proof}

\section*{Acknowledgments}

This work was supported in part by the National Natural Science Foundation of China (NSFC Grant No. 12401080), the Scientific Research Foundation of Chongqing University of Technology, the National Key Research and Development Program of China (NKPs Grant No. 2024YFA1013201), the Natural Science Foundation of Chongqing (NSFCQ Grant No. CSTB2024NSCQ-LZX0040), and the Special Project of Chongqing Municipal Science and Technology Bureau (Grant No. 2025CCZ015). Li Shen was supported by an NITMB fellowship supported by grants from the NSF (Grant No. DMS-2235451) and the Simons Foundation (Grant No. MP-TMPS-00005320).

\bibliographystyle{plain}  
\bibliography{Reference}

\end{document}